\documentclass[11pt,a4paper,final]{amsart}
\usepackage[top=32mm,bottom=32mm,left=27mm,right=27mm]{geometry}
\usepackage[hidelinks]{hyperref}
\usepackage{enumerate,xcolor,booktabs,color,amsmath,amsthm,amssymb,tikz-cd}
\usepackage{lscape}
\definecolor{LightCyan}{rgb}{0.88,1,1}
\newcommand{\s}{\mathsf}
\newcommand{\g}{\mathfrak}
\newcommand{\R}{\mathbb{R}}

\newcommand{\Ad}{\operatorname{Ad}}

\newcommand{\C}{\mathbb{C}}
\newcommand{\vecspan}{\operatorname{span}}
\newcommand{\Exp}{\operatorname{Exp}}
\newcommand{\spann}{\operatorname{span}}
\renewcommand{\H}{\mathbb{H}}

\renewcommand{\hat}{\widehat}
\renewcommand{\tilde}{\widetilde}

\newtheorem{theorem}{Theorem}[section]
\newtheorem{lemma}[theorem]{Lemma}
\newtheorem{corollary}[theorem]{Corollary}
\newtheorem{proposition}[theorem]{Proposition}
\newtheorem{thmx}{Theorem}

\newtheorem{corx}[thmx]{Corollary}
\theoremstyle{definition}

\newtheorem{remark}[theorem]{Remark}
\newtheorem{example}[theorem]{Example}
\renewcommand{\arraystretch}{1.2}

\title[Totally geodesic submanifolds of the homogeneous nearly Kähler 6-manifolds]{Totally geodesic submanifolds of the\\ homogeneous nearly Kähler 6-manifolds\\ and their $\mathsf{G}_2$-cones}

\author[J.~M. Lorenzo-Naveiro]{Juan Manuel Lorenzo-Naveiro}
\address{University of Oklahoma, United States of America}
\email{jnaveiro@ou.edu}

\author[A. Rodr\'{\i}guez-V\'{a}zquez]{Alberto Rodr\'{\i}guez-V\'{a}zquez}
\address{Universit\'{e} Libre de Bruxelles, Belgium}
\email{alberto.rodriguez.vazquez@ulb.be}
\thanks{Both authors have been supported by PID2022-138988NB-I00 funded by MICIU/AEI/10.13039/501100011033 and by ERDF (European Union).
The first author is supported by the FPI program (Spain) and the research project ED431C 2023/31 (Xunta de Galicia, Spain). The second author has been supported by the FWO Postdoctoral grant with project number 1262324N, and by the Horizon Europe research and innovation programme under Marie Sklodowska Curie Actions with grant agreement 101149711 - HOLYFLOW. \\ {Ethical Statement:} On behalf of all authors, the corresponding author (A.R.V.) states that there is no conflict of interest.\\ {Data Availability Statement:} No data sets have been generated or used in the writing of this article}

\subjclass[2010]{Primary 53C30, Secondary 53C35, 53C40}
\keywords{Totally geodesic, homogeneous space, Riemannian cone, nearly Kähler, holonomy~$\mathsf{G}_2$}
\begin{document}
	\begin{abstract}
	In this article we classify totally geodesic submanifolds of homogeneous nearly Kähler $6$-manifolds, and of the $\mathsf{G}_2$-cones over these $6$-manifolds. To this end, we develop new techniques for the study of totally geodesic submanifolds of analytic Riemannian manifolds, naturally reductive homogeneous spaces, and  Riemannian cones.
	In particular, we obtain an example of a totally geodesic submanifold with self-intersections in a simply connected homogeneous space.
	\end{abstract}
\maketitle

\section{Introduction}

An almost Hermitian structure on a Riemannian manifold $(M,g)$ is a $(1,1)$ tensor field $J$ preserving the Riemannian metric $g$ and satisfying the identity $J^2=-\operatorname{Id}$. Gray and Hervella~\cite{grayhervella} showed that there are 16 natural classes of almost Hermitian structures.
A nice example of those are Kähler structures, which are characterized by the equation $\nabla J=0$, where $\nabla$ denotes the Levi-Civita connection of $(M,g)$. Among the non-integrable Hermitian structures, the nearly Kähler structures are particularly noteworthy.
An almost Hermitian structure $J$ on a Riemannian manifold $(M,g)$ is \textit{nearly K\"ahler} if it satisfies
\[ (\nabla_X J)(Y) =-(\nabla_Y J)(X) \quad \textrm{for all vector fields $X$ and $Y$ of $M$}. \]
The study of nearly Kähler geometry is particularly interesting in dimension $6$, as strictly nearly Kähler $6$-manifolds (that is, those which are not Kähler) are automatically Einstein, and their Riemannian cones are 7-manifolds with holonomy groups contained in $\s{G}_2$, see~\cite{bar}.   Indeed, in~\cite{bryant}, Bryant constructed the first examples of manifolds with holonomy exactly equal to $\s{G}_2$, one of which is a Riemannian cone over the flag manifold $\s{F}(\C^3)$
equipped with its homogeneous nearly Kähler metric.

Although investigations about nearly Kähler manifolds began in the 1950s, significant progress has been made in recent decades.
In 2005, Butruille~\cite{Butruille} classified all simply connected, homogeneous, strictly nearly Kähler manifolds of dimension six. These are:
\[
\begin{aligned}
\mathsf{S}^6&=\s{G}_2/\s{SU}(3),  &\qquad \C\s{P}^3&=\s{Sp}(2)/(\s{U}(1)\times\s{Sp}(1)), \\ \mathsf{F}(\C^3)&=\s{SU}(3)/\s{T}^2, &\qquad \s{S}^3\times\s{S}^3&=\s{SU}(2)^3/\Delta \s{SU}(2).
\end{aligned}
\]

More recently, in 2017, Foscolo and Haskins~\cite{foscolo} produced the first inhomogeneous nearly Kähler structures on $\s{S}^6$ and $\s{S}^3\times\s{S}^3$.
These structures are of cohomogeneity one and are $\s{SU}(2)\times\s{SU}(2)$-invariant.
The aforementioned examples constitute all the known simply connected  strictly nearly Kähler $6$-manifolds so far.

The main goal of this article is to classify totally geodesic submanifolds in homogeneous strictly nearly Kähler $6$-manifolds. For this purpose, we develop some general tools for the study of totally geodesic submanifolds of naturally reductive homogeneous spaces.
Moreover, we also classify maximal totally geodesic submanifolds of the $\s{G}_2$-cones over homogeneous strictly nearly Kähler $6$-manifolds.

The theory of totally geodesic submanifolds in homogeneous spaces is especially fruitful due to the large isometry groups of these spaces.  Symmetric spaces are a particular type of homogeneous spaces with an isometric geodesic reflection at each point.
Totally geodesic submanifolds of symmetric spaces have been extensively studied, see~e.g.~\cite{BO1,chen1,chen2,exceptional}.
In this setting, it turns out that totally geodesic submanifolds arise as orbits of isometric actions and can be characterized algebraically by means of the so-called Lie triple systems.
That being said, the classification of totally geodesic submanifolds in classical symmetric spaces of rank greater than two remains widely open.
In contrast, in a non-symmetric Riemannian homogeneous space we lack a nice algebraic characterization of totally geodesic submanifolds.
Furthermore, although these submanifolds are intrinsically homogeneous, they are not necessarily orbits of Lie groups acting isometrically, see~\cite{nikolayevskytg,RodriguezVazquezOlmos}. 

Another relevant aspect is that in the (non-symmetric) homogeneous setting, it is natural to consider totally geodesic immersions that are not necessarily injective.
The theory of totally geodesic immersions, developed in \cite{Tsukada}  and \cite{HeintzeLiuOlmos}, is central to this work, and in this article we aim to further build upon it, see~Section~\ref{sec:tganalytic}.
It is well known that in homogeneous spaces, every one-dimensional totally geodesic submanifold is injectively immersed (in other words, every geodesic loop is a closed geodesic), and in symmetric spaces, all totally geodesic submanifolds are injectively immersed.
In this work, we present what is, up to the best of our knowledge, the first example of a totally geodesic submanifold  in a simply connected homogeneous space with self-intersections, see~Subsection~\ref{subsec:projplanerealflag}.

In this article, we introduce the class of \textit{$D$-invariant totally geodesic submanifolds} of a reductive homogeneous space, where $D$ denotes the difference between the Levi-Civita connection $\nabla$ and the canonical connection $\nabla^c$, see Subsection~\ref{subsec:homspaces}.
In the setting of naturally reductive homogeneous spaces, which includes symmetric spaces,   $D$-invariant totally geodesic submanifolds are orbits of Lie groups acting isometrically, and they admit a nice algebraic description similar to that of Lie triple systems in symmetric spaces.
Furthermore, totally geodesic submanifolds of symmetric spaces are trivially $D$-invariant.
Thus, the class of $D$-invariant totally geodesic submanifolds seems to be a natural generalization of totally geodesic submanifolds of symmetric spaces.
As we will see, not all totally geodesic submanifolds in a naturally reductive homogeneous space are $D$-invariant.
It will follow from our classification that all maximal totally geodesic submanifolds---that is, those that are not contained in any other proper totally geodesic submanifold---in the homogeneous nearly Kähler $6$-manifolds are $D$-invariant. 

As we have seen, there are exactly four examples of  simply connected homogeneous strictly nearly Kähler $6$-manifolds.
The sphere $\s{S}^6=\s{G}_2/\s{SU}(3)$, whose nearly Kähler structure is induced by octonionic multiplication, was the first one to appear.
Since $\s{S}^6=\s{G}_2/\s{SU}(3)$ is isotropy irreducible, it carries a unique $\s{G}_2$-invariant metric up to scaling, which is precisely the round one.
Thus, its totally geodesic submanifolds are open parts of intersections of vector subspaces of $\R^7$ passing through the center of the unit sphere $\s{S}^6\subseteq\R^7$.
Each of the remaining examples appears as the total space of a homogeneous fibration
$F=\s{K}/\s{H}\rightarrow M=\s{G}/\s{H}\rightarrow B=\s{G}/\s{K}$
induced by a triple of compact Lie groups $\s{H}\subseteq \s{K}\subseteq \s{G}$.
In the cases of $\C\s{P}^3$ and $\s{F}(\C^3)$, these fibrations are also examples of twistor fibrations.
A \textit{twistor fibration} of an oriented Riemannian  $4$-manifold $N$ is a fiber bundle $\pi\colon M\rightarrow N$, where each fiber over $p\in N$ is equal to the set of complex structures of $T_p N$ which preserve the orientation and the Riemannian metric of $N$. It turns out that $\pi$ is an $\s{S}^2$-bundle over $N$, and the twistor space $M$ admits two different natural almost Hermitian structures. One of them is called the Atiyah--Hitchin--Singer structure and it is integrable if and only if the $4$-manifold $N$ is self-dual, see~\cite{atiyah}. The other one is the Eells--Salamon structure, which can be obtained from the  Atiyah--Hitchin--Singer structure by changing the sign only in the fibers.  The nearly Kähler structures of $\C\s{P}^3$ and $\s{F}(\C^3)$ that we are considering are precisely the Eells--Salamon structures that we get when $N$ is the round sphere $\s{S}^4$ and the complex projective plane $\C\s{P}^2$, respectively. Moreover,  the celebrated Eells--Salamon correspondence, see~\cite{eels}, states a one-to-one correspondence between (branched) minimal surfaces in $N$, and non-vertical \textit{$J$-holomorphic curves} in $M$, i.e.\ $J$-invariant immersions of $2$-manifolds in~$M$.

When studying totally geodesic submanifolds of the total space $M$ of a Riemannian submersion  $F\rightarrow M\rightarrow B$, it is also relevant to consider their behavior with respect to the underlying Riemannian submersion. Following~\cite{RodriguezVazquezOlmos}, we say that a totally geodesic submanifold $\Sigma$ of the total space of a Riemannian submersion $M$ is \textit{well-positioned} if 
\[T_p \Sigma= (\mathcal{V}_p \cap T_p \Sigma) \oplus (\mathcal{H}_p \cap T_p \Sigma) \quad \text{for all $p\in\Sigma$},\]
where $\mathcal{V}$ and $\mathcal{H}$ denote the vertical and horizontal distributions associated with the Riemannian submersion $F\rightarrow M\rightarrow B$. It turns out that if a totally geodesic submanifold is well-positioned, the metric of the total space can be rescaled in the direction of the fibers while preserving the totally geodesic property for all these new metrics, see~\cite[Lemma 3.12]{dearricott}.  In this work, we find several examples of not well-positioned totally geodesic submanifolds.

We are also interested in understanding the interaction between totally geodesic
submanifolds of the nearly Kähler spaces under investigation and their ambient almost
complex structure.
A way to measure how a submanifold fails to be complex is by using the notion of Kähler angle, see e.g.~\cite{bolton}. We say that a submanifold $\Sigma$ of an almost Hermitian manifold $M$ has constant Kähler angle $\Phi(\Sigma)=\varphi\in[0,\pi/2]$ if
\begin{equation*}
	\lvert\lvert \mathrm{proj}_{T_p \Sigma} J v\rvert\rvert^2=\cos^2(\varphi)\lvert \lvert v\rvert \rvert^2 \quad \text{for all $v\in T_p\Sigma$ and every $p\in M$},
\end{equation*}
where $\mathrm{proj}_{T_p \Sigma}$ denotes the orthogonal projection onto $T_p \Sigma$.
The submanifolds satisfying $\Phi(\Sigma)=0$ or $\Phi(\Sigma)=\pi/2$ are exactly those which are almost complex or totally real, respectively.
An interesting question is to determine the possible constant Kähler angles of the totally geodesic submanifolds of an almost Hermitian manifold.
Of course, this  question is only interesting for spaces with non-constant curvature, since in $\C^n$ every number in $[0,\pi/2]$ can be realized as the constant Kähler angle of a totally geodesic submanifold.
In the setting of Hermitian symmetric spaces, not all totally geodesic submanifolds $\Sigma$ satisfy $\Phi(\Sigma)\in\{0,\pi/2\}$.
For instance, Klein realized in~\cite{kleindga} that there is a totally geodesic $2$-sphere in the Hermitian symmetric space $\s{G}^+_2(\R^5)=\s{SO}(5)/(\s{SO}(3)\times\s{SO}(2))$ with constant Kähler angle $\arccos(1/5)$.
Even more, the second author of this article proved in~\cite{products} that every rational number between $[0,1]$ can be realized as the arccosine of the Kähler angle of a totally geodesic submanifold embedded in a Hermitian symmetric space of large enough rank.

There is a relatively large number of articles focusing on the  investigation of totally geodesic submanifolds of nearly Kähler homogeneous $6$-manifolds under strong assumptions.
In these works the authors use special frames to carry out the classification for Lagrangian totally geodesic submanifolds, or  totally geodesic $J$-holomorphic curves.
In  $\C\s{P}^3$, the Lagrangian totally geodesic submanifolds  were classified in~\cite{Aslan,liefsoens}.
In the flag manifold $\s{F}(\C^3)$ the Lagrangian totally geodesic submanifolds were classified in~\cite{Storm}, and the totally geodesic $J$-holomorphic  curves were classified in~\cite{CwiklinskiVrancken}.
In $\s{S}^3\times\s{S}^3$, the totally geodesic Lagrangian submanifolds were classified in~\cite{ZhangDioosHuVranckenWang}, and the totally geodesic $J$-holomorphic curves were classified in~\cite{BoltonDillenDioosVrancken}.
Although they listed six congruence classes of totally geodesic submanifolds in $\s{S}^3\times\s{S}^3$, there are just two different ones: either a round sphere or a  Berger sphere, where the latter was first constructed in~\cite{moroianu}.
It is important to remark that there are no known  results obstructing the existence of totally geodesic submanifolds $\Sigma$ of nearly Kähler $6$-manifolds when: $\Sigma^3$ is not Lagrangian, $\Sigma^2$ is not $J$-holomorphic, or $\Sigma$ has dimension~$4$.
In this article, we generalize the aforementioned partial classifications  following an entirely different approach.
By employing tools from the theory of Riemannian homogeneous spaces, we address the classification problem of totally geodesic submanifolds in its full generality.

In what follows we state the main results of this article. We denote by $\s{S}^n(r)$ the $n$-dimensional sphere of radius $r$, and by $\R\s{P}^n(r)$  its $\mathbb{Z}_2$-quotient under the antipodal map. Moreover, let us consider the sphere $\s{S}^{3}$ with the Berger metric $g_\tau$ given by taking the round metric (of radius one) and rescaling the vertical subspace of the Hopf fibration $\s{S}^1\rightarrow \s{S}^3\rightarrow \s{S}^2$ by a factor of $\tau>0$.  Then $\s{S}^{3}_{\C,\tau}(r)$ denotes the sphere $\s{S}^{3}$ equipped with the Riemannian metric $r^{2} g_\tau$, and we denote by $\R\s{P}^{3}_{\C,\tau}(r)$ its $\mathbb{Z}_2$-quotient.  We also denote by $\s{T}^2_{\Lambda}$ the torus induced by a lattice $\Lambda\subseteq \R^2$.

\begin{thmx}\label{th:tg-cp3-classification}
	Let $\Sigma$ be a complete submanifold of the homogeneous nearly Kähler manifold $\C\s{P}^3=\s{Sp}(2)/\s{U}(1)\times\s{Sp}(1)$ of dimension~$d\ge 2$. Then, $\Sigma$ is totally geodesic if and only if it is congruent to one of the submanifolds listed in Table~\ref{table:CP3}.
\end{thmx}

\begin{table}[h!]
	\centering
	\caption{Totally geodesic submanifolds of $\C\mathsf{P}^{3}$ of dimension $d\ge 2$.}
	\label{table:CP3}
	\renewcommand{\arraystretch}{1.2}
	\begin{tabular}{cccc}
		\toprule 
		Submanifold & Relationship with $J$ & Comments & Well-positioned? \\
		\toprule
		$\R\mathsf{P}^{3}_{\C,1/2}(2)$ & Lagrangian & Orbit of $\mathsf{U}(2)$ & Yes\\
		\hline
		$\mathsf{S}^{2}\left(1/\sqrt{2}\right)$ & $J$-holomorphic & Fiber of $\C\mathsf{P}^{3}\to\mathsf{S}^{4}$ & Yes \\
		\hline
		$\mathsf{S}^{2}(1)$ & $J$-holomorphic & Orbit of $\mathsf{SU}(2)$ & Yes \\
		\hline
		$\mathsf{S}^{2}\left(\sqrt{5}\right)$ & $J$-holomorphic & Orbit of $\mathsf{SU}(2)_{\Lambda_{3}}$ & No \\
		\bottomrule
	\end{tabular}
\end{table}
As far as we know, the totally geodesic $\s{S}^2(\sqrt{5})$ has not appeared previously in the literature.
This is an orbit of the group $\s{SU}(2)_{\Lambda_{3}}$, which is the maximal connected subgroup of $\s{Sp}(2)$ induced by the $4$-dimensional complex irreducible representation of $\s{SU}(2)$ (note that this representation is of symplectic type).
All the examples in this theorem are maximal.
The non-vertical totally geodesic $J$-holomorphic curves are $\s{S}^2(1)$ and $\s{S}^2(\sqrt{5})$. Their associated minimal surfaces in $\s{S}^4$  under the Eells--Salamon correspondence are a totally geodesic $2$-sphere in $\s{S}^4$, and the Veronese embedding of the projective plane in $\s{S}^4$, respectively.

\begin{thmx}\label{th:tg-fc3-classification}
	Let $\Sigma$ be a complete submanifold of the homogeneous nearly Kähler manifold $\s{F}(\C^3)=\s{SU}(3)/\s{T}^2$ of dimension~$d\ge 2$. Then, $\Sigma$ is totally geodesic if and only if it is congruent to one of the submanifolds listed in Table~\ref{table:FC3}.
\end{thmx}

\begin{table}[h!]
	\centering
	\caption{Totally geodesic submanifolds of $\mathsf{F}(\C^{3})$ of dimension $d\ge 2$.}
	\label{table:FC3}
	\renewcommand{\arraystretch}{1.2}
	\begin{tabular}{cccc}
		\toprule 
		Submanifold & Relationship with $J$ & Comments & Well-positioned? \\
		\toprule
		$\mathsf{F}(\R^{3})$ & Lagrangian & Orbit of $\mathsf{SO}(3)$ & Yes\\
		\hline
		$\mathsf{S}^{3}_{\C,1/4}(\sqrt{2})$ & Lagrangian & Orbit of $\mathsf{SU}(2)$ & No \\
		\hline
		$\mathsf{T}^2_{\Lambda}$ & $J$-holomorphic & Orbit of $\s{T}^2$ & No \\
		\hline
		$\mathsf{S}^{2}\left(1/\sqrt{2}\right)$ & $J$-holomorphic & Fiber of $\mathsf{F}(\C^{3})\to\C\mathsf{P}^{2}$ & Yes \\
		\hline
		$\mathsf{S}^{2}\left(\sqrt{2}\right)$ & $J$-holomorphic & Orbit of $\mathsf{SO}(3)$ & No \\
		\hline
		$\R\mathsf{P}^{2}\left(2 \sqrt{2}\right)$ & Totally real & Not injectively immersed & No \\
		\bottomrule
	\end{tabular}
\end{table}
To the best of our knowledge, $\R\s{P}^2(2\sqrt{2})$ is the first example in the literature of a totally geodesic immersed submanifold of dimension $d\geq 2$ with self-intersections in a simply connected homogeneous space.
Moreover, this is the only non-maximal example, it is not $D$-invariant, and not extrinsically homogeneous, i.e.\ an orbit of a subgroup of the isometry group of the ambient space.  The non-vertical totally geodesic $J$-holomorphic curves are  $\s{T}^2_\Lambda$ and $\s{S}^2(\sqrt{2})$. Their associated minimal surfaces in $\C\s{P}^2$ under the Eells--Salamon correspondence are the Clifford torus in~$\C\s{P}^2$, that is, $\{[z_0:z_1:z_2]\in \mathbb{C}\s{P}^2: \lvert z_{0}\rvert=\lvert z_{1}\rvert=\lvert z_{2}\rvert\}$, and a totally geodesic $\R\s{P}^2$ in $\C\s{P}^2$, respectively.

\begin{thmx}\label{th:tg-s3s3-classification}
	Let $\Sigma$ be a complete submanifold of the homogeneous nearly Kähler manifold $\s{S}^3 \times \s{S}^3=\s{SU}(2)^{3}/\Delta\s{SU}(2)$  of dimension~$d\ge 2$. Then, $\Sigma$ is totally geodesic if and only if it is congruent to one of the submanifolds listed in Table~\ref{table:S3S3}.
\end{thmx}

\begin{table}[h!]
	\centering
	\caption{Totally geodesic submanifolds of $\mathsf{S}^{3}\times\mathsf{S}^{3}$ of dimension $d\ge 2$.}
	\renewcommand{\arraystretch}{1.2}
	\begin{tabular}{cccc}
		\toprule 
		Submanifold & Relationship with $J$ & Comments & Well-positioned? \\
		\toprule
		$\mathsf{S}^{3}\left(2/\sqrt{3}\right)$ & Lagrangian & Fiber of $\mathsf{S}^{3}\times\mathsf{S}^{3}\to\mathsf{S}^{3}$ & Yes \\
		\hline
		$\mathsf{S}^{3}_{\C,1/3}(2)$ & Lagrangian & Orbit of $\Delta_{1,3}\mathsf{SU}(2)\times\mathsf{SU}(2)_{2}$ & Yes \\
		\hline
		$\mathsf{T}^2_{\Gamma}$ & $J$-holomorphic & Orbit of a two-dimensional torus & Yes \\
		\hline
		$\mathsf{S}^{2}(\sqrt{3/2})$ & $J$-holomorphic & Orbit of $\Delta\s{SU}(2)$ & No \\
		
		\hline
		$\mathsf{S}^{2}\left(2/\sqrt{3}\right)$ & Totally real & Orbit of $\Delta\s{SU}(2)$ & Yes \\
		\bottomrule
	\end{tabular}
	\label{table:S3S3}
\end{table}

Interestingly, $\s{S}^2(2/\sqrt{3})$ is not a $D$-invariant totally geodesic submanifold, but is extrinsically homogeneous.
This together with the characterization of $D$-invariant totally geodesic submanifolds given in Theorem~\ref{th:CharacterizationDInvariantTGSubmanifolds} gives a counterexample to Proposition~2 in~\cite{AlekseevskyNikonorov}, see~Remark~\ref{rem:alekseevski-nikonorov}. Furthermore, $\s{S}^2(2/\sqrt{3})$ is the only non-maximal example in the list above.

As a consequence of Theorem~\ref{th:tg-cp3-classification}, Theorem~\ref{th:tg-fc3-classification}, and Theorem~\ref{th:tg-s3s3-classification}, we have:

\begin{corx}\label{cor:tg-kahler-angle}
	Let $\Sigma$ be a maximal totally geodesic submanifold of a  homogeneous nearly Kähler $6$-manifold of non-constant curvature.
	Then the following statements hold:
	\begin{enumerate}[\rm (i)]
		\item if $\Sigma$ has dimension two, then $\Sigma$ is a $J$-holomorphic curve.
		\item if $\Sigma$ has dimension three, then $\Sigma$ is a Lagrangian submanifold.
	\end{enumerate}
\end{corx}
Thus, every totally geodesic submanifold $\Sigma$ of a homogeneous strictly nearly Kähler $6$-manifold of non-constant curvature satisfies $\Phi(\Sigma)\in\{0,\pi/2\}$.
This raises the question whether this also holds for  (not necessarily homogeneous) irreducible strictly nearly Kähler manifolds.

In this article, we also study totally geodesic submanifolds of Riemannian cones.
It can be checked that for every totally geodesic submanifold $\Sigma$ of a Riemannian manifold $M$, the cone over $\Sigma$ is a totally geodesic submanifold of the cone over $M$, see~Lemma~\ref{lemma:tgcone}.
However, there might be totally geodesic submanifolds of a Riemannian cone that do not arise as cones, see~Examples~\ref{ex:unitspherecone} and~\ref{ex:nontrivalexconetg}.
In Section~\ref{sec:tgcones}, we prove a structure result for totally geodesic submanifolds in cones, see Theorem~\ref{th:TGSubmanifoldsOfCones}.
As a  consequence of this, we deduce that maximal totally geodesic submanifolds of Riemannian cones are either cones over a totally geodesic submanifold or totally geodesic hypersurfaces, see Corollary~\ref{cor:MaximalTGSubmanifoldsOfCones}. 
It was observed in~\cite{bar} that Riemannian cones are intimately linked to special holonomy.
For instance, a special class of Sasakian manifolds is that of Sasakian--Einstein manifolds, whose investigation has led to the construction of many inhomogeneous Einstein metrics on spheres, see~\cite{boyer}.
It turns out that the holonomy of the cone over a Sasakian--Einstein manifold is contained in $\s{SU}(n)$.
Similarly, Riemannian cones over strictly nearly Kähler $6$-manifolds have its holonomy contained in $\s{G}_2$.
This holonomy reduction is equivalent to the existence of a parallel $3$-form $\phi$ that is locally modeled on the associative $3$-form on $\R^7$, or alternatively, the existence of a torsion-free $\s{G}_2$-structure.
Another class of $\s{G}_2$-structures defined on $7$-dimensional manifolds is that of nearly parallel $\s{G}_2$-structures.
A $\s{G}_2$-structure $\phi$ is \textit{nearly parallel} if it satisfies $\star d\phi=c  \phi$, for $c\in\R\setminus\{0\}$, where $\star$ denotes the Hodge star operator. It can also be proved that cones over nearly parallel $\s{G}_2$-manifolds have its holonomy contained in $\s{Spin}(7)$.
The first example of a manifold with exceptional holonomy $\s{Spin}(7)$ was constructed in~\cite{bryant}, and it is the cone over the homogeneous nearly parallel $\s{G}_2$-manifold $\s{B}^7=\s{SO}(5)/\s{SO}(3)$.
We prove that Sasakian--Einstein, strictly nearly Kähler $6$-manifolds, and nearly parallel $\s{G}_2$-manifolds do not admit totally geodesic hypersurfaces, see~Theorem~\ref{th:hyp}.

Moreover, we consider the classification problem of totally geodesic submanifolds in cones with holonomy $\s{G}_2$ over homogeneous nearly Kähler manifolds.
As a consequence of Corollary~\ref{cor:ConeIsometryGroup}, these are the only Riemannian cones with holonomy equal to $\s{G}_2$ equipped  with a metric of cohomogeneity one, and thus with the highest possible degree of symmetry, contributing to a rich presence of totally geodesic submanifolds. 
\begin{thmx}\label{th:tg-g2-cones}
	Let $M$ be a  homogeneous nearly Kähler $6$-manifold of non-constant curvature and let $\Sigma$ be a maximal totally geodesic submanifold of the $\s{G}_2$-cone $\widehat{M}$ over $M$ of dimension greater than one.	Then $\Sigma$ is the Riemannian cone of a maximal totally geodesic submanifold $S$ of $M$.
\end{thmx}
Notice that combining Theorem~\ref{th:tg-g2-cones} with the classification of totally geodesic submanifolds in cones over space forms~(see Proposition~\ref{prop:ConstantCurvatureCone}) and three-dimensional Berger spheres (see Proposition~\ref{prop:bergercone}), one can list all totally geodesic submanifolds in cohomogeneity one $\s{G}_2$-cones and thus obtain the full classification.

Moreover, Riemannian cones over $J$-holomorphic curves or Lagrangian submanifolds of a nearly Kähler $6$-manifold give rise to associative or coassociative manifolds of the corresponding $\s{G}_2$-cone over~$M$, respectively.
By definition, \textit{associative} and \textit{coassociative} submanifolds are the submanifolds of a $\s{G}_2$-manifold calibrated by $\phi$ and the Hodge dual of $\phi$, respectively; see \cite{HarveyLawson} and \cite[Chapters 4 and 12]{joyce} for an introduction to calibrated geometry.
As a consequence of Corollary~\ref{cor:tg-kahler-angle} and Theorem~\ref{th:tg-g2-cones}, one has the following:

\begin{corx}\label{cor:tg-associative-coassociative}
	Let $\Sigma$ be a maximal totally geodesic submanifold of the $\s{G}_2$-cone over a homogeneous nearly Kähler $6$-manifold of non-constant curvature.
	Then the following statements hold:
	\begin{enumerate}[\rm (i)]
		\item if $\Sigma$ has dimension three, then $\Sigma$ is an associative submanifold.
		\item if $\Sigma$ has dimension four, then $\Sigma$ is a coassociative submanifold.
	\end{enumerate}
\end{corx}

Both Corollaries~\ref{cor:tg-kahler-angle} and~\ref{cor:tg-associative-coassociative} seem to indicate that for totally geodesic submanifolds of nearly Kähler $6$-manifolds and their $\s{G}_2$-cones, there is a strong link between this purely Riemannian property and the underlying nearly Kähler and $\s{G}_2$-structures, respectively.
Consequently, we find that it would be very interesting to investigate whether both corollaries hold true without the homogeneity assumptions.

\subsection*{Organization of the paper}
In Section~\ref{sec:prelim}, we summarize some standard facts about Riemannian homogeneous spaces~(Subsection~\ref{subsec:homspaces}), and we describe the simply connected nearly Kähler homogeneous $6$-manifolds (Subsection~\ref{subsec:homNK}).
In Section~\ref{sec:tganalytic}, we recall and further develop the theory of totally geodesic immersions in analytic Riemannian manifolds.
In particular, we prove a characterization of inextendable totally geodesic immersions in terms of their geodesics, see Proposition~\ref{prop:TGInextendableCharacterization}; and we make sense of the maximality notion for the case of totally geodesic immersions, see Proposition~\ref{prop:TGInclusions}.
In Section~\ref{sec:tgnatred} we introduce and develop new tools for the study of totally geodesic submanifolds in naturally reductive homogeneous spaces.
In particular, in Subsection~\ref{subsec:Dinvtg}, we consider the class of $D$-invariant totally geodesic submanifolds, and we characterize those both from an algebraic and geometric point of view, see Theorem~\ref{th:CharacterizationDInvariantTGSubmanifolds}.
In Subsection~\ref{subsec:tgsurf}, we derive a useful criterion for the existence of totally geodesic surfaces in naturally reductive homogeneous spaces, see~Proposition~\ref{prop:TojoSurfaces}.
Also, in Subsection~\ref{subsec:wellpostg}, we characterize when a totally geodesic submanifold of the total space of a homogeneous fibration is well-positioned, see~Lemma~\ref{lemma:HomogeneousFibrationWellPositioned}.
In Section~\ref{sec:examples}, we exhibit and discuss the properties of the totally geodesic submanifolds of~$\C\s{P}^3$, the flag manifold $\s{F}(\C^3)$, and $\s{S}^3\times \s{S}^3$. 
In Section~\ref{sec:tgcones}, we study totally geodesic submanifolds of Riemannian cones.
In Subsection~\ref{subsec:tgcones}, we start by observing that totally geodesic submanifolds of a Riemannian manifold induce totally geodesic submanifolds in its Riemannian cone, see Lemma~\ref{lemma:tgcone}.
We also obtain a non-existence result for totally geodesic hypersurfaces in three distinct types of manifolds: Sasakian--Einstein manifolds, six-dimensional strictly nearly Kähler manifolds, and nearly parallel $\mathsf{G}_{2}$-manifolds.
Moreover, we prove  Theorem~\ref{th:TGSubmanifoldsOfCones}, which gives a structure result for totally geodesic submanifolds in Riemannian cones.
Finally, in Section~\ref{sec:mainths} we provide the proofs for the main theorems. 

\textbf{Acknowledgments}. We wish to thank Prof.~Jason Lotay and Prof.~Thomas~Leistner for very helpful discussions, and Prof.~Miguel Domínguez-Vázquez for his valuable comments on an earlier draft of this manuscript.
The first author would also like to thank his PhD advisor Prof. Jos\'{e} Carlos D\'{i}az-Ramos for his constant support.

\section{Preliminaries}\label{sec:prelim}
	
\subsection{Riemannian homogeneous spaces}\label{subsec:homspaces}
	
Let $M=\mathsf{G}/\mathsf{K}$ be a Riemannian homogeneous space and $o=e\mathsf{K}$.
We denote by $\mathsf{G}_{p}$ the isotropy subgroup of $p$, that is, the set of all elements of $\s{G}$ that fix $p$.
For our purposes we assume that $\mathsf{G}$ is connected and the action of $\s{G}$ on $M$ is almost effective, in the sense that $\bigcap_{p\in M}\mathsf{G}_{p}$ is a discrete subgroup of $\s{G}$.
The (Lie) group of isometries of $M$ is written as $I(M)$, and $I^{0}(M)\subseteq I(M)$ denotes its identity component.
Throughout this article, Lie groups are denoted by uppercase letters and their Lie algebras are denoted by their corresponding lowercase gothic letters.
	
To each $X\in \g{g}$ we can associate its corresponding fundamental vector field $X^{*}\in\mathfrak{X}(M)$ defined via $X^{*}_{p}=\frac{d}{dt}\vert_{t=0}\Exp(tX)\cdot p$ for all $p\in M$.
Clearly, $X^{*}$ is a Killing vector field. 
Moreover, one can see that the map $X\in\g{g}\mapsto X^{*}\in\g{X}(M)$ is a Lie algebra anti-homomorphism.
We may identify the quotient $\g{g}/\g{k}$ with the tangent space $T_{o}M$ via the isomorphism $X+\g{k}\mapsto X^{*}_{o}$, which also establishes an equivalence of representations between the isotropy representation $\mathsf{K}\to\mathsf{GL}(T_{o}M)$ and the adjoint representation $\mathsf{K}\to \mathsf{GL}(\g{g}/\g{k})$.
	
We say that $M$ is \textit{reductive} if there exists a vector subspace $\g{p}\subseteq \g{g}$ such that $\g{g}=\g{k}\oplus \g{p}$ and $\Ad(\mathsf{K})\g{p}=\g{p}$.
The subspace $\g{p}$ is called a \textit{reductive complement} and the decomposition $\g{g}=\g{k}\oplus \g{p}$ is known as a \textit{reductive decomposition}.
It can be shown that every Riemannian homogeneous space is reductive.
As a consequence, we may identify $\g{p}$ with $T_{o}M$ and the isotropy representation is equivalent to the adjoint representation $\mathsf{K}\to \mathsf{GL}(\g{p})$.
In particular, we may endow $\g{p}$ with the scalar product induced by the metric on $M$, and one sees that it is $\mathsf{K}$-invariant.
We denote both the metric on $M$ and the inner product on $\g{p}$ by $\langle\cdot,\cdot\rangle$.
If $X\in \g{g}$, we denote by $X_{\g{k}}$ and $X_{\g{p}}$ the unique vectors in $\g{k}$ and $\g{p}$ respectively such that $X=X_{\g{k}}+X_{\g{p}}$.
Furthermore, if $V$ is a Euclidean space and $W\subseteq V$ is a subspace, we denote by $v_{W}$ the orthogonal projection of $v\in V$ onto $W$. We define the symmetric bilinear map $U\colon\g{p}\times\g{p}\rightarrow \g{p}$ by
\begin{equation}\label{eq:UTensor}
	2 \langle U(X, Y),Z \rangle=\langle [Z,X]_{\g{p}},Y\rangle + \langle X, [Z,Y]_{\g{p}}\rangle, \quad X,Y,Z\in\g{p}.
\end{equation}
The reductive decomposition $\g{g}=\g{k}\oplus\g{p}$ is \textit{naturally reductive} if $U$ is identically zero.
In that case, one sees that the curves $\Exp(tX)\cdot o$ (where $X\in\g{p}$) are geodesics.
Suppose that $\g{g}$ admits an $\Ad(\s{G})$-invariant inner product $q$ and consider $\g{p}=\g{k}^\perp$, where $\g{k}^{\perp}$ denotes the orthogonal complement of $\g{k}$ in $\g{g}$ with respect to $q$.
Then, we say that the reductive decomposition $\g{g}=\g{k}\oplus\g{p}$ is \textit{normal homogeneous}. In particular, normal homogeneous reductive decompositions are naturally reductive.
	
Every reductive homogeneous space $M=\mathsf{G}/\mathsf{K}$ admits two natural connections.
On the one hand, we may consider the canonical connection $\nabla^{c}$, which is characterized by the equation
\[
	(\nabla^{c}_{X^{*}}Y^{*})_{o}=(-[X,Y]_{\g{p}})^{*}_{o}, \quad X,Y\in\g{p}.
\]
On the other hand, we have the Levi-Civita connection $\nabla$. We can also consider the tensor  $D:=\nabla - \nabla^c$. It turns out that $D$ is a $\s{G}$-invariant tensor defined on $M=\s{G}/\s{K}$ called the \textit{difference tensor}. Notice that $U=0$ holds if and only if the operators $A_{X}\colon Y\in \g{p}\to [X,Y]_{\g{p}}$ are skew-symmetric for all $X\in \g{p}$.
In particular, if $D$ is an element of $\operatorname{Hom}_{\s{K}}\bigl(\bigwedge^{2}\g{p},\g{p}\bigr)$, then  the reductive decomposition $\g{g}=\g{k}\oplus\g{p}$ is naturally reductive. 
	
Following \cite[Chapter X]{KobayashiNomizu2}, we can express $\nabla$ and $D$ at the base point $o\in M$ in terms of the Lie bracket of $\g{g}$ and the metric of $M$ in the following way:
\begin{equation*}
	(\nabla_{X^*} Y^*)_o=\left(-\frac{1}{2}[X,Y]_{\g{p}} + U(X,Y) \right)^*_o, \quad
	D_X Y=\frac{1}{2}[X,Y]_{\g{p}}+ U(X,Y),
\end{equation*}
where $X,Y\in\g{p}$. 
Moreover,  the curvature tensor associated with the canonical connection and the Levi-Civita connection can be written, respectively,  as
\begin{equation*}
	R^c(X,Y)Z=-[[X,Y]_{\g{k}},Z], \qquad 					
	R(X,Y)Z=D_X D_Y Z-D_Y D_X Z-[[X,Y]_{\g{k}},Z]-D_{[X,Y]_{\g{p}}}Z,
\end{equation*}
where $X,Y,Z\in \g{p}$.
We also need to compute the covariant derivative $\nabla R$ of the curvature tensor.
As shown in \cite[Remark~2.5]{RodriguezVazquezOlmos}, the tensor $\nabla R$ is given at $o$ by
\begin{equation*}
	\nabla_{V}R(X,Y,Z)=D_V(R(X,Y)Z)-R(D_V X,Y)Z-R(X,D_V Y)Z-R(X,Y)D_V Z
\end{equation*}
for all $X,Y,Z,V\in\g{p}$.
	
An important class of Riemannian submersions involving homogeneous spaces is that of homogeneous fibrations, which we briefly describe.
The main reference is \cite{GromollWalschap}.
Let $\s{H}\subseteq \s{K}\subseteq \s{G}$ be a chain of inclusions of compact connected subgroups of the connected Lie group $\s{G}$, and endow $\s{G}/\s{K}$ with a $\s{G}$-invariant Riemannian metric.
Then there exists a left-invariant metric on $\s{G}$ that is also right $\s{K}$-invariant, and a $\s{G}$-invariant metric on $\s{G}/\s{H}$ that makes the canonical projection $\pi\colon\s{G}/\s{H}\to\s{G}/\s{K}$ a Riemannian submersion with totally geodesic fibers isometric to $\s{K}/\s{H}$.
We denote by $\mathcal{V}$ and $\mathcal{H}$ the vertical and horizontal distributions associated with the Riemannian submersion $\pi$. 
We are concerned with the case that $\s{G}$ is a compact group with a bi-invariant metric and the homogeneous spaces $F=\s{K}/\s{H}$, $M=\s{G}/\s{H}$ and $B=\s{G}/\s{K}$ are endowed with the corresponding normal homogeneous metrics.
If $V$ is a Euclidean vector space and $W\subseteq V$ is a vector subspace, we denote by $V\ominus W$ the orthogonal complement of $W$ in $V$.
In this case, the tangent space $T_{o}M$ at $o=e\s{H}$ is identified with $\g{p}=\g{g}\ominus\g{h}$, so the vertical and horizontal subspaces at $o$ are $\mathcal{V}_{o}=\g{k}\ominus\g{h}$ and $\mathcal{H}_{o}=\g{p}\ominus \g{k}$.
It turns out that the distributions $\mathcal{V}$ and $\mathcal{H}$ are $\s{G}$-invariant in the sense that for every $p\in M$ and $g\in \s{G}$, we have $g_{*p}\mathcal{V}_{p}=\mathcal{V}_{g\cdot p}$ and $g_{*p}\mathcal{H}_{p}=\mathcal{H}_{g\cdot p}$.
	
A submanifold $N\subseteq M$ is said to be \textit{extrinsically homogeneous with respect to the presentation} $M=\s{G}/\s{K}$ if there exists a Lie subgroup $\s{S}\subseteq \s{G}$ such that $N$ is an orbit of $\s{S}$.
More generally, a submanifold $N$ of a Riemannian manifold $M$ is \textit{extrinsically homogeneous} if it is an orbit of a Lie subgroup $\s{S}\subseteq I(M)$.
In the following lemma, we provide a formula for the second fundamental form for an extrinsically homogeneous submanifold of a homogeneous space.
This formula was originally derived by Solonenko in the case of Riemannian symmetric spaces, see~\cite[Proposition~2.2.43]{Solonenko}.
See also \cite[Proposition~2.2]{AlekseevskyDiScala} for an alternative expression.
	
\begin{lemma}
	Let $M=\s{G}/\s{K}$ be a Riemannian homogeneous space equipped with an arbitrary $\s{G}$-invariant metric and reductive decomposition $\g{g}=\g{k}\oplus\g{p}$.
	Assume that $\s{S}$ is a Lie subgroup of $\s{G}$, and let $\g{s}_{\g{p}}$ and $\g{s}_{\g{p}}^{\perp}$ be the tangent and normal spaces to $\s{S}\cdot o$ at $o$ (regarded as subspaces of $\g{p}$).
	Let $V_{\g{s}_{\g{p}}^{\perp}}$ denote the orthogonal projection of $V\in\g{p}$ onto $\g{s}_{\g{p}}^{\perp}$.
	Then the second fundamental form of $\s{S}\cdot o$ at $o$ is given by
	\begin{equation}\label{eq:SecondFundamentalFormOrbit}
		\mathbb{II}(X_{\g{p}},Y)=([X_{\g{k}},Y]+D_{X_{\g{p}}}Y)_{\g{s}_{\g{p}}^{\perp}}
	\end{equation}
	for all $X\in\g{s}$ and $Y\in \g{s}_{\g{p}}$.
	In particular, $\s{S}\cdot o$ is totally geodesic if and only if $[X_{\g{k}},Y]+D_{X_{\g{p}}}Y\in \g{s}_{\g{p}}$ for all $X\in \g{s}$ and $Y\in\g{s}_{\g{p}}$.
	\end{lemma}
\begin{proof}
	Choose arbitrary elements $X$, $Y\in\g{s}$, so that the vector fields $X^{*}$ and $Y^{*}$ are tangent to $\s{S}\cdot o$ and their values at $o$ are $X_{\g{p}}$ and $Y_{\g{p}}$ respectively.
	We evaluate the covariant derivative $\nabla_{Y^{*}}X^{*}$ at $o$.
	We see that
	\[
	\begin{aligned}
		\nabla_{Y^{*}}X^{*}={}&\nabla_{Y_{\g{p}}^{*}}X^{*}=\nabla_{X^{*}}Y_{\g{p}}^{*}+[Y_{\g{p}}^{*},X^{*}]=\nabla_{X_{\g{p}}^{*}}Y_{\g{p}}^{*}+[X,Y_{\g{p}}]^{*}=\nabla_{X_{\g{p}}^{*}}Y_{\g{p}}^{*}+[X_{\g{k}},Y_{\g{p}}]^{*}+[X_{\g{p}},Y_{\g{p}}]^{*} \\
		={}&-\frac{1}{2}[X_{\g{p}},Y_{\g{p}}]_{\g{p}}+U(X_{\g{p}},Y_{\g{p}})+[X_{\g{k}},Y_{\g{p}}]+[X_{\g{p}},Y_{\g{p}}]_{\g{p}}=[X_{\g{k}},Y_{\g{p}}]+D_{X_{\g{p}}}Y_{\g{p}}.
	\end{aligned}
	\]
	Thus, projecting onto the normal space and using the fact that the second fundamental form is symmetric, we obtain that
	\[
	\mathbb{II}(X_{\g{p}},Y_{\g{p}})=(\nabla_{Y^{*}}X^{*})_{\g{s}_{\g{p}}^{\perp}}=([X_{\g{k}},Y_{\g{p}}]+D_{X_{\g{p}}}Y_{\g{p}})_{\g{s}_{\g{p}}^{\perp}},
	\]
	as desired. \qedhere
\end{proof}

\subsection{Homogeneous nearly Kähler manifolds}\label{subsec:homNK}
	
In this section we present the ambient spaces that we work with throughout the rest of this article.
The classification of homogeneous nearly Kähler manifolds in dimension six was done by Butruille \cite{Butruille}.
Indeed, every simply connected Riemannian manifold satisfying the previous conditions is homothetic to either the sphere $\mathsf{S}^{6}$, the complex projective space $\C\mathsf{P}^{3}$,  the flag manifold $\mathsf{F}(\C^{3})$, and $\mathsf{S}^{3}\times\mathsf{S}^{3}$.
Since $\mathsf{S}^{6}$ carries its natural round metric, its totally geodesic submanifolds are well-known, so we only need to focus on the other three spaces.
It turns out that these manifolds are examples of $3$-symmetric spaces.
The main reference for the description of $3$-symmetric spaces is \cite{GrayWolfI}.
	
Let $(M,J)$ be an almost Hermitian manifold.
Recall that $M$ is \textit{nearly Kähler} if for every vector field $X\in\g{X}(M)$ we have $(\nabla_{X}J)X=0$, and that it is \textit{strictly nearly Kähler} if it is nearly Kähler and $\nabla_{X}J\neq 0$ for all nonzero $X\in TM$.
If $\dim M = 6$, this condition is equivalent to $\nabla J \neq 0$.
On the other hand, a \textit{$3$-symmetric space} $M$ is a Riemannian manifold $M$ together with a family of isometries $s_{p}\colon M\to M$, where $p\in M$, satisfying $s_{p}^{3}=\operatorname{Id}_{M}$ for all $p\in M$, $p$ is an isolated fixed point of $s_{p}$, and each $s_{p}$ is holomorphic with respect to the so-called \textit{canonical almost complex structure} $J$ defined via
\begin{equation}\label{eq:3SymmetricSpaceJ}
	(s_{p})_{*p}=-\frac{1}{2}\operatorname{Id}_{T_{p}M}+\frac{\sqrt{3}}{2}J_{p}, \quad p\in M.
\end{equation}
Any connected 3-symmetric space is automatically homogeneous \cite[Theorem 4.8]{Gray}.
Conversely, one can construct a $3$-symmetric space in terms of algebraic data.
Indeed, let $\s{G}$ be a connected Lie group and $\s{K}$ a closed subgroup of $\s{G}$.
Assume that there exists an automorphism $\Theta\colon \s{G}\to\s{G}$ of order three such that $\s{G}^{\Theta}_{0}\subseteq\s{K}\subseteq\s{G}^{\Theta}$, where $\s{G}^{\Theta}$ is the fixed point set of $\Theta$ and $\s{G}^{\Theta}_{0}$ is its identity component.
It turns out that $M=\s{G}/\s{K}$ is a reductive homogeneous space in such a way that $\theta=\Theta_{*}$ preserves the reductive complement.
Let $\g{g}=\g{k}\oplus\g{p}$ be a reductive decomposition of $\g{g}$ satisfying $\theta\g{p}=\g{p}$.
Then, any inner product on $\g{p}$ that is invariant under $\Ad(\s{K})$ and $\theta$ gives rise to a $\s{G}$-invariant metric on $M$ that turns $M$ into a $3$-symmetric space, where the isometry of order three at $o=e\s{K}$ is given by $s_{o}(x\s{K})=\Theta(x)\s{K}$.
We say that $(\s{G},\s{K},\Theta)$ is the \textit{triple} associated with the $3$-symmetric space $M$.
The corresponding almost complex structure is the $\s{G}$-invariant tensor field $J$ defined at $o$ by~\eqref{eq:3SymmetricSpaceJ}.
By \cite[Proposition~5.6]{Gray}, the almost Hermitian manifold $(M,J)$ is nearly Kähler if and only if the decomposition $\g{g}=\g{k}\oplus\g{p}$ is naturally reductive.

We now proceed to describe our six-dimensional examples, exhibiting them as $3$-symmetric spaces.

\subsubsection{The complex projective space $\C\mathsf{P}^{3}$}\label{subsection:CP3Description}
	
Consider $\mathbb{H}^{2}$ as a right $\C$-vector space, so that the projective space $\mathsf{P}(\mathbb{H}^{2})$ is exactly $\C\mathsf{P}^{3}$.
The natural action of $\s{G}=\mathsf{Sp}(2)$ on $\C\mathsf{P}^{3}$ is transitive, and the isotropy subgroup at $o=[1:0]$ is $\s{K}=\mathsf{U}(1)\times\mathsf{Sp}(1)$, so that $\C\mathsf{P}^{3}$ can be viewed as the quotient $\mathsf{Sp}(2)/\mathsf{U}(1)\times\mathsf{Sp}(1)$.
The Killing form of $\g{g}=\g{sp}(2)$ is $\mathcal{B}(X,Y)=12\operatorname{Re}\operatorname{tr}_{\mathbb{H}}(XY)$, so $-\mathcal{B}$ is an $\Ad(\mathsf{Sp}(2))$-invariant inner product in $\g{g}$,
but we renormalize it so that the inner product on $\g{g}$ is $\langle X,Y\rangle = -2\operatorname{Re}\operatorname{tr}_{\mathbb{H}}(XY)$.
Let $\g{p}$ be the orthogonal complement of $\g{u}(1)\oplus\g{sp}(1)$ in $\g{sp}(2)$.
We endow $\C\mathsf{P}^{3}$ with the homogeneous metric induced by the restriction of $\langle \cdot, \cdot \rangle$ to $\g{p}$.
We also consider the element $\omega=\operatorname{diag}\bigl(e^\frac{2\pi i}{3},1\bigr)\in\s{K}$.
Then the conjugation $\Theta=I_{\omega}$ defines an inner automorphism of order three in $\mathsf{G}$, whose fixed point set is $\mathsf{Sp}(2)^{\Theta}=\mathsf{U}(1)\times\mathsf{Sp}(1)$, and $(\mathsf{Sp}(2),\mathsf{U}(1)\times\mathsf{Sp}(1),\Theta)$ is the triple associated with the $3$-symmetric space $\C\mathsf{P}^{3}$.
The nearly Kähler complex structure $J$ is defined as $J=\frac{2}{\sqrt{3}}\Theta_{*}+\frac{1}{\sqrt{3}}\operatorname{Id}_{\g{p}}$.

Let $E_{ij}\in\g{gl}(2,\H)$ be the elementary matrix which has all components equal to zero except for the $(i,j)$ component, which is one.
We use the orthonormal basis $\{e_{1},\dots,e_{6}\}$ of $\g{p}$ defined by
\begin{align*}
	e_{1}=&~\frac{j}{\sqrt{2}}E_{11}, & e_{2}=&~\frac{k}{\sqrt{2}}E_{11}, & e_{3}=&~\frac{1}{2}(E_{21}-E_{12}), \\
	e_{4}=&~\frac{i}{2}(E_{12}+E_{21}), & e_{5}=&~\frac{j}{2}(E_{12}+E_{21}), & e_{6}=&~\frac{k}{2}(E_{12}+E_{21}).
\end{align*}
	
The isotropy representation allows us to decompose $\g{p}$ as the direct sum of two irreducible submodules $\g{p}_{1}=\vecspan\{e_{1},e_{2}\}$ and $\g{p}_{2}=\vecspan\{e_{3},\dots,e_{6}\}$.
Indeed, the subrepresentation $\g{p}_{1}$ of $\mathsf{U}(1)\times\mathsf{Sp}(1)$ is isomorphic to the representation $\C$ with the action given by $(\lambda,\mu)\cdot z=\lambda^{2}z$, whereas $\g{p}_{2}$ is isomorphic to the representation $\R^{4}$ of $\mathsf{U}(1)\times\mathsf{Sp}(1)$ under $(\lambda,\mu)\cdot x=\mu x \bar{\lambda}$.
In particular, the group $\s{U}(1)\times\s{Sp}(1)$ acts transitively on the unit sphere of each $\g{p}_{i}$.
	
The isometry group of $\C\mathsf{P}^{3}$ is $I(\C\mathsf{P}^{3})=(\mathsf{Sp}(2)/\mathbb{Z}_{2})\rtimes \mathbb{Z}_{2}$, where the outer $\mathbb{Z}_{2}$ is generated by conjugation by $\operatorname{diag}(j,1)\in\mathsf{Sp}(2)$ (see for example~\cite{ShankarIsometryGroups}).
	
Now, consider the chain of subgroups $\mathsf{U}(1)\times\mathsf{Sp}(1)\subseteq \mathsf{Sp}(1)\times\mathsf{Sp}(1)\subseteq\mathsf{Sp}(2)$.
This gives rise to the homogeneous fibration $\C\mathsf{P}^{1}\to\C\mathsf{P}^{3}\to\mathbb{H}\mathsf{P}^{1}=\mathsf{S}^{4}$, which is precisely the twistor fibration, the fiber of which is a totally geodesic $\C\mathsf{P}^{1}=\mathsf{S}^{2}$.
The decomposition of $\g{p}$ into the vertical and horizontal subspaces of this submersion is given by $\mathcal{V}_{o}=\g{p}_{1}$ and $\mathcal{H}_{o}=\g{p}_{2}$.

\subsubsection{The flag manifold $\mathsf{F}(\C^{3})$}\label{subsection:FlagDescription}
Recall that a full flag in $\C^3$ is a chain $0=V_{0}\subseteq V_{1}\subseteq V_{2}\subseteq V_{3}=\C^3$ (also denoted by $(V_1,V_2)$) of subspaces such that $\dim_{\C}V_{k}=k$ for each $k$.
We denote by $\s{F}(\C^{3})$ the space of all flags in $\C^3$, which is naturally identified with the quotient of the Stiefel manifold of orthonormal bases of $\C^3$ under the standard action of $\s{U}(1)^{3}$.
The group $\s{G}=\s{SU}(3)$ acts transitively on $\s{F}(\C^3)$, and if $o$ is the standard flag $0\subseteq \C e_1\subseteq \spann\{e_1,e_2\}\subseteq \C^3$, its isotropy subgroup is the maximal torus $\s{K}=\s{T}^{2}$ of diagonal matrices in $\s{SU}(3)$, so we have $\s{F}(\C^3)=\s{SU}(3)/\s{T}^{2}$.
	
Let us endow $\mathsf{F}(\C^{3})$ with a reductive decomposition and a Riemannian metric.
Note that the Killing form of $\g{g}=\g{su}(3)$ satisfies $\mathcal{B}(X,Y)=6\operatorname{tr}(XY)$ for all $X$, $Y\in \g{g}$.
As a consequence, the negative Killing form gives a bi-invariant metric on $\mathsf{SU}(3)$.
However, for the sake of convenience, we rescale this metric so that the inner product in $\g{g}$ is $\langle X,Y \rangle = -\operatorname{tr}(XY)$ for all $X$, $Y\in\g{su}(3)$.
Let $\g{p}$ be the orthogonal complement of $\g{t}=\g{u}(1)\oplus\g{u}(1)$ in $\g{g}$.
Then, the restriction of $\langle \cdot ,\cdot \rangle$ to $\g{p}$ induces an $\Ad(\mathsf{T}^2)$-invariant inner product on $\g{p}$, that is, a $\mathsf{G}$-invariant metric on $\mathsf{F}(\C^{3})$.
This metric is homothetic to the standard homogeneous metric on $M$.
We also consider the automorphism $\Theta=I_{\omega}\colon \mathsf{SU}(3)\to\mathsf{SU}(3)$, where $\omega=\operatorname{diag}\bigl(e^{\frac{2\pi i}{3}},1,e^{-\frac{2\pi i}{3}}\bigr)\in\s{T}^{2}$.
Then $\Theta$ is an automorphism of order three whose fixed point set is precisely $\mathsf{SU}(3)^{\Theta}=\mathsf{T}^2$, so $(\mathsf{SU}(3),\mathsf{T}^2,\Theta)$ is the triple associated with the $3$-symmetric space $\mathsf{F}(\C^{3})$.
The corresponding almost complex structure $J$ at $\g{p}$ is determined by $J=\frac{2}{\sqrt{3}}\Theta_{*}+\frac{1}{\sqrt{3}}\operatorname{Id}_{\g{p}}$.
We consider the orthonormal basis $\{e_{1},\dots, e_{6}\}$ of $\g{p}$, where
\begin{align*}
	e_1=&~\frac{1}{\sqrt{2}}(E_{12}-E_{21}), & e_2=&~\frac{i}{\sqrt{2}}(E_{12}+E_{21}), & e_3=&~\frac{1}{\sqrt{2}}(E_{23}-E_{32}), \\
	e_4=&~\frac{i}{\sqrt{2}}(E_{23}+E_{32}), & e_5=&~\frac{1}{\sqrt{2}}(E_{13}-E_{31}), & e_6=&~\frac{i}{\sqrt{2}}(E_{13}+E_{31}).
\end{align*}
Once again, the $E_{ij}$ denote elementary $3\times 3$ matrices.
It is easy to check that the tangent space splits as the direct sum of irreducible submodules $\g{p}=\g{p}_{1}\oplus\g{p}_{2}\oplus\g{p}_{3}$, where each $\g{p}_{k}=\vecspan\{e_{2k-1},e_{2k}\}$ is isomorphic to $\C$.
To be more precise, if $g=\operatorname{diag}\bigl(e^{i x},e^{i y},e^{-i (x+y)}\bigr)$ is an arbitrary element of $\mathsf{T}^2$, then $\Ad(g)$ acts on $\g{p_1}$ as multiplication by $e^{i(x-y)}$, on $\g{p}_{2}$ as multiplication by $e^{i(x+2y)}$, and on $\g{p}_{3}$ as multiplication by $e^{i(2x+y)}$.
Note that $\g{p}_{1}$, $\g{p}_{2}$ and $\g{p}_{3}$ are pairwise nonisomorphic as representations of $\mathsf{T}^2$.
Furthermore, if $g\in\s{U}(3)$ is a permutation matrix, then the map $a\s{T}^{2}\to gag^{-1}\s{T}^{2}$ is an isometry fixing $o$ and whose differential at $o$ permutes the irreducible submodules of $\g{p}$, and every permutation of these submodules can be achieved in this way.
For example, the transposition $(1,2)$ interchanges $\g{p}_{2}$ and $\g{p}_{3}$ and the cycle $(1,2,3)$ acts as the cycle $(\g{p}_{1},\g{p}_{2},\g{p}_{3})$.
	
Consider the chain of inclusions $\mathsf{T}^2\subseteq \mathsf{U}(2)\subseteq\mathsf{SU}(3)$.
The corresponding homogeneous fibration is $\C\mathsf{P}^{1}\to\mathsf{F}(\C^{3})\to\C\mathsf{P}^{2}$ (explicitly, it takes the flag $(V_1,V_2)\in\s{F}(\C^{3})$ to $V_2^\perp\in\C\s{P}^{2}$) and the fiber $\C\mathsf{P}^{1}=\mathsf{U}(2)\cdot o$ is totally geodesic.
The vertical and horizontal subspaces of this fibration at $o$ are precisely $\mathcal{V}_{o}=\g{p}_{1}$ and $\mathcal{H}_{o}=\g{p}_{2}\oplus\g{p}_{3}$.
		
We now determine the full isometry group of $\s{F}(\C^{3})$.
This computation was done by Shankar in \cite{ShankarIsometryGroups} when $\s{F}(\C^{3})$ carries a metric of positive sectional curvature. However, the homogeneous metric that we are considering in $\s{F}(\C^{3})$ does not have positive sectional curvature. In our case, we may calculate the isometry group of $\s{F}(\C^{3})$ via the following approach (based on the proof of~\cite[\S4, Proposition~6 and \S16, Theorem~3]{Onishchik}).
Firstly, the effectivized version of the presentation $\s{SU}(3)/\s{T}^2$ is $\s{PSU(3)}/(\s{T}^2/\mathbb{Z}_{3})$, and we may apply~\cite[Theorem~5.1]{WangZiller} to conclude that $I^{0}(\s{F}(\C^{3}))=\s{PSU}(3)$.
As for the group of components $I(\s{F}(\C^{3}))/I^{0}(\s{F}(\C^{3}))$, since the flag manifold is simply connected, this group is equal to $\s{H}/\s{H}^{0}$, where $\s{H}$ is the isotropy subgroup of $I(\s{F}(\C^{3}))$ at $o$ and $\s{H}^{0}$ its identity component.
This follows from the long exact sequence of homotopy groups associated with the fibration $\s{H}\hookrightarrow I(\s{F}(\C^{3}))\to\s{F}(\C^{3})$.
Now,  the conjugation map $C\colon \s{H}\to\s{Aut}(\s{PSU}(3),\s{T}^2/\mathbb{Z}_{3})=\{\varphi\in\s{Aut}(\s{PSU}(3)): \varphi \text{ preserves } \s{T}^2/\mathbb{Z}_{3}\}$ is injective by~\cite[Proposition~1.7]{ShankarIsometryGroups}.
In addition, any $\varphi\in\s{Aut}(\s{PSU}(3),\s{T}^2/\mathbb{Z}_{3})$ descends to a diffeomorphism $\bar{\varphi}$ of $\s{F}(\C^{3})$, which is actually an isometry because $\varphi$ preserves the Killing form and the metric on the flag manifold is induced by it.
It is easy to show that $\bar{\varphi}\in\s{H}$ and $C(\bar{\varphi})=\varphi$, so $C$ is an isomorphism and $\s{H}=\s{Aut}(\s{PSU}(3),\s{T}^2/\mathbb{Z}_{3})$.
Now, since $\s{Aut}(\s{PSU}(3))=\Ad(\s{PSU}(3))\rtimes\mathbb{Z}_{2}$, where the outer $\mathbb{Z}_{2}$ is generated by complex conjugation, the computation of $\s{H}/\s{H}^{0}$ reduces to that of $\frac{\s{N}_{\s{PSU}(3)}(\s{T}^2/\mathbb{Z}_{3})}{\s{Z}_{\s{PSU}(3)}(\s{T}^2/\mathbb{Z}_{3})}\rtimes \mathbb{Z}_{2}$.
The first factor is merely the Weyl group $\s{W}(\s{PSU}(3))=\g{S}_{3}$, so we have obtained $\s{H}/\s{H}^{0}=\g{S}_{3}\rtimes\mathbb{Z}_{2}$, and we conclude that the full isometry group is $I(\mathsf{F}(\C^{3}))=\mathsf{PSU}(3)\rtimes (\g{S}_{3}\rtimes\mathbb{Z}_{2})$.
	
\subsubsection{The almost product $\mathsf{S}^{3}\times\mathsf{S}^{3}$}\label{subsection:S3S3Description}
This space is obtained via the Ledger--Obata construction from the group $\mathsf{SU}(2)$ (see for example \cite{LedgerObata}).
We consider the product $\mathsf{G}=\mathsf{SU}(2)^3$ and the subgroup $\mathsf{K}=\Delta\mathsf{SU}(2)$ obtained by embedding $\mathsf{SU}(2)$ diagonally in $\mathsf{G}$.
Then $\mathsf{G}$ acts on $M=\mathsf{S}^{3}\times\mathsf{S}^{3}=\mathsf{SU}(2)\times\mathsf{SU}(2)$ via the equation $(g,h,k)\cdot(x,y)=(gxk^{-1},hyk^{-1})$, and the isotropy subgroup at $o=(I,I)$ is $\mathsf{K}$, so we obtain that $M=\mathsf{G}/\mathsf{K}$.
The Killing form of $\g{su}(2)$ is $\mathcal{B}(X,Y)=4\operatorname{tr}(XY)$, and the direct sum $\mathcal{B}\oplus\mathcal{B}\oplus\mathcal{B}$ is precisely the Killing form of $\g{g}$ whose opposite yields the standard homogeneous metric on $M$.
Similarly to the previous cases, we consider the renormalized metric given by $\langle (X_{1},X_{2},X_{3}),(Y_{1},Y_{2},Y_{3}) \rangle=-\operatorname{tr}(X_{1} Y_{1})-\operatorname{tr}(X_{2}Y_{2})-\operatorname{tr}(X_{3}Y_{3})$.
We denote by $\g{p}$ the orthogonal complement of $\g{k}=\Delta\g{su}(2)$ in $\g{g}$ and we consider the order three automorphism $\Theta\colon \mathsf{G}\to\mathsf{G}$ defined by $\Theta(g,h,k)=(h,k,g)$.
The nearly Kähler complex structure $J$ is given as $J=\frac{2}{\sqrt{3}}\Theta_{*}+\frac{1}{\sqrt{3}}\operatorname{Id}_{\g{p}}$.
Furthermore, if $L_{g}\colon \s{S}^{3}\to\s{S}^{3}$ denotes left multiplication by $g\in\s{S}^{3}$, the almost product structure of $\s{S}^{3}\times\s{S}^{3}$ is the $\s{G}$-invariant tensor $P$ of type $(1,1)$ defined by
\begin{equation*}
	P(v,w)=\left(\left(L_{ab^{-1}}\right)_{*b}w,\left(L_{ba^{-1}}\right)_{*a}v\right), \quad v\in T_{a}\s{S}^{3},w\in T_{b}\s{S}^{3}.
\end{equation*}
The restriction of $P$ to $T_{o}(\s{S}^{3}\times\s{S}^{3})$ is identified with the $\Ad(\s{K})$-invariant map $P\colon \g{p}\to \g{p}$ given by $P(X,Y,Z)=(Y,X,Z)$.
	
We choose the following basis of $\g{su}(2)$ in terms of elementary matrices:
\begin{equation}\label{eq:su2Matrices}
	H=i(E_{11}-E_{22}),\quad
	E=E_{21}-E_{12},\quad
	F=i(E_{12}+E_{21}).
\end{equation}
Then we can give a basis $\{e_{1},\dots,e_{6}\}$ of $\g{p}$ as follows:
\begin{align*}
	e_{1}={}&\frac{1}{\sqrt{12}}\left(H,-2H,H\right), & e_{2}={}&\frac{1}{\sqrt{12}}\left(E,-2E,E\right), & e_{3}={}&\frac{1}{\sqrt{12}}\left(F,-2F,F\right), \\
	e_{4}={}&\frac{1}{2}\left(H,0,-H\right), &
	e_{5}={}&\frac{1}{2}\left(E,0,-E\right), &
	e_{6}={}&\frac{1}{2}\left(F,0,-F\right).
\end{align*}
Consider the chain of inclusions $\Delta\mathsf{SU}(2)\subseteq \Delta_{1,3}\mathsf{SU}(2)\times\mathsf{SU}(2)_{2}\subseteq\mathsf{SU}(2)\times\mathsf{SU}(2)\times\mathsf{SU}(2)$.
This gives rise to the homogeneous fibration $\mathsf{S}^{3}\hookrightarrow \mathsf{S}^{3}\times\mathsf{S}^{3}\to\mathsf{S}^{3}$ which is merely the projection onto the first factor.
Once again the fibers are totally geodesic.
The vertical and horizontal subspaces at $o$ are given by
\begin{equation*}
	\mathcal{V}_{o}=\g{p}_{1}=\vecspan\{e_{1},e_{2},e_{3}\}, \quad
	\mathcal{H}_{o}=\g{p}_{2}=\vecspan\{e_{4},e_{5},e_{6}\}.
\end{equation*}	
The isometry group of $\mathsf{S}^{3}\times\mathsf{S}^{3}$ is $I(\mathsf{S}^{3}\times\mathsf{S}^{3})=(\mathsf{SU}(2)^{3}/\Delta \mathbb{Z}_{2})\rtimes \mathfrak{S}_{3}$
where $\mathfrak{S}_{3}$ denotes the symmetric group on three elements acting in the natural manner on $\mathsf{S}^{3}\times\mathsf{S}^{3}$, see~\cite[Lemma 3.3]{VasquezIsospectral} for a proof.
	
\section{Totally geodesic immersions in real analytic Riemannian manifolds}
\label{sec:tganalytic}
In this section we develop some new tools in the theory of totally geodesic immersions in analytic Riemannian manifolds.

To conduct the global study of totally geodesic submanifolds, it is natural to assume that the ambient space is real analytic, since homogeneous spaces are examples of real analytic manifolds~\cite[Proposition~4.2]{KobayashiNomizu}.
As we will see, it is also necessary to allow self-intersections in our submanifolds, leading to the study of totally geodesic immersions in analytic Riemannian manifolds. 

Let $M$ be a Riemannian manifold and $f\colon \Sigma \to M$ an isometric immersion of a connected Riemannian manifold $\Sigma$ to $M$.
We say that $f$ is a \textit{totally geodesic} immersion (and $\Sigma$ is a \textit{totally geodesic submanifold} of $M$) if for every geodesic $\gamma$ of $\Sigma$ the composition $f\circ \gamma$ is also a geodesic of $M$.
This condition is equivalent to the vanishing of the second fundamental form $\mathbb{II}$ of $\Sigma$.
We refer the reader to~\cite[Appendix~A]{HeintzeLiuOlmos} and~\cite[Section~10.3]{BerndtConsoleOlmos} for more details on totally geodesic immersions.
A similar treatment can be found in~\cite{Tsukada}.

When studying general totally geodesic immersions, certain redundancies arise that need to be avoided.
The first one is due to reparametrizations of the domain, which is handled by introducing the notion of equivalence.
We say that two totally geodesic immersions $f_{i}\colon \Sigma_{i}\to M$ are \textit{equivalent} if there exists an isometry $h\colon \Sigma_{1}\to \Sigma_{2}$ such that $f_{1}=f_{2}\circ h$.
The second redundancy occurs when considering surjective local isometries.
For instance, a totally geodesic embedding of $\R\s{P}^2$ in $\R\s{P}^3$ yields the same information as its composition $\s{S}^{2}\to\R\s{P}^{3}$ with the universal cover $\s{S}^{2}\to\R\s{P}^{2}$, though the latter is not an embedding.
We can deal with this issue by working exclusively with compatible immersions.
Let $f\colon \Sigma^{k}\to M$ be totally geodesic, and consider the Grassmann bundle $\s{G}_{k}(TM)$ of $k$-planes in $TM$.
Then $f$ induces a smooth map $\tilde{f}\colon \Sigma \to \s{G}_{k}(TM)$ by letting $\tilde{f}(p)=(f(p),f_{*p}(T_{p}\Sigma))$ (for ease of notation, we simply write $\tilde{f}(p)=f_{*p}(T_{p}\Sigma)$).
The map $f$ is said to be \textit{compatible} if $\tilde{f}$ is injective (this is the case, for example, if $f$ is injective).
In other words, we allow $\Sigma$ to have self-intersections in $M$, but at each point of self-intersection the corresponding tangent spaces must be different.
It turns out that any totally geodesic immersion can be factored through a compatible one~\cite[p.~272]{BerndtConsoleOlmos}, so imposing this condition does not cause us to lose information.
We note that in general, given compatible totally geodesic immersions $f_{1}\colon \Sigma_{1}\to\Sigma_{2}$ and $f_{2}\colon \Sigma_{2}\to M$, the composition $f_{2}\circ f_{1}$ may not be compatible.
A counterexample will be provided in Remark~\ref{remark:CompositionCompatibleImmersions}.

One sees that a compatible immersion $f\colon \Sigma \to M$ is determined up to equivalence by the image $\tilde{f}(\Sigma)\subseteq \s{G}_{k}(TM)$.
Given two compatible totally geodesic immersions $f_{i}\colon \Sigma_{i}\to M$, we say that $f_{2}\colon \Sigma_{2}\to M$ \textit{extends} $f_{1}\colon \Sigma_{1}\to M$ if there exists an injective local isometry $j\colon \Sigma_{1}\to \Sigma_{2}$ satisfying $f_{1}=f_{2}\circ j$, or equivalently, if $\tilde{f}_{1}(\Sigma_{1})\subseteq \tilde{f}_{2}(\Sigma_{2})$.
This defines a partial ordering on any set of equivalence classes of compatible totally geodesic immersions.
By~\cite[Proposition~10.3.2]{BerndtConsoleOlmos}, every compatible totally geodesic immersion can be extended to an \textit{inextendable} one (that is, an immersion which is maximal with respect to this partial ordering), and this extension is unique up to equivalence (we say in this case that the submanifold is \textit{inextendable}). Note that Heintze, Liu and Olmos~\cite{HeintzeLiuOlmos} refer to these submanifolds as \textit{maximal} totally geodesic submanifolds, whereas we prefer to reserve this term for another concept.
We thus aim to classify inextendable totally geodesic submanifolds up to congruence.

Moreover, given $p\in \Sigma$ and a vector subspace $V\subseteq T_{p}M$, if there exists an inextendable totally geodesic immersion $f\colon \Sigma \to M$ with such that $(p,V)\in\tilde{f}(\Sigma)$, then $f$ is unique up to equivalence.
In other words, there exists at most one inextendable totally geodesic submanifold passing through $p$ with tangent space $V$.
Because of this, we are interested in determining which subspaces $V$ of $T_{p}M$ are tangent to a totally geodesic submanifold.
The answer to this question is given by means of the following result, see \cite[Section~10.3.2]{BerndtConsoleOlmos}:
\begin{theorem}\label{th:CharacterizationTGSubmanifolds}
	Let $M$ be a real analytic Riemannian manifold, $p\in M$ and $V\subseteq T_{p}M$ a vector subspace.
    Then there exists a totally geodesic submanifold $\Sigma$ of $M$ whose tangent space at $p$ is $V$ if and only if $(\nabla^{k}R)_{p}$ preserves $V$ for all $k\geq 0$.
\end{theorem}

Motivated by the previous theorem, we introduce the following definition.
A vector subspace $V\subseteq T_{p}M$ is \textit{totally geodesic} if there exists a totally geodesic immersion $f\colon \Sigma\to M$ such that $V\in\tilde{f}(\Sigma)$.

\begin{remark}
	For the sake of convenience, we refer to the totally geodesic submanifold \linebreak$f\colon \Sigma \to M$ as $\Sigma$ during the rest of this paper unless we need to specify the immersion.
	In most cases, the totally geodesic immersions that will appear are inclusions of embedded submanifolds $\Sigma\subseteq M$, so in this context there is no ambiguity in omitting the immersion. 
\end{remark}
\color{black}

For each $X\in T_{p}M$ we can define its corresponding \textit{Jacobi operator} $R_{X}\colon T_{p}M\to T_{p}M$ via the equation $R_{X}Y=R(Y,X)X$.
Moreover, we may also consider the \textit{Cartan operator} $C_{X}\colon T_{p}M\to T_{p}M$ given by $C_{X}Y=\nabla R(X,X,Y,X)$.
Both linear maps are symmetric with respect to the inner product on $T_{p}M$, so we may decompose $T_{p}M$ as the orthogonal direct sum of the eigenspaces of $R_{X}$, as well as the orthogonal direct sum of the eigenspaces of $C_{X}$.

\begin{remark}\label{remark:JacobiInvariance}
	If $V\subseteq T_{p}M$ is a totally geodesic subspace and $X\in T_{p}M$ is any vector, it follows from Theorem~\ref{th:CharacterizationTGSubmanifolds} that $V$ is an invariant subspace for $R_{X}$ and $C_{X}$.
	In particular, if $T_{p}M=(T_{p}M)_{\lambda_{1}}\oplus\dots\oplus (T_{p}M)_{\lambda_{r}}$ is the eigenspace decomposition of $T_{p}M$ with respect to $R_{X}$, we have $V=(V\cap (T_{p}M)_{\lambda_{1}})\oplus \dots \oplus (V\cap (T_{p}M)_{\lambda_{r}})$. 
	The same argument can be applied for the operator $C_{X}$.
\end{remark}	

We also establish the following notation.
Given a smooth curve $\alpha\colon I\to P$, where~$I$ is an interval and $P$ is a Riemannian manifold, the parallel transport map from $\alpha(t_{1})\in P$ to $\alpha(t_{2})\in P$ along $\alpha$ is denoted by $\mathcal{P}_{t_{1},t_{2}}^{\alpha}$.

In this paper we focus on real analytic manifolds which are not necessarily complete (see Section~\ref{sec:tgcones}).
Because the exponential map of a real analytic Riemannian manifold is also real analytic, one sees that for any compatible totally geodesic immersion $f\colon \Sigma\to M$, the domain $\Sigma$ is also a real analytic manifold and $f$ is a real analytic map.

We now provide a criterion that allows us to determine whether a totally geodesic submanifold $\Sigma$ of $M$ is inextendable in terms of the maximal geodesics of $\Sigma$ and $M$.
The key observation that makes use of real analyticity is the following:
if $p\in M$ and $V\subseteq T_{p}M$ is a totally geodesic subspace, then for every geodesic $\gamma\colon I\subseteq \R\to M$ satisfying $\gamma(0)=p$ and $\gamma'(0)\in V$, the parallel translates $\mathcal{P}_{0,t}^{\gamma}V$ are also totally geodesic subspaces of $T_{\gamma(t)}M$ for all $t\in I$ (see for example the proof in~\cite[Corollary~3.7]{Tsukada}).

\begin{proposition}\label{prop:TGInextendableCharacterization}
	Let $M$ be a connected real analytic Riemannian manifold and $f\colon \Sigma \to M$ a compatible totally geodesic immersion.
	Then the following conditions are equivalent:
	\begin{enumerate}[\rm (i)]
		\item\label{propItem:Inextendable} The immersion $f$ is inextendable.
		\item\label{propItem:MaximalToMaximal} For every maximal geodesic $\gamma\colon I\subseteq \R \to \Sigma$, the composition $f\circ \gamma\colon I \to M$ is also a maximal geodesic of $M$.
	\end{enumerate}
\end{proposition}
\begin{proof}
	Firstly, suppose that~(\ref{propItem:Inextendable}) holds.
	If~(\ref{propItem:MaximalToMaximal}) is not satisfied, then there exists a geodesic $\gamma\colon [0,1)\to\Sigma$ with $\gamma(0)=x$ that is not extendable to the right, while the composition $f\circ\gamma$ admits an extension $\sigma\colon[0,1]\to M$.
	We consider $q=\sigma(1)$.
	The subspace $W=\mathcal{P}_{0,1}^{\sigma}V\subseteq T_{q}M$, where $V=\tilde{f}(x)$, is therefore totally geodesic.
	
	We show that $W\notin \tilde{f}(\Sigma)$.
	Indeed, suppose that for a certain $y\in \Sigma$ we have $W=\tilde{f}(y)$.
	Since $-\sigma'(1)=-\mathcal{P}_{0,1}^{\sigma}\sigma'(0)\in W=f_{*y}(T_{y}\Sigma)$, there exists a geodesic $\alpha\colon [0,\delta)\to \Sigma$ satisfying $\alpha(0)=y$ and $\alpha'(0)=-(f_{*y})^{-1}\sigma'(1)$.
	Suppose without loss of generality that $\delta < 1$.
	The composition $f\circ \alpha$ satisfies $f(\alpha(0))=q$ and $(f\circ \alpha)'(0)=-\sigma'(1)$, so $f(\alpha(t))=\sigma(1-t)$ for all $t \in [0,\delta)$.
	Furthermore, we have
	\[
	\begin{aligned}
		\tilde{f}(\alpha(t))={}&f_{*\alpha(t)}(T_{\alpha(t)}\Sigma)=f_{*\alpha(t)}(\mathcal{P}_{0,t}^{\alpha}\Sigma)=\mathcal{P}_{0,t}^{f\circ \alpha} f_{*y}(T_{y}\Sigma)=\mathcal{P}_{1,1-t}^{\sigma}W=\mathcal{P}_{0,1-t}^{\sigma}V \\
		={}&\mathcal{P}_{0,1-t}^{f\circ \gamma}V=\mathcal{P}_{0,1-t}^{f\circ\gamma}f_{*x}(T_{x}\Sigma)=f_{*\gamma(1-t)}(T_{\gamma(1-t)}\Sigma)=\tilde{f}(\gamma(1-t)),
	\end{aligned}
	\]
	so using that $\tilde{f}$ is injective we see that $\alpha(t)=\gamma(1-t)$ for all $t\in [0,\delta)$.
	Because $\alpha$ is continuous at $0$, we obtain that the limit $\lim_{t\to 1^{-}}\gamma(t)$ exists and coincides with $y$, but this contradicts the fact that $\gamma$ is not extendable to the right.
	We deduce that $W$ is not in the image of $\tilde{f}$.
	
	We now consider an $\varepsilon > 0$ sufficiently small so that $\exp_{q}\colon B_{T_{q}M}(0,\varepsilon)\to M$ is a diffeomorphism onto its image and $S=\exp_{q}(B_{T_{q}M}(0,\varepsilon)\cap W)$ is a totally geodesic submanifold of $M$.
	As $-\sigma'(1)\in W$, there exists a $\delta > 0$ such that $\sigma(t)\in S$ for all $t\in (1-\delta,1]$.
	In particular, we have for all $t\in (1-\delta,1)$ that
	\[
	\begin{aligned}
		\tilde{f}(\gamma(t))={}&f_{*\gamma(t)}(T_{\gamma(t)}\Sigma)=\mathcal{P}_{0,t}^{f\circ\gamma}V=\mathcal{P}_{0,t}^{\sigma}\mathcal{P}_{1,0}^{\sigma}V=\mathcal{P}_{1,t}^{\sigma}V=\tilde{i}(\sigma(t)),
	\end{aligned}
	\]
	where $i\colon S\hookrightarrow M$ is the inclusion map.
	This means that $\tilde{f}(\Sigma)\cap\tilde{i}(S)\neq \varnothing$, so~\cite[Lemma~10.3.1]{BerndtConsoleOlmos} allows us to construct a compatible totally geodesic immersion that strictly extends $f$ and $i$, as $W=T_q S$ is not contained in $\widetilde{f}(\Sigma)$, contradicting the fact that $f$ is inextendable.
	
	Conversely, suppose that $f$ satisfies~(\ref{propItem:MaximalToMaximal}) and let $g\colon E\to M$ be an extension of $f$.
	By definition, we can find an injective local isometry $\phi\colon \Sigma\to E$ satisfying $f=g\circ \phi$.
	Replacing $\Sigma$ by $\phi(\Sigma)$, we may suppose directly that $\Sigma\subseteq E$ is an open subset and $f=g\vert_{\Sigma}$.
	If we show that $\Sigma$ is also closed, then we may conclude that $\Sigma=E$ and $f=g$.
	
	If $\Sigma$ is not closed in $E$, then we can find a geodesic $\gamma\colon [0,1]\to E$ such that $\gamma(t)\in\Sigma$ for all $t\in [0,1)$ and $\gamma(1)\in E\setminus \Sigma$.
	Write $x=\gamma(0)\in \Sigma$.
	The composition $f\circ \gamma\colon [0,1]\to M$ is also a geodesic with $f(\gamma(0))=f(x)$ and $(f\circ \gamma)'(0)\in \tilde{f}(x)$, so by~(\ref{propItem:MaximalToMaximal}) we may find a geodesic $\beta\colon [0,1]\to \Sigma$ satisfying $\beta(0)=x$ and $f(\beta(t))=f(\gamma(t))$ for all $t\in [0,1]$.
	By uniqueness of $E$-geodesics, we have $\beta=\gamma$, so $\gamma(1)=\beta(1)\in \Sigma$, which gives a contradiction.
	We conclude that $\Sigma=E$ and $f=g$, and because the choice of $g$ is arbitrary we obtain that $\Sigma$ does not admit a proper extension, thus showing~(\ref{propItem:Inextendable}). \qedhere
\end{proof}
Thus, as a consequence of the Hopf--Rinow theorem and Proposition~\ref{prop:TGInextendableCharacterization}, one obtains the following corollary which generalizes a result by Hermann~\cite{HermannCompleteTotallyGeodesic} to the case of non-complete ambient manifolds.

\begin{corollary}\label{cor:CompleteTGSubmanifolds}
	Let $M$ be a connected real analytic Riemannian manifold, $p\in M$ and $V$ a totally geodesic subspace of $T_{p}M$.
	If $f\colon \Sigma\to M$ is the inextendable compatible totally geodesic immersion associated with the subspace $V$, then $\Sigma$ is complete if and only if the exponential map $\exp_{p}$ is defined on all $V$.
\end{corollary}

\subsection{Maximal totally geodesic submanifolds in analytic Riemannian manifolds}
\label{subsec:maxtganalytic}

We now concern ourselves with defining a notion of maximality for totally geodesic submanifolds.
Indeed, if $M$ is a real analytic Riemannian manifold and $\Sigma_{1}$, $\Sigma_{2}\subseteq M$ are two inextendable and embedded totally geodesic submanifolds, one can wonder if $\Sigma_{1}\subseteq \Sigma_{2}$.
In this case, it is easy to see that $\Sigma_{1}\subseteq \Sigma_{2}$ if and only if there exists a point $p\in \Sigma_{1}\cap\Sigma_{2}$ such that $T_{p}\Sigma_{1}\subseteq T_{p}\Sigma_{2}$.
Therefore, the study of inclusions between embedded totally geodesic submanifolds of $M$ containing the point $p$ is equivalent to that of inclusions between totally geodesic subspaces of $T_{p}M$.
For general totally geodesic immersions, the situation is more involved, and one needs to introduce the following ``pullback-type" construction to make sense of the inclusion relationship.

\begin{proposition}\label{prop:TGInclusions}
	Let $M$ be a connected real analytic Riemannian manifold, $p\in M$ and $V_{1}$, $V_{2}\subseteq T_{p}M$ two totally geodesic subspaces.
	For each $i\in \{1,2\}$, consider the inextendable compatible totally geodesic immersion $f_{i}\colon \Sigma_{i}\to M$ satisfying $V_{i}\in \tilde{f}_{i}(\Sigma_{i})$ and let $x_{i}\in\Sigma_{i}$ be the unique point such that $\tilde{f}(x_{i})=V_{i}$.
	Then the following assertions are equivalent:
	\begin{enumerate}[\rm (i)]
		\item\label{propItem:SubspaceInclusion} $V_{1}\subseteq V_{2}$.
		\item\label{propItem:ImmersionInclusion} There exists a connected Riemannian manifold $E$, a surjective local isometry $\pi\colon E\to \Sigma_{1}$, an inextendable compatible totally geodesic immersion $h\colon E\to \Sigma_{2}$ and a point $z\in E$ such that $f_{1}\circ \pi=f_{2}\circ h$, $\pi(z)=x_{1}$ and $h(z)=x_{2}$.
		In other words, the following diagram commutes:
		\[
			\begin{tikzcd}
				E \arrow[r, "h"] \arrow[d, "\pi"] & \Sigma_{2} \arrow[d, "f_{2}"] \\
				\Sigma_{1} \arrow[r, "f_{1}"]     & M                            
			\end{tikzcd}
		\]
	\end{enumerate}
	
	Furthermore, if $\Sigma_{2}\subseteq M$ is injectively immersed and $f_{2}=\iota \colon \Sigma_{2}\hookrightarrow M$, one can take $E=\Sigma_{1}$, $\pi=\operatorname{Id}_{\Sigma_{1}}$, $h\colon \Sigma_{1}\to\Sigma_{2}$ given by $h(x)=f_{1}(x)$ and $z=x_{1}$, so $V_{1}\subseteq V_{2}$ if and only $f_{1}(\Sigma_{1})\subseteq \Sigma_{2}$.
\end{proposition}

\begin{proof}
	Start by assuming~(\ref{propItem:ImmersionInclusion}).
	Then we have
	\[
		\begin{aligned}
			V_{1}={}&\tilde{f_{1}}(x_{1})=\tilde{f_{1}}(\pi(z))=(f_{1})_{*\pi(z)}(T_{\pi(z)}\Sigma_{1})=(f_{1}\circ\pi)_{*z}(T_{z}E)=(f_{2}\circ h)_{*z}(T_{z}E) \\
			={}&(f_{2})_{*x_{2}}(h_{*z}(T_{z}E))\subseteq (f_{2})_{*x_{2}}(T_{x_{2}}\Sigma_{2})=\tilde{f_{2}}(x_{2})=V_{2},
		\end{aligned}
	\]
	which proves~(\ref{propItem:SubspaceInclusion}).
	
	Now, suppose~(\ref{propItem:SubspaceInclusion}) is true, and let $W=(f_{2})_{*x_{2}}^{-1}(V_{1})$, which is a totally geodesic subspace of $T_{x_{2}}\Sigma_{2}$.
	We can construct an inextendable compatible totally geodesic immersion $h\colon E \to \Sigma_{2}$ such that $W\in\tilde{h}(E)$, and there exists a unique $z\in E$ for which $h(z)=x_{2}$ and $\tilde{h}(z)=W$.
	As $M$ and $\Sigma_{2}$ are real analytic, we may apply Proposition~\ref{prop:TGInextendableCharacterization} twice to see that the composition $f_{2}\circ h\colon E\to M$ sends maximal geodesics of $E$ to maximal geodesics of $M$. 
	However, $f_{2}\circ h$ need not be compatible.
	Let $\mathcal{R}$ be the equivalence relation on $E$ defined by \[x\mathcal{R}y:\Leftrightarrow\tilde{f_{2}\circ h}(x)=\tilde{f_{2}\circ h}(y).\]
	The quotient space $E/\mathcal{R}$ admits a unique smooth structure and Riemannian metric such that the natural projection $\rho\colon E\to E/\mathcal{R}$ is a surjective local isometry and the map $g\colon E/\mathcal{R}\to M$ given by $g([x])=f_{2}(h(x))$ is a compatible totally geodesic immersion~\cite[Section~10.3.1]{BerndtConsoleOlmos}.
	Because $f_{2}\circ h$ sends maximal geodesics to maximal geodesics and $\rho$ is a surjective local isometry, the immersion $g$ sends the maximal geodesics of $E/\mathcal{R}$ to maximal geodesics of $M$, so $g$ is inextendable by Proposition~\ref{prop:TGInextendableCharacterization}.
	Observe that
	\[
		\tilde{g}([z])=g_{*[z]}(T_{[z]}E/\mathcal{R})=(g\circ \rho)_{*z}(T_{z}E)=(f_{2}\circ h)_{*z}(T_{z}E)=(f_{2})_{*x_{2}}(W)=V_{1},
	\]
	so by uniqueness of $f_{1}$ there exists a global isometry $\phi\colon E/\mathcal{R}\to \Sigma_{1}$ such that $g=f_{1}\circ\phi$.
	By considering $\pi=\phi\circ \rho\colon E\to \Sigma_{1}$, we obtain the equalities $f_{1}\circ \pi=g\circ\rho=f_{2}\circ h$ and $\pi(z)=x_{1}$ because $\tilde{f_{1}}$ is injective and
	\[
		\tilde{f_{1}}(\pi(z))=(f_{1})_{*\pi(z)}(T_{\pi(z)}\Sigma_{1})=(f_{1}\circ\pi)_{*z}(T_{z}E)=(g\circ\rho)_{*z}(T_{z}E)=V_{1}=\tilde{f_{1}}(x_{1}).
	\]
	Therefore,~(\ref{propItem:ImmersionInclusion}) holds.
	
	Finally, note that if $\Sigma_{2}\subseteq M$ and $f_{2}=\iota$ is the inclusion map, one sees easily that the composition $\iota \circ h$ in the previous paragraph is also a compatible totally geodesic immersion, so $E/\mathcal{R}=E$ and we obtain in this case a global isometry $\phi\colon E\to \Sigma_{1}$ satisfying $\iota\circ h=f_{1}\circ \phi$.
	By replacing $E$ with $\Sigma_{1}$ and $h$ with $h\circ\phi^{-1}$, we obtain $\iota\circ h=f_{1}$, so $h\colon \Sigma_{1}\to\Sigma_{2}$ is simply the restriction in codomain of $f_{1}$, and $f_{1}(\Sigma_{1})\subseteq \Sigma_{2}$. \qedhere
\end{proof}

Motivated by the previous proposition, we say that an inextendable compatible totally geodesic immersion $f\colon \Sigma\to M$ (or simply, $\Sigma$) is \emph{maximal} if it is not a global isometry and whenever we have another inextendable compatible totally geodesic immersion $f'\colon \Sigma'\to M$, a Riemannian manifold $E$, a surjective local isometry $\pi\colon E\to \Sigma$ and a compatible totally geodesic immersion $h\colon E\to \Sigma'$ satisfying $f'\circ h = f \circ \pi$, we have that $f'$ is either a global isometry or equivalent to $f$.
From Proposition~\ref{prop:TGInclusions}, the following conditions are equivalent:
\begin{itemize}
	\item $f\colon\Sigma \to M$ is maximal.
	\item For all $x\in\Sigma$, $\tilde{f}(x)=f_{*x}(T_{x}\Sigma)$ is a maximal totally geodesic subspace of $T_{f(x)}M$.
	\item There exists an $x\in\Sigma$ such that $\tilde{f}(x)$ is a maximal totally geodesic subspace of $T_{f(x)}M$.
\end{itemize}

\section{Totally geodesic submanifolds in naturally reductive homogeneous spaces}
\label{sec:tgnatred}
In this section, we introduce new techniques for studying totally geodesic submanifolds in naturally reductive homogeneous spaces.
Because homogeneous spaces are complete and real analytic, by Corollary~\ref{cor:CompleteTGSubmanifolds} their inextendable totally geodesic submanifolds are complete.
These submanifolds are also homogeneous as Riemannian manifolds as a consequence of~\cite[Theorem~8.9]{KobayashiNomizu2}, but they need not be extrinsically homogeneous.

We denote by $M=\mathsf{G}/\mathsf{K}$ a naturally reductive homogeneous space endowed with a reductive decomposition $\g{g}=\g{k}\oplus \g{p}$.
Since $M$ is homogeneous, we may only consider totally geodesic submanifolds passing through $o=e\s{K}$, which is equivalent to studying totally geodesic subspaces of $T_{o}M\equiv \g{p}$.
In this setting, we have the following characterization of these subspaces due to Tojo:

\begin{theorem}[Tojo's criterion \cite{TojoTotallyGeodesic}]\label{thm:TojoCriterion}
	Let $M=\mathsf{G}/\mathsf{K}$ be a naturally reductive homogeneous space with reductive decomposition $\g{g}=\g{k}\oplus\g{p}$.
	Assume $\g{v}\subseteq \g{p}$ is a vector subspace and consider for each $X\in \g{v}$ the operator $D_{X}\colon \g{p}\to \g{p}$.
	Then, the following conditions are equivalent:
	\begin{enumerate}[\rm (i)]
		\item\label{thItem:TojoExistence} There exists a totally geodesic submanifold $\Sigma$ of $M$ passing through $o$ with tangent space~$\g{v}$.
		\item\label{thItem:TojoHalfInvariance} For each $X\in \g{v}$, we have $R(X,e^{-D_{X}}\g{v})e^{-D_{X}}\g{v}\subseteq e^{-D_{X}}\g{v}$.
		\item\label{thItem:TojoFullInvariance} For each $X\in \g{v}$, the subspace $e^{-D_{X}}\g{v}$ is $R$-invariant.
	\end{enumerate}
\end{theorem}

We now give a geometric interpretation of the subspace $e^{-t D_{X}}\g{v}$.
Consider the geodesic $\gamma(t)=\Exp(t X)\cdot o$ with initial condition $X\in \g{v}$.
Then there are two vector space isomorphisms that we can establish between $\g{p}=T_{o}M$ and $T_{\gamma(t)}M$: parallel translation $\mathcal{P}_{0,t}^{\gamma}\colon T_{o}M\to T_{\gamma(t)}M$ and the pushforward of the flow of $X^{*}$, given by $\Exp(tX)_{*o}\colon T_{o}M\to T_{\gamma(t)}M$.
Both maps are related by \cite[Equation~(2.2.1)]{DO-nullity}
\begin{equation}\label{eq:OlmosFormula}
	\Exp(tX)^{-1}_{*o}\circ \mathcal{P}_{0,t}^{\gamma}=e^{-t D_{X}}.
\end{equation}
Suppose that $\g{v}$ is totally geodesic. We consider the complete totally geodesic immersion $f\colon \Sigma\to M$ such that $\g{v}=\tilde{f}(p)$ for some point $p\in \Sigma$ and take $v=(f_{*p})^{-1}(X)\in T_{p}\Sigma$, $g=\Exp(-t X)$.
From~\eqref{eq:OlmosFormula} and the fact that $f$ commutes with parallel translations we see that $e^{-tD_X}\g{v}=\tilde{g\circ f}(\exp_{p}(tv))$, yielding the following result:
\begin{corollary}[{\cite[Proposition 3.5]{TojoTotallyGeodesic}}]\label{cor:TojoCongruentSubmanifolds}
	If $\g{v}\subseteq \g{p}$ is a totally geodesic subspace, then for every $X\in \g{v}$ the subspace $e^{-D_X}\g{v}$ is also totally geodesic, and the corresponding totally geodesic submanifolds are congruent.
\end{corollary}

\subsection{Totally geodesic submanifolds invariant under $D$}\label{subsec:Dinvtg}

We now study a particular class of totally geodesic submanifolds of $M=\s{G}/\s{K}$.
Consider the canonical connection $\nabla^{c}$ associated with the reductive decomposition $\g{g}=\g{k}\oplus\g{p}$ and the difference tensor $D=\nabla-\nabla^{c}$.
We say that an immersion $f\colon \Sigma\to M$ is $D$-\textit{invariant} (and $\Sigma$ is a $D$-\textit{invariant} submanifold) if for every $x\in\Sigma$ the subspace $\tilde{f}(x)=f_{*x}(T_x\Sigma)\subseteq T_{f(x)}M$ is invariant under $D$.
It is immediate that a $D$-invariant submanifold is totally geodesic if and only if for every $X$, $Y\in\g{X}(\Sigma)$ the covariant derivative $\nabla^{c}_{X}Y$ remains tangent to $\Sigma$.

These submanifolds are related to certain subalgebras of $\g{g}$.
We say that a Lie subalgebra $\g{s}$ is \textit{canonically embedded} in $\g{g}$ if it splits with respect to the reductive decomposition, that is,
\[
	\g{s}=(\g{s}\cap \g{k})\oplus (\g{s}\cap \g{p})=\g{s}_\g{k}\oplus \g{s}_{\g{p}}.
\]

The following theorem gives an algebraic characterization of $D$-invariant totally geodesic submanifolds passing through the origin.
Furthermore, it gives an explicit method to construct them from their tangent space at $o$.
The proof can be obtained by combining the theorem in~\cite[p.~11]{Sagle} and the first result in \cite[\S 2]{HermannTotallyGeodesicOrbits}.
However, we include it for the sake of completeness.
\begin{theorem}\label{th:CharacterizationDInvariantTGSubmanifolds}
	Let $M=\s{G}/\s{K}$ be a naturally reductive homogeneous space with reductive decomposition $\g{g}=\g{k}\oplus\g{p}$ and $\g{v}\subseteq \g{p}$ a vector subspace.
	The following conditions are equivalent:
	\begin{enumerate}[\rm (i)]
		\item\label{thItem:RDInvariant} The subspace $\g{v}$ is invariant under the tensors $R$ and $D$.
		\item\label{thItem:RcDInvariant} The subspace $\g{v}$ is invariant under the tensors $R^c$ and $D$.
		\item\label{thItem:CanonicallyEmbeddedOrbit} There exists a connected Lie subgroup $\s{S}\subseteq \s{G}$ such that its Lie algebra $\g{s}$ is canonically embedded in $\g{g}$ and the tangent space to the orbit $\s{S}\cdot o$ at $o$ is $\g{v}$.
		\item\label{thItem:DInvariantTGSubmanifold} There exists a connected, injectively immersed, and complete $D$-invariant totally geodesic submanifold $\Sigma$ such that $o\in \Sigma$ and $T_{o}\Sigma=\g{v}$.
	\end{enumerate}
	Furthermore, if any of the four previous conditions hold, we have:
	\begin{enumerate}[\rm(1)]
		\item\label{thItem:CanonicallyEmbeddedSubalgebraConstruction} a Lie subgroup satisfying the conditions of item~\eqref{thItem:CanonicallyEmbeddedOrbit} is the connected subgroup $\s{S}$ with Lie algebra
		\[
			\g{s}=[\g{v},\g{v}]+\g{v}=[\g{v},\g{v}]_{\g{k}}\oplus\g{v},
		\]
		\item\label{thItem:TGOrbitConstruction} the totally geodesic submanifold $\Sigma$ passing through $o$ with tangent space $\g{v}$ is $\Sigma=\s{S}\cdot o$.
	\end{enumerate}
\end{theorem}
\begin{proof}
	Firstly, note that the formula
	\[
	R^{c}(X,Y)Z={}-[[X,Y]_{\g{k}},Z]=R(X,Y)Z-D_X D_Y Z+D_Y D_X Z+2D_{D_X Y}Z
	\]
	implies that any $D$-invariant subspace of $\g{p}$ is invariant under $R$ if and only if it is invariant under $R^{c}$, so (\ref{thItem:RDInvariant}) and (\ref{thItem:RcDInvariant}) are equivalent.
	
	Now, suppose that $\g{v}$ satisfies (\ref{thItem:RcDInvariant}).	
	We prove that $\g{s}=[\g{v},\g{v}]+\g{v}$ is a Lie subalgebra of $\g{v}$.
	This amounts to checking that $[\g{v},[\g{v},\g{v}]]$ and $[[\g{v},\g{v}],[\g{v},\g{v}]]$ are contained in $\g{s}$.
	Let $X$, $Y$, $Z\in \g{v}$.
	Then we have
	\[
	\begin{aligned}
		[[X,Y],Z]=&{}[[X,Y]_{\g{k}},Z]+[[X,Y]_{\g{p}},Z]=-R^{c}(X,Y)Z+2[D_X Y,Z]\in \g{s}.
	\end{aligned}
	\]
	In particular, $[\g{v},\g{s}]\subseteq \g{s}$.
	Similarly, by the Jacobi identity, we see that
	\[
	\begin{aligned}
		[[\g{v},\g{v}],[\g{v},\g{v}]]\subseteq[\g{v},[\g{v},[\g{v},\g{v}]]]\subseteq [\g{v},\g{s}]\subseteq \g{s}.
	\end{aligned}
	\]
	Therefore, $\g{s}$ is a Lie subalgebra.
	Because $\g{v}$ is $D$-invariant, we see that $\g{s}_{\g{p}}=\g{v}\subseteq \g{s}$, and from this inclusion it follows that $\g{s}=\g{s}_{\g{k}}\oplus\g{v}$, which proves that $\g{s}$ is canonically embedded.
	It is also immediate from the description of $\g{s}$ that $\g{s}_{\g{k}}=[\g{v},\g{v}]_{\g{k}}$.
	As a consequence, if we consider the Lie subgroup $\s{S}$ of $\s{G}$ whose Lie algebra is $\g{s}$, then the tangent space $T_{o}(\s{S}\cdot o)$ coincides with $\g{s}_{\g{p}}=\g{v}$, and therefore (\ref{thItem:CanonicallyEmbeddedOrbit}) holds.
	
	We now prove that (\ref{thItem:CanonicallyEmbeddedOrbit}) implies (\ref{thItem:DInvariantTGSubmanifold}).
	Assume $\s{S}\subseteq \s{G}$ is a Lie subgroup whose Lie algebra is canonically embedded in $\g{g}$ and $T_{o}(\s{S}\cdot o)=\g{s}_{\g{p}}=\g{v}$.
	It is clear from \eqref{eq:SecondFundamentalFormOrbit} that $\mathbb{II}$ is zero at $o$.
	Because $\s{S}\cdot o$ is extrinsically homogeneous, this implies that the second fundamental form vanishes everywhere, and thus $\s{S}\cdot o$ is totally geodesic.
	We show that $\s{S}\cdot o$ is $D$-invariant, which by $\s{G}$-invariance of $D$ is equivalent to checking that $\g{v}$ is $D$-invariant.
	Given $X$, $Y\in \g{s}_{\g{p}}\subseteq \g{s}$, we have $D_{X}Y=(1/2)[X,Y]_{\g{p}}\in \g{s}_{\g{p}}$, so the claim follows.
	
	Finally, it is immediate that (\ref{thItem:DInvariantTGSubmanifold}) implies (\ref{thItem:RDInvariant}) from the definition of $D$-invariant submanifolds and Theorem~\ref{th:CharacterizationTGSubmanifolds}. \qedhere
\end{proof}

\begin{corollary}
	Every complete $D$-invariant totally geodesic submanifold of a naturally reductive homogeneous space $M=\s{G}/\s{K}$ is extrinsically homogeneous with respect to the given presentation of $M$, and thus an injectively immersed submanifold.
\end{corollary}

\begin{remark}
	Theorem~\ref{th:CharacterizationDInvariantTGSubmanifolds} is also a refinement of \cite[Lemma~3.1]{RodriguezVazquezOlmos}, which states that for a general reductive homogeneous space a subspace $\g{v}\subseteq \g{p}$ invariant under $R$ and $D$ is tangent to a complete totally geodesic submanifold.  Also, notice that the class of $D$-invariant totally geodesic submanifolds  includes all totally geodesic submanifolds of symmetric spaces, since in a symmetric space $D=0$.
    Thus, in the symmetric setting the subspaces $\g{v}$ appearing in the previous theorem are the Lie triple systems~(see~\cite[\S 11.1]{BerndtConsoleOlmos}).
\end{remark}

\begin{remark}
	The case $\g{v}=\g{p}$ in Theorem~\ref{th:CharacterizationDInvariantTGSubmanifolds} is part of a result by Kostant~\cite{Kostant}, which implies in particular that the connected (normal) subgroup with Lie algebra $[\g{p},\g{p}]+\g{p}$ acts transitively on $M$. 
\end{remark}

\begin{remark}
It is worth noting that in the irreducible setting the conditions of Theorem~\ref{th:CharacterizationDInvariantTGSubmanifolds} do not depend on the naturally reductive decomposition that we choose.
Indeed, from \cite[Theorem~2.1]{OlmosReggiani}, we see that if $M=\s{G}/\s{K}$ is a simply connected irreducible naturally reductive space which is not symmetric, then the canonical connection $\nabla^{c}$ is unique.
Therefore, given any naturally reductive decomposition $\g{g}=\g{k}\oplus\g{p}$ of~$\g{g}$, the subspaces $\g{v}\subseteq \g{p}$ that are invariant under $R$ and $D$ correspond under the identification $\g{p}\equiv T_{o}M$ to the subspaces $V\subseteq T_{o}M$ that are invariant under $R$ and $\nabla-\nabla^{c}$, and the uniqueness of the canonical connection implies that these subspaces are always the same regardless of the decomposition. 
Similarly, if $M$ is a nearly Kähler 3-symmetric space and one restricts their attention to the reductive decompositions invariant under the automorphism of order three, then all of their associated canonical connections coincide by~\cite[Lemma~3.1]{TojoLagrangian}, and the same argument applies.
\end{remark}

\subsection{Totally geodesic surfaces}\label{subsec:tgsurf}

As an application of Corollary~\ref{cor:TojoCongruentSubmanifolds}, we derive a necessary condition for the existence of totally geodesic surfaces with a given tangent plane.

Let $\g{v}\subseteq \g{p}$ be a $2$-dimensional  subspace, and assume that $\g{v}$ is the tangent plane at $o$ of a complete totally geodesic surface $\Sigma$ of $M$.
Fix a nonzero element $X\in \g{v}$ and choose any $Y\in \g{v}\setminus \{0\}$ that is orthogonal to $X$, so that $\{X,Y\}$ is an orthogonal basis of $\g{v}$.
Since $\Sigma$ is homogeneous and two-dimensional, it follows that $\Sigma$ is a space of constant curvature $\kappa\in \R$, and the same can be said for the totally geodesic submanifold $\Sigma_{t}$ associated with $e^{-tD_{X}}\g{v}$ for all $t\in\R$.
This implies in particular that the restriction of $\nabla R$ to the tangent space of $\Sigma$ and $\Sigma_{t}$ at any of their points is the zero tensor.
Furthermore, since $D_{X}X=0$ due to the skew-symmetry of $D$, we have $e^{-t D_{X}}\g{v}=\vecspan\{X,e^{-t D_{X}}Y\}$.
Because $\Sigma_{t}$ has curvature $\kappa$, it follows that $e^{-t  D_{X}}Y$ is an eigenvector of the Jacobi operator $R_{X}$ with eigenvalue $\kappa \lvert X \rvert^{2}$, as well as an element of $\ker C_{X}$.
One can argue similarly with the so-called \textit{Cartan operators of order} $j$ given by $C^j_XY=\nabla^jR(X,\dots,X,Y,X)$, because they vanish identically on $\g{v}$.
Since the subspace of $\g{p}$ generated by the curve $e^{-t D _{X}} Y$ is the span of all vectors of the form $D_{X}^{k}Y$ with $k\geq 0$, we have obtained the following:

\begin{proposition}\label{prop:TojoSurfaces}
	Let $M=\mathsf{G}/\mathsf{K}$ be a naturally reductive homogeneous space with reductive decomposition $\g{g}=\g{k}\oplus \g{p}$.
	Choose orthogonal vectors $X$, $Y\in \g{p}$ and suppose that $\g{v}=\vecspan\{X,Y\}$ is the tangent space of a totally geodesic surface $\Sigma$ of $M$.
	Then we have the inclusion
	\[
	\spann \{ D_{X}^{k}Y\colon k\geq 0 \}\subseteq \ker (R_{X}-\kappa\lvert X \rvert^{2}\operatorname{Id}_{\g{p}})\cap\bigcap_{j= 1}^{\infty}\ker C_{X}^{j}.
	\]
\end{proposition}

\subsection{Well-positioned totally geodesic submanifolds and homogeneous fibrations}\label{subsec:wellpostg}

We now study the case that $M=\s{G}/\s{H}$ is also the total space of the homogeneous fibration induced by the inclusions $\s{H}\subseteq \s{K}\subseteq \s{G}$ (observe the change of notation).
Let $B=\s{G}/\s{K}$ be the base space and $F=\s{K}/\s{H}$ be the fiber of the given submersion.
Consider a totally geodesic immersion $f\colon\Sigma\to M$.
We say that $\Sigma$ is \textit{well-positioned} at $p\in \Sigma$ (with respect to the fibration $F \to M\to B$) if $\tilde{f}(p)=\left(\tilde{f}(p) \cap \mathcal{V}_{f(p)}\right)\oplus\left(\tilde{f}(p) \cap \mathcal{H}_{f(p)}\right)$.
Furthermore, $\Sigma$ is said to be \textit{well-positioned} if it is well-positioned at every point $p\in \Sigma$.
The next result allows us to give an algebraic characterization for a totally geodesic submanifold to be well-positioned.

\begin{lemma}\label{lemma:HomogeneousFibrationWellPositioned}
	Let $F\to M \to B$ be the homogeneous fibration induced by the chain of inclusions $\s{H}\subseteq \s{K}\subseteq \s{G}$, where $\s{H}$, $\s{K}$ and $\s{G}$ are compact and the Riemannian metrics on $F$, $M$ and $B$ are induced by a bi-invariant metric on $\s{G}$.
	Let $f\colon\Sigma\to M$ be a complete totally geodesic immersion passing through the point $o$ with tangent space $\g{v}$.
	Then, the following conditions are equivalent:
	\begin{enumerate}[\rm (i)]
		\item $\Sigma$ is well-positioned with respect to the submersion $M\to B$,
		\item for all $X\in \g{v}$, the subspace $e^{-D_{X}}\g{v}$ splits with respect to the decomposition $\g{p}=\mathcal{V}_{o}\oplus\mathcal{H}_{o}$.
	\end{enumerate}		
\end{lemma}
\begin{proof}
	Let $p\in\Sigma$ be such that $\tilde{f}(p)=\g{v}$.
	Since we are assuming that $\Sigma$ is connected and complete, every point of $\Sigma$ is of the form $q=\exp_{p}(v)$ for a certain $v\in T_{p}\Sigma$.
	Consider the geodesic $\gamma(t)=f(\exp(tv))=\Exp(tX)\cdot o$ (where $X=f_{*p}(v)$), which connects $p$ and $q$.
	Then, we have $\tilde{f}(q)=\mathcal{P}_{0,1}^{\gamma}\g{v}$.
	As $\mathcal{V}$ and $\mathcal{H}$ are invariant under $\s{G}$, we see that $\mathcal{V}_{f(q)}=\Exp(X)_{*o}\mathcal{V}_{o}$ and $\mathcal{H}_{f(q)}=\Exp(X)_{*o}\mathcal{H}_{o}$.
	Therefore, using~\eqref{eq:OlmosFormula} we see that $\Sigma$ is well-positioned at $q$ if and only if the subspace
	$\Exp(X)_{*o}^{-1}(\tilde{f}(q))=e^{-D_{X}}\g{v}$
	splits with respect to the decomposition $\g{p}=\mathcal{V}_{o}\oplus\mathcal{H}_{o}$.
	As $\g{v}=f_{*p}(T_{p}\Sigma)$, the equivalence follows. \qedhere
\end{proof}

\begin{corollary}\label{cor:CanonicallyEmbeddedWellPositioned}
	Let $F\to M \to B$ be as in Lemma~\ref{lemma:HomogeneousFibrationWellPositioned}, and let $\Sigma\subseteq M$ be a $D$-invariant totally geodesic submanifold passing through $o$.
	Then $\Sigma$ is well-positioned if and only if it is well-positioned at $o$.
\end{corollary}

\begin{proof}
	This follows from noting that the $D$-invariance of $\g{v}$ implies that $e^{-D_{X}}\g{v}=\g{v}$ for all $X\in\g{v}$. \qedhere
\end{proof}

\section{The examples}\label{sec:examples}
	
In this section we will describe the totally geodesic submanifolds of $\C\s{P}^{3}$, $\s{F}(\C^{3})$ and $\s{S}^{3}\times\s{S}^{3}$ that appear in the classification and determine their isometry type.
We indicate if the examples are well-positioned with respect to the homogeneous fibrations given in Subsection~\ref{subsec:homNK}.
	
Let us recall some definitions about special submanifolds of almost Hermitian manifolds.	
If $(M^{2n},J)$ is an almost Hermitian manifold and $f\colon \Sigma\to M$ is an immersion, we say that $f$ (and $\Sigma$) is \textit{totally real} if for all $p\in \Sigma$ the subspaces $f_{*p}(T_p\Sigma)$ and $Jf_{*p}(T_p\Sigma)$ of $T_{f(p)}M$ are orthogonal.
If we also have $T_{p}M=f_{*p}(T_p\Sigma)\oplus Jf_{*p}(T_p\Sigma)$ (that is, if $\dim\Sigma = n$), then $\Sigma$ is a \textit{Lagrangian} submanifold.
Separately, $f$ (and $\Sigma$) is \textit{almost complex} (or $J$\textit{-holomorphic}) if $f_{*p}(T_p\Sigma)$ is invariant under $J$ for all $p\in\Sigma$.
Furthermore, if $\Sigma$ is a surface, we will refer to it as an \textit{almost complex surface} or a $J$\textit{-holomorphic curve}.
\begin{remark}
	Many of the totally geodesic submanifolds that appear in this section are isometric to a sphere with a round or complex Berger metric.
	We can compute the radius $r$ of the sphere $\s{S}^{n}(r)$, as well as the parameters of the Berger sphere $\s{S}^{3}_{\C,\tau}(r)$ from its sectional curvature. 
	Indeed, it is well known that the sectional curvature of $\s{S}^{n}(r)$ is equal to $1/r^{2}$.
	In the case of $\s{S}^{3}_{\C,\tau}(r)$, the parameters $r$ and $\tau$ can be obtained from the equations $\tau=r^{2}\sec(U,X)$ and
	$4-3\tau=r^{2}\sec(X,Y)$,
	where $U$ is a vertical vector and $X$, $Y$ are horizontal vectors with respect to the Hopf fibration (see~\cite{GorodskiKollrossRV}).
\end{remark}
		
\subsection{The complex projective space}\label{subsec:excp3}
We describe the totally geodesic examples of the nearly Kähler complex projective space $\C\s{P}^3$.
\subsubsection{The real projective space \cite[Example~3.9]{Aslan}}
	
Consider the subgroup $\mathsf{U}(2)^{j}\subseteq \mathsf{Sp}(2)$ whose Lie algebra is given by
\begin{equation*}
	\g{u}(2)^{j}=\spann\bigl\{
		jE_{11},jE_{22},E_{21}-E_{12},j(E_{12}+E_{21})
	\bigr\}=(\g{u}(2)^{j}\cap\g{k})\oplus(\g{u}(2)^{j}\cap \g{p}).
\end{equation*}
Then, $\g{u}(2)^{j}$ is canonically embedded in $\g{sp}(2)$, so the orbit $\mathsf{U}(2)^{j}\cdot o$ is a totally geodesic submanifold of $\C\mathsf{P}^{3}$ whose tangent space is $\vecspan\{e_{1},e_{3},e_{5}\}$.
The isotropy subgroup $\mathsf{U}(2)^{j}\cdot o$ is equal to $\mathbb{Z}_{2}\times \mathsf{U}(1)$, so $\mathsf{U}(2)^{j}\cdot o$ is diffeomorphic to a real projective space $\R \mathsf{P}^{3}$.
The induced metric is Berger-like.
Indeed, this totally geodesic submanifold is isometric to $\R\s{P}^3_{\C,1/2}(2)$.
Let us write $\g{p}_{\R\mathsf{P}^{3}_{\C,1/2}(2)}=\vecspan\{e_1,e_3,e_5\}$.
A computation gives $J(\g{p}_{\R\mathsf{P}^{3}_{\C,1/2}(2)})=\g{p}\ominus \g{p}_{\R\mathsf{P}^{3}_{\C,1/2}(2)}$, and since $\R\mathsf{P}^{3}_{\C,1/2}(2)$ is extrinsically homogeneous we see that $\R\mathsf{P}^{3}_{\C,1/2}(2)$ is a Lagrangian submanifold.
Finally, note that $\R\s{P}^{3}_{\C,1/2}(2)$ is well-positioned at $o$, so by Corollary~\ref{cor:CanonicallyEmbeddedWellPositioned}, it is well-positioned.
	
\subsubsection{The fiber of the twistor fibration}
	
Recall that the fibers of the twistor fibration are totally geodesic surfaces in $\C\mathsf{P}^{3}$.
In particular, the orbit through $o$ is $\left(\mathsf{Sp}(1)\times\mathsf{Sp}(1)\right)\cdot o=\mathsf{Sp}(1)_{f}\cdot o$, where $\mathsf{Sp}(1)_{f}$ denotes the image of the standard embedding of $\s{Sp}(1)$ in $\s{Sp}(2)$ in the first block.
The isotropy subgroup $(\mathsf{Sp}(1)_{f})_{o}$ coincides with $\mathsf{U}(1)$, so $\mathsf{Sp}(1)_{f}\cdot o$ is diffeomorphic to a sphere.
Its tangent space at $o$ is $\g{p}_{\mathsf{Sp}(1)_{f}\cdot o}=\g{p}_{1}$.
The sectional curvature in this case is $\operatorname{sec}(\g{p}_{1})=2$, so $\mathsf{Sp}(1)_{f}\cdot o$ is a round sphere of radius $1/\sqrt{2}$.
Furthermore, the fact that $J(\g{p}_{1})=\g{p}_{1}$ implies that $\mathsf{Sp}(1)_{f}\cdot o$ is an almost complex surface in $\C\mathsf{P}^{3}$.
By definition, $\s{Sp}(1)_{f}\cdot o$ is well-positioned.
	
\subsubsection{The horizontal sphere $\mathsf{SU}(2)\cdot o$}
	
Consider the standard embedding of $\mathsf{SU}(2)$ in $\mathsf{Sp}(2)$.
Since the Lie algebra $\g{su}(2)$ is given by
\begin{equation*}
	\g{su}(2)=\vecspan\bigl\{
		i(E_{11}-E_{22}),E_{21}-E_{12},i(E_{12}+E_{21})
	\bigr\}
	=(\g{su}(2)\cap\g{k})\oplus(\g{su}(2)\cap \g{p}),
\end{equation*}
it follows that $\g{su}(2)$ is canonically embedded, and the orbit $\mathsf{SU}(2)\cdot o$ is a totally geodesic submanifold of $\C\mathsf{P}^{3}$ with tangent space $\g{p}_{\mathsf{SU}(2)\cdot o}=\vecspan\{e_{3},e_{4}\}$.
The isotropy subgroup of $\mathsf{SU}(2)$ at $o$ is the canonical $\mathsf{U}(1)$, so $\mathsf{SU}(2)\cdot o$ is diffeomorphic to a sphere.
Furthermore, its sectional curvature is given by $\operatorname{sec}(\g{p}_{\mathsf{SU}(2)\cdot o})=1$, so this submanifold is a round sphere of radius $1$.
Finally, note that $\g{p}_{\mathsf{SU}(2)\cdot o}$ is $J$-invariant, so this sphere is also an almost complex surface in $\C\mathsf{P}^{3}$.
Note that $\s{SU}(2)\cdot o$ is well-positioned by Corollary~\ref{cor:CanonicallyEmbeddedWellPositioned}.
Indeed, its tangent space at every point is always contained in the horizontal subspace of the twistor fibration.
	
\subsubsection{The sphere $\mathsf{SU}(2)_{\Lambda_{3}}\cdot o$}
	
Consider the unique complex irreducible representation $\Lambda_{3}$ of $\mathsf{SU}(2)$ of dimension four.
Since this representation is unitary and of symplectic type, it restricts to a homomorphism $\mathsf{SU}(2)\to\mathsf{Sp}(2)$.
To get an explicit description of this map at the Lie algebra level (which is enough for our purposes), it suffices to see that the linear map $\g{su}(2)\to \g{sp}(2)$ defined via
\begin{equation*}
	\begin{aligned}
		H&\mapsto i(E_{11}+3E_{22}), & E&\mapsto\sqrt{3}(E_{21}-E_{12})+2jE_{11} ,& F&\mapsto -2kE_{11}-i\sqrt{3}(E_{12}+E_{21}),
	\end{aligned}
\end{equation*}
is a Lie algebra homomorphism which is also irreducible as a representation, so by uniqueness it must be equal to $\Lambda_3$.

We denote by $\mathsf{SU}(2)_{\Lambda_{3}}$ the image of the previous homomorphism.
The Lie algebra of this group satisfies
\begin{equation*}
	\g{su}(2)_{\Lambda_{3}}=\vecspan\Bigl\{
		i(E_{11}+3E_{22}),
		\sqrt{2}e_{1}+\sqrt{3}e_{3},
		\sqrt{2}e_{2}+\sqrt{3}e_{4}
	\Bigr\}=(\g{su}(2)_{\Lambda_{3}}\cap \g{k})\oplus (\g{su}(2)_{\Lambda_{3}}\cap \g{p}),
\end{equation*}
so it is canonically embedded in $\g{sp}(2)$.
As a consequence, the orbit $\mathsf{SU}(2)_{\Lambda_{3}}\cdot o$ is a totally geodesic submanifold of $\C\mathsf{P}^{3}$ with tangent space $\g{p}_{\mathsf{SU}(2)_{\Lambda_{3}}\cdot o}=\vecspan\left\{ \sqrt{2}e_{1}+\sqrt{3}e_{3},\sqrt{2}e_{2}+\sqrt{3}e_{4} \right\}$.
The isotropy subgroup at $o$ is the $\mathsf{U}(1)$ subgroup with Lie algebra generated by $\operatorname{diag}(i,3i)$, so this orbit is actually a sphere.
Since the sectional curvature of $\g{p}_{\mathsf{SU}(2)_{\Lambda_{3}}\cdot o}$ is $1/5$, we see that $\mathsf{SU}(2)_{\Lambda_{3}}\cdot o$ is a sphere of radius $\sqrt{5}$.
One sees that $\g{p}_{\mathsf{SU}(2)_{\Lambda_{3}}\cdot o}$ is $J$-invariant, and by homogeneity it follows that $\mathsf{SU}(2)_{\Lambda_{3}}\cdot o$ is an almost complex submanifold of $\C\mathsf{P}^{3}$.
Clearly, $\mathsf{SU}(2)_{\Lambda_{3}}\cdot o$ is not well-positioned at $o$, so it is not well-positioned.
	
\subsection{The flag manifold}\label{subsec:exflag} We describe the totally geodesic examples of the nearly Kähler flag manifold $\s{F}(\C^3)$.
	
\subsubsection{The real flag manifold $\mathsf{F}(\R^{3})$ \cite[Example~3.1]{Storm}}\label{ex:realflag}
There is a natural embedding of the real flag manifold $\mathsf{F}(\R^{3})$ in $\mathsf{F}(\C^{3})$ which is induced by the usual inclusion of $\R^3$ in $\C^3$.
This submanifold can also be seen as the orbit $\mathsf{SO}(3)\cdot o$ of the standard $\mathsf{SO}(3)\subseteq \mathsf{SU}(3)$, and the corresponding isotropy subgroup is $\mathsf{SO}(3)_{o}=\mathbb{Z}_{2}\oplus \mathbb{Z}_{2}$, so we get $\mathsf{F}(\R^{3})=\mathsf{SO}(3)/(\mathbb{Z}_{2}\oplus\mathbb{Z}_{2})=\mathsf{Sp}(1)/\mathsf{Q}_{8}$, where $\mathsf{Q}_{8}=\{\pm 1, \pm i, \pm j, \pm k\}$.
Observe that $\g{so}(3)$ is canonically embedded in $\g{su}(3)$, since $\g{so}(3)\subseteq \g{p}$. Thus, Theorem~\ref{th:CharacterizationDInvariantTGSubmanifolds} allows us to conclude that $\mathsf{F}(\R^{3})$ is totally geodesic in $M$, and its tangent space is precisely $\g{p}_{\mathsf{F}(\R^{3})}=\g{so}(3)=\vecspan\{e_{1},e_{3},e_{5}\}$.
A direct computation shows that $\mathsf{F}(\R^{3})$ has constant curvature equal to $1/8$.
Furthermore, we have the equality $J (\g{so}(3))=\vecspan\{e_{2},e_{4},e_{6}\}$, implying that the inclusion $\mathsf{F}(\R^{3})\subseteq \mathsf{F}(\C^{3})$ is Lagrangian.
Finally, note that $\s{F}(\R^{3})$ is well-positioned at $o$, so $\s{F}(\R^{3})$ is well-positioned by Corollary~\ref{cor:CanonicallyEmbeddedWellPositioned}.
	
\subsubsection{The Berger sphere \cite[Example~3.2]{Storm}}\label{ex:bergerflag}
Let $\mathsf{SU}(2)_{(1,0,1)}$ denote the subgroup of $\mathsf{SU}(3)$ that fixes $(1,0,1)\in \C^{3}$.
This subgroup is conjugate to the standard $\mathsf{SU}(2)$ inside $\mathsf{SU}(3)$.
The Lie algebra $\g{su}(2)_{(1,0,1)}$ is the set of all $X\in \g{su}(3)$ such that $X(1,0,1)=0$, and its projection to $\g{p}$ is
spanned by $\{e_{1}+e_{3},e_{2}-e_{4},e_{6}\}$.
It is easy to check that the isotropy subgroup of $\mathsf{SU}(2)_{(1,0,1)}$ at $o$ is trivial, so the corresponding orbit $\mathsf{SU}(2)_{(1,0,1)}\cdot o$ is diffeomorphic to a $3$-sphere.
A direct application of~\eqref{eq:SecondFundamentalFormOrbit} yields that $\mathsf{SU}(2)_{(1,0,1)}\cdot o$ is a totally geodesic submanifold of $\mathsf{F}(\C^{3})$ isometric to $\mathsf{S}^{3}_{\C,1/4}(\sqrt{2})$ and whose tangent space is given by $\g{p}_{\mathsf{S}^{3}_{\C,1/4}(\sqrt{2})}=\vecspan\{e_{1}+e_{3},e_{2}-e_{4},e_{6}\}$.
The subspace $\g{p}_{\mathsf{S}^{3}_{\C,1/4}(\sqrt{2})}$ is also invariant under $D$.
However, the Lie algebra $\g{su}(2)_{(1,0,1)}$ is not canonically embedded in $\g{su}(3)$.
In this case, the connected subgroup given by Theorem~\ref{th:CharacterizationDInvariantTGSubmanifolds} is actually the subgroup	$\s{U}(2)_{(1,0,1)}$ that fixes the complex line generated by $(1,0,1)$.
A direct calculation shows that $J(\g{p}_{\mathsf{S}^{3}_{\C,1/4}(\sqrt{2})})=\g{p}\ominus \g{p}_{\mathsf{S}^{3}_{\C,1/4}(\sqrt{2})}$, so the Berger sphere is Lagrangian.
Note from the expression of $\g{p}_{\mathsf{S}^{3}_{\C,1/4}(\sqrt{2})}$ that $\mathsf{S}^{3}_{\C,1/4}(\sqrt{2})$ is not well-positioned.
	
\subsubsection{The torus \cite[Example~3.3]{CwiklinskiVrancken}}
Consider the maximal torus $\mathsf{H}\subseteq\mathsf{SU}(3)$ whose Lie algebra is given by $\g{h}=\vecspan\{e_{1}+e_{3}+e_{5},e_{2}+e_{4}-e_{6}\}$.
Observe that $\g{h}\subseteq \g{p}$, so it is a canonically embedded subalgebra of $\g{g}$, and the orbit $\mathsf{H}\cdot o$ is a totally geodesic surface.
Since $\mathsf{H}\cdot o$ is a quotient of $\mathsf{H}$ by a finite group, it is a compact abelian Lie group itself and hence diffeomorphic to a torus.
However, we see that $\mathsf{H}\cdot o$ is not isometric to the standard flat torus.
Indeed, in order to determine the isometry type of $\mathsf{H}\cdot o$, we compute the preimage $\exp_{o}^{-1}(o)$.
For this, we need a description of the Riemannian exponential map $\exp_{o}\colon \g{h}\to \mathsf{H}\cdot o\subseteq M$, which is merely the restriction of the Riemannian exponential map of $M$.
To this end, we define the following orthonormal vectors $X=\frac{1}{\sqrt{3}}(e_{1}+e_{3}+e_{5})$, $Y=\frac{1}{\sqrt{3}}(e_{2}+e_{4}-e_{6})$ of $\g{h}$.
Then the exponential map of $\g{h}$ satisfies $\exp_{o}(u X + v Y)=e^{u X}e^{v Y} \cdot o$, and this element is equal to $o$ if and only if $e^{uX}e^{vY}$ is a diagonal matrix.
The solutions to $\exp_{o}(u X + v Y)=o$ are given by the lattice $\Lambda=\vecspan_{\mathbb{Z}}\left\{ \left( \sqrt{2}\pi ,\sqrt{2}\pi/\sqrt{3} \right),\left( 0,2\sqrt{2}\pi/\sqrt{3} \right) \right\}$.
Since $\exp_{o}$ is $\g{h}$-equivariant, in the sense that it satisfies the equation
$
\exp_{o}(T+S)=\Exp(T)\cdot \exp_{o}(S),
$
it follows that $\exp_{o}$ is actually a Riemannian covering map, so $\mathsf{H}\cdot o$ is isometric to the quotient $\R^{2}/\Lambda$.
We will refer to this orbit as $\mathsf{T}^2_{\Lambda}=\mathsf{H}\cdot o=\R^{2}/\Lambda$.
Note that $\mathsf{H}\cdot o$ is not a product $\mathsf{S}^{1}(r_{1})\times \mathsf{S}^{1}(r_{2})$, since the closest points in $\Lambda\setminus\{(0,0)\}$ to the origin are those in
$
\left\{
	\left( \pm\sqrt{2}\pi ,\pm\sqrt{6}\pi/3 \right),
	\left( \pm\sqrt{2}\pi ,\mp\sqrt{6}\pi/3 \right),
	\left( 0,\pm2\sqrt{6}\pi/3 \right)
\right\},
$
and thus there exist three different closed geodesics of minimum length passing through $o$, as opposed to two in the case of $\mathsf{S}^{1}(r)\times\mathsf{S}^{1}(r)$ or one in the case of $\mathsf{S}^{1}(r_{1})\times\mathsf{S}^{1}(r_{2})$ with $r_{1}\neq r_{2}$.
Let $\g{p}_{\mathsf{T}^2_{\Lambda}}=\g{h}=\vecspan\{e_{1}+e_{3}+e_{5},e_{2}+e_{4}-e_{6}\}$
be the tangent space of this surface, then $J(\g{p}_{\mathsf{T}^2_{\Lambda}})=\g{p}_{\mathsf{T}^2_{\Lambda}}$, and by homogeneity it follows that $\mathsf{T}^2_{\Lambda}$ is an almost complex surface in $\mathsf{F}(\C^{3})$.
However, it is clear from the expression of $\g{p}_{\s{T}^2_{\Lambda}}$ that $\s{T}^2_{\Lambda}$ is not well-positioned.
	
\subsubsection{The fiber of the submersion $\mathsf{F}(\C^{3})\to \C\mathsf{P}^{2}$ \cite[Example~3.1]{CwiklinskiVrancken}}
Recall that the fibers of the Riemannian submersion $\mathsf{F}(\C^{3})\to\C\mathsf{P}^{2}$ are totally geodesic.
The fiber through $o$ is $\C\mathsf{P}^{1}=\mathsf{U}(2)\cdot o=\mathsf{SU}(2)\cdot o$, where the isotropy subgroup of $\mathsf{SU}(2)$ at $o$ is $\mathsf{U}(1)$.
The tangent space, as said before, is $\g{p}_{\C\mathsf{P}^{1}}=\g{p}_{1}$.
Since the sectional curvature of $\g{p}_{1}$ is $2$, it follows that $\mathsf{SU}(2)\cdot o$ is isometric to the round sphere of radius $1/\sqrt{2}$.
Furthermore, $J(\g{p}_{1})=\g{p}_{1}$, so $\mathsf{SU}(2)\cdot o$ is an almost complex surface in $\mathsf{F}(\C^{3})$.
Clearly, $\C\s{P}^{1}$ is well-positioned as it is a fiber itself.
	
\subsubsection{The sphere \cite[Example~3.2]{CwiklinskiVrancken}}
Consider $E=\vecspan\{(0,1,0),(1,0,-1),(i,0,i)\}$, which is a real form of $\C^3$, and let $\sigma\colon \C^{3}\to \C^{3}$ be the associated real structure.
Then the normalizer
\begin{equation*}
	\mathsf{SO}(3)^{\sigma}=\{g\in \mathsf{SU}(3)\colon g(E)=E\}=\{g\in\mathsf{SU}(3)\colon g\sigma=\sigma g\}
\end{equation*}
is a subgroup of $\mathsf{SU}(3)$ conjugate to the standard $\mathsf{SO}(3)$.
The corresponding Lie algebra is given by $\g{so}(3)^{\sigma}=\vecspan\{ \operatorname{diag}(i,0,-i),e_{1}+e_{3},e_{2}+e_{4}\}$, and in particular it is canonically embedded in $\g{su}(3)$.
One sees that the isotropy subgroup $\mathsf{SO}(3)^{\sigma}_{o}$ is the $\mathsf{U}(1)$ subgroup generated by $\g{so}(3)^{\sigma}\cap\g{k}$, so we obtain that $\mathsf{SO}(3)^{\sigma}\cdot o$ is a totally geodesic submanifold of $\mathsf{F}(\C^{3})$ that is diffeomorphic to a sphere.
Its tangent space at $o$ is
$
	\g{p}_{\g{so}(3)^{\sigma}}=\vecspan\{e_{1}+e_{3},e_{2}+e_{4}\},
$
and this plane has sectional curvature $1/2$, so $\mathsf{SO}(3)^{\sigma}\cdot o$ is isometric to a two-dimensional sphere of radius $\sqrt{2}$.
The equality $J(\g{p}_{\g{so}(3)^{\sigma}})=\g{p}_{\g{so}(3)^{\sigma}}$ implies that $\mathsf{SO}(3)^{\sigma}\cdot o$ is an almost complex surface in $\mathsf{F}(\C)^{3}$.
Since $\mathsf{SO}(3)^{\sigma}\cdot o$ is not well-positioned at $o$, it is not well-positioned.
	
\subsubsection{Real projective planes inside $\mathsf{F}(\R^{3})$}
\label{subsec:projplanerealflag}
Recall that $\mathsf{F}(\R^{3})$ is a Lagrangian submanifold with constant sectional curvature.
In particular, every $2$-plane inside $\g{p}_{\mathsf{F}(\R^{3})}$ will give rise to a totally geodesic surface inside $\mathsf{F}(\R^{3})$ (hence inside $\mathsf{F}(\C^{3})$).	
We describe these examples.
	
As we saw earlier, $\s{F}(\R^{3})$ can be regarded as the Lie group quotient $\s{Sp}(1)/\s{Q}_{8}$ with a metric of constant curvature equal to $1/8$.
As a consequence, $\s{F}(\R^{3})$ is isometric to the quotient $\s{S}^{3}(2\sqrt{2})/\s{Q}_{8}$, and the projection map $\pi\colon \s{S}^{3}(2\sqrt{2})\to\s{F}(\R^{3})$ is a Riemannian covering map.
This projection is equivariant with respect to the double cover $\s{Sp}(1)\to\s{SO}(3)$.
	
We view $\H\equiv\R^{4}$.
Consider the totally geodesic embedding $h\colon\s{S}^{2}(2\sqrt{2})\hookrightarrow \s{S}^{3}(2\sqrt{2})$ defined by $h(x,y,z)=xi+yj+zk$.
Then $\pi\circ h$ is also a totally geodesic immersion of the $2$-sphere satisfying $\pi\circ h(-x,-y,-z)=\pi\circ h(x,y,z)$ for all $(x,y,z)\in\s{S}^{2}(2\sqrt{2})$, so it factors through an isometric immersion $\phi\colon \R\s{P}^{2}(2\sqrt{2})\to\s{F}(\R^{3})$ defined via
\begin{equation}\label{eq:RP2Immersion}
	\phi([x:y:z])=(xi+yj+zk)\s{Q}_{8}, \quad (x,y,z)\in\s{S}^{2}(2\sqrt{2}).
	\end{equation}
As the projection $\s{S}^{2}(2\sqrt{2})\to\R\s{P}^{2}(2\sqrt{2})$ is also a covering map, we deduce that $\phi$ is a totally geodesic immersion.
Note that $\phi$ is not injective, as the points $[2\sqrt{2}:0:0]$, $[0:2\sqrt{2}:0]$ and $[0:0:2\sqrt{2}]$ have the same image.
	
\begin{proposition}
	The map $\phi\colon \R\s{P}^{2}(2\sqrt{2})\to\s{F}(\R^{3})$ defined by \eqref{eq:RP2Immersion} is a non-injective inextendable compatible totally geodesic immersion.
\end{proposition}
	
\begin{proof}
	Since $\R\s{P}^{2}(2\sqrt{2})$ is complete, we only need to show that $\phi$ is compatible.
	This is equivalent to proving the following: for every pair of different points $p=[x:y:z]$ and $q=[x':y':z']\in\R\s{P}^{2}(2\sqrt{2})$ such that $\phi(p)=\phi(q)$, the images $\phi_{*p}(T_{p}\R\s{P}^{2}(2\sqrt{2}))$ and $\phi_{*q}(T_{q}\R\s{P}^{2}(2\sqrt{2}))$ are different subspaces of $T_{\phi(p)}\s{F}(\R^{3})$.
		
	Let $p=[x:y:z]$ and $q=[x':y':z']$ be as above.
	Then $\phi(p)=\phi(q)$ implies that there is an element $\lambda\in\s{Q}_{8}$ such that $x'i+y'j+z'k=(xi+yj+zk)\lambda$.
	Changing the sign of the homogeneous coordinates of $q$ if necessary, we may assume that $\lambda\in \{i,j,k\}$.
	We will deal with the case $\lambda=i$, as the other two cases can be treated in an analogous manner.
	In this setting, we obtain that $x'i+y'j+z'k=(xi+yj+zk)i=-x-yk+zj$, which yields $x=x'=0$, $y'=z$ and $z'=-y$, so $p=[0:y:z]$ and $q=[0:z:-y]$. 
	
	Let us compute $\phi_{*p}(T_{p}\R\s{P}^{2}(2\sqrt{2}))$.
	On the one hand, we can identify the tangent space of $\R\s{P}^{2}(2\sqrt{2})$ at $p=[0:y:z]$ with the tangent space $T_{(0,y,z)}\s{S}^{2}(2\sqrt{2})\equiv \R (0,y,z)^{\perp}=\spann\{(1,0,0),(0,-z,y)\}$.
	Moreover, we can also view $T_{(yj+zk)\s{Q}_{8}}\s{F}(\R^{3})$ as $T_{(yj+zk)}\s{S}^{3}(2\sqrt{2})\equiv\R(yj+zk)^{\perp}$.
	Under these identifications, $\phi_{*p}(T_{p}\R\s{P}^{2}(2\sqrt{2}))$ is spanned by $\phi_{*p}(1,0,0)=i$, and  $\phi_{*p}(0,-z,y)=-zj+yk$.		
	We now determine $\phi_{*q}(T_{q}\R\s{P}^{2}(2\sqrt{2}))$.
	For this, we have identifications $T_{q}\R \s{P}^{2}(2\sqrt{2})\equiv \R(0,z,-y)^{\perp}=\spann \{(1,0,0),(0,y,z)\}$ and $T_{(-zj+yk)}\s{F}(\R^{3})\equiv \R(-zj+yk)^{\perp}$.
	We obtain that $\phi_{*q}(T_{q}\R\s{P}^{2}(2\sqrt{2}))$ is generated by $
	\phi_{*q}(1,0,0)=i$, and $\phi_{*q}(0,y,z)=yj+zk$.
	In order to finish, observe that the composition of the isomorphisms \[\R(yj+zk)^{\perp}\to T_{(yj+zk)\s{Q}_{8}}\s{F}(\R^{3})=T_{(zj-yk)\s{Q}_{8}}\s{F}(\R^{3})\to \R(zj-yk)^{\perp}\]
	is simply right multiplication by $i$, so $\phi_{*p}(T_{p}\R\s{P}^{2}(2\sqrt{2}))$, regarded as a subspace of $\R (zj-yk)^{\perp}$, is spanned by $1$ and $yj+zk$.
	Thus, we obtain that the images of $\phi_{*p}$ and $\phi_{*q}$ are different, and therefore $\phi$ is a compatible immersion. \qedhere
\end{proof}
	
The next lemma shows that, up to congruence, $\R \s{P}^{2}(2\sqrt{2})$ is the unique totally geodesic surface of $\s{F}(\R^{3})$.
	
\begin{lemma}
	Let $\psi\colon \Sigma \to \s{F}(\R^{3})$ be a compatible totally geodesic immersion of a complete two-dimensional Riemannian manifold.
	Then $\psi$ is congruent to $\phi$ under an element of $\s{SO}(3)$.
	In particular, $\psi$ and $\phi$ are congruent as immersions into $\s{F}(\C^{3})$ as well.
\end{lemma}
\begin{proof}
	Let $a\in\s{Sp}(1)$ be arbitrary, and take the map $\phi_{a}\colon \R\s{P}^{2}(2\sqrt{2})\to \s{F}(\R^{3})=\s{S}^{3}(2\sqrt{2})/\s{Q}_{8}$ given by $\phi_{a}([x:y:z])=a(xi+yj+zk)\s{Q}_{8}$.
	Since left multiplication by $a$ is an isometry, $\phi_{a}$ is also a compatible totally geodesic immersion of $\R\s{P}^{2}$ congruent to $\phi$.
	We show that all totally geodesic surfaces arise in this manner.
	
	Let $\pi\colon \s{S}^{3}(2\sqrt{2})\to \s{F}(\R^{3})$ be the projection map and consider the totally geodesic sphere $\s{S}^{2}_{1}(2\sqrt{2})\subseteq \s{S}^{3}(2\sqrt{2})$ obtained as the intersection of $\spann\{i,j,k\}$ with our $3$-sphere.
	Take any point $p=\pi(z)\in\s{F}(\R^{3})$ and a two-dimensional subspace $V\subseteq T_{\pi(z)}\s{F}(\R^{3})$.
	As $\pi$ is a Riemannian covering map, we may regard $V$ as a subspace of $T_{z}\s{S}^{3}(2\sqrt{2})\equiv \R z ^{\perp}$, where we are considering the standard inner product on $\H\equiv \R^{4}$.
	Let $a\in \s{Sp}(1)$ be orthogonal to $V$ and $z$ and consider the great sphere $\s{S}^{2}_{a}=a\cdot \s{S}^{2}_{1}(2\sqrt{2})$.
	Note that $\s{S}^{2}_{a}$ coincides with the great sphere obtained by intersecting $\s{S}^{3}(2\sqrt{2})$ with the subspace $V\oplus \R z$, and thus the map $h_{a}\colon \s{S}^{2}(2\sqrt{2})\to \s{S}^{3}(2\sqrt{2})$ defined by $h_{a}(x,y,z)=a(xi+yj+zk)$ is the unique compatible totally geodesic immersion passing through $z$ with tangent space $V$.
	Since $h_{a}(-x,-y,-z)=-h_{a}(x,y,z)$, the map $h_{a}$ descends to the map $\phi_{a}\colon \R\s{P}^{2}(2\sqrt{2})\to \s{F}(\R^{3})$, so $\phi_{a}$ passes through $p=\pi(z)$ with tangent space $V$.
	As $p$ and $V$ are arbitrary, we conclude that every complete compatible totally geodesic immersion from a surface to $\s{F}(\R^{3})$ is equivalent to one of the form $\phi_{a}$, and is thus congruent to $\phi\colon \R\s{P}^{2}(2\sqrt{2})\to\s{F}(\R^{3})$.
	The element in $\s{SO}(3)$ that achieves this congruence is the image of $a$ under the double cover $\s{Sp}(1)\to \s{SO}(3)$. \qedhere
\end{proof}	
	
Clearly, the fact that $\phi(\R\mathsf{P}^{2}(2\sqrt{2}))$ is contained in a Lagrangian submanifold implies that $\phi$ is totally real.
Note that none of these submanifolds are well-positioned.
Indeed, the totally geodesic $\R\s{P}^{2}(2\sqrt{2})$ corresponding to the subspace $\g{v}=\spann\{e_{1}+e_{3},e_{5}\}$ is not well-positioned (since it is not well-positioned at $o$), and because all of these submanifolds are congruent to this $\R\s{P}^{2}(2\sqrt{2})$ by an element of $\s{SO}(3)$, it follows that no totally geodesic $\R\s{P}^{2}(2\sqrt{2})$ is well-positioned.	
Also, as $\R\s{P}^{2}(2\sqrt{2})$ is not injectively immersed, it can not arise as an extrinsically homogeneous submanifold of $\s{F}(\C^{3})$.

\begin{remark}\label{remark:CompositionCompatibleImmersions}
	Let us consider the unit speed geodesic $\gamma$ of $\R\s{P}^{2}(2\sqrt{2})$ given by the expression $\gamma(t)=\left[\cos \frac{t}{2\sqrt{2}}: \sin \frac{t}{2\sqrt{2}} :0\right]$.
	Then $\gamma$ descends to an injective totally geodesic immersion $f\colon \s{S}^{1}=\R/(2\sqrt{2}\pi\mathbb{Z})\to\R\s{P}^{2}(2\sqrt{2})$ defined via $f([t])=\left[\cos \frac{t}{2\sqrt{2}} : \sin \frac{t}{2\sqrt{2}} : 0\right]$. Thus, $f$ is a compatible totally geodesic immersion.
	We now take the compatible totally geodesic immersion $\phi\colon \R\s{P}^{2}(2\sqrt{2})\to\s{F}(\R^{3})$ defined as in~\eqref{eq:RP2Immersion}.
	The composition $\beta=\phi\circ f\colon \s{S}^{1}\to \s{F}(\R^{3})$ is not compatible.
	Indeed, a short calculation yields
	\[
		\begin{aligned}
		\phi(\gamma(0))=&{}\phi(\gamma(\sqrt{2}\pi))=2\sqrt{2}\s{Q}_{8}, \\
		(\phi\circ\gamma)'(0)=&{}(\phi\circ\gamma)'(\sqrt{2}\pi)=\frac{d}{dt}\bigg\vert_{t=0}2\sqrt{2}\left(\cos \frac{t}{2\sqrt{2}} i + \sin \frac{t}{2\sqrt{2}} j\right)\s{Q}_{8},
		\end{aligned}
	\]
	so $\tilde{\beta}([0])=\tilde{\beta}([\sqrt{2}\pi])$, implying that $\tilde{\beta}$ is not injective.
\end{remark}

\subsection{The almost product $\mathsf{S}^{3}\times\mathsf{S}^{3}$}\label{subsec:exs3s3}We describe the totally geodesic examples of the almost product $\s{S}^3\times \s{S}^3$ equipped with a homogeneous nearly Kähler metric.
	
\subsubsection{The fiber of $\mathsf{S}^{3}\times\mathsf{S}^{3}\to\mathsf{S}^{3}$ \cite[Example~3.1]{ZhangDioosHuVranckenWang}}\label{subsec:S3S3RoundS3}
Let $\Sigma=\mathsf{S}^{3}$ be the fiber of the projection map $(x,y)\mapsto x$, which we know from Subsection~\ref{subsection:S3S3Description} that is a totally geodesic submanifold of $\mathsf{S}^{3}\times\mathsf{S}^{3}$, and coincides with the orbit $\left(\Delta_{1,3}\mathsf{SU}(2)\times\mathsf{SU}(2)_{2}\right)\cdot o$.
It is immediate to check that the normalizer of $\Sigma$ in $\mathsf{G}$ is precisely $\mathsf{N}_{\mathsf{G}}(\Sigma)=\Delta_{1,3}\mathsf{SU}(2)\times\mathsf{SU}(2)_{2}$, and the restricted action $\mathsf{N}_{\mathsf{G}}(\Sigma)\curvearrowright \Sigma$ satisfies
\begin{equation}\label{eq:DoubleCover}
	(g,h,g)\cdot (I,x)=(I,hxg^{-1}), \quad g,h,x\in\mathsf{SU}(2),
\end{equation}
so this action coincides with the double cover $\mathsf{Spin}(4)=\mathsf{SU}(2)\times\mathsf{SU}(2)\to\mathsf{SO}(4)$ acting on $\mathsf{S}^{3}$.
As a consequence, $\Sigma$ is isometric to a round sphere.
A direct calculation yields that its sectional curvature is $3/4$, so we actually have $\Sigma=\mathsf{S}^{3}\left(2/\sqrt{3}\right)$.
The tangent space of $\Sigma$ through $o$ is
$
	\g{p}_{\mathsf{S}^{3}\left( 2/\sqrt{3} \right) }=\g{p}_{1}
$.
A direct calculation yields $J (\g{p}_{1})=\g{p}_{2}$, so $\mathsf{S}^{3}\left( 2/\sqrt{3} \right)$ is a Lagrangian submanifold.
As $\mathsf{S}^{3}\left(2/\sqrt{3}\right)$ is the fiber, it is obviously well-positioned.
	
\subsubsection{The Berger sphere \cite[Example~3.4]{ZhangDioosHuVranckenWang}}\label{subsec:S3S3BergerS3}
Consider the subgroup \[\mathsf{B}=\left\{ \left(g,k,HgH^{-1}\right)\in\mathsf{G}\colon g\in\mathsf{SU}(2),k\in\mathsf{U}(1) \right\},\] where $H\in\mathsf{SU}(2)$ is the element defined in \eqref{eq:su2Matrices} and $\mathsf{U}(1)$ is embedded in $\mathsf{SU}(2)$ diagonally.
The Lie algebra $\g{b}\subseteq \g{g}$ satisfies $\g{b}=\R(H,H,H)\oplus\vecspan\{ e_{1},e_{5},e_{6} \}=(\g{b}\cap\g{k})\oplus(\g{b}\cap\g{p})$,
so the orbit $\mathsf{B}\cdot o$ is a totally geodesic submanifold of $\mathsf{S}^{3}\times\mathsf{S}^{3}$.
As the isotropy subgroup $\mathsf{B}_{o}$ is merely the diagonally embedded $\mathsf{U}(1)$, it follows that $\mathsf{B}\cdot o$ is diffeomorphic to a $3$-sphere.
More precisely, $\mathsf{B}\cdot o$ is the Berger sphere $\mathsf{S}^{3}_{\C,1/3}(2)$. 
Its tangent space at $o$ is $\g{p}_{\mathsf{S}^{3}_{\C,1/3}(2)}=\g{b}\cap\g{p}=\vecspan\{e_{1},e_{5},e_{6}\}$.
One sees that $J(\g{p}_{\mathsf{S}^{3}_{\C,1/3}(2)})=\g{p}\ominus \g{p}_{\mathsf{S}^{3}_{\C,1/3}(2)}$, so $\mathsf{S}^{3}_{\C,1/3}(2)$ is a Lagrangian submanifold.
Finally, a direct application of Corollary~\ref{cor:CanonicallyEmbeddedWellPositioned} yields that $\mathsf{S}^{3}_{\C,1/3}(2)$ is well-positioned.

\begin{remark}
    Note that in~\cite{ZhangDioosHuVranckenWang}, the authors classify Lagrangian totally geodesic
    submanifolds in the nearly Kähler $\s{S}^3\times\s{S}^3$, where they list six examples.
    As stated in~\cite[Remark~1.4]{constant-curvature}, the first three examples in their list are
    mutually congruent (corresponding to the round sphere in Example~\ref{subsec:S3S3RoundS3}),
    whereas the last three are also mutually congruent (corresponding to the Berger sphere from
    Example~\ref{subsec:S3S3BergerS3}).
    The reason why the classification results in~\cite{ZhangDioosHuVranckenWang} are formulated
    in that manner is that the authors mainly consider the isometries of the nearly Kähler
    $\s{S}^3\times\s{S}^3$ that preserve the almost complex structure $J$, as well as its
    almost product structure $P$, see~\cite[Remark~1.4]{constant-curvature}.
\end{remark}


\subsubsection{The torus \cite[Example 1]{BoltonDillenDioosVrancken}}
	
Let $\mathsf{T}^2$ be the connected subgroup of $\mathsf{G}$ with Lie algebra $\g{t}=\vecspan\{e_{1},e_{4}\}$.
As $\mathsf{T}^2$ is contained in the torus $\mathsf{U}(1)\times\mathsf{U}(1)\times\mathsf{U}(1)$ (where $\mathsf{U}(1)$ is the diagonal subgroup of $\mathsf{SU}(2)$), it follows that
$\mathsf{T}^2$ is a two-dimensional torus.
Furthermore, as $\g{t}\subseteq \g{p}$, we see that $\g{t}$ is canonically embedded, so $\mathsf{T}^2\cdot o$ is a totally geodesic surface of $\mathsf{S}^{3}\times\mathsf{S}^{3}$ diffeomorphic to a torus. 
Consider the exponential map $\exp_{o}\colon\g{t}\to\mathsf{T}^2\cdot o$, which is $\g{t}$-equivariant in the sense that $\exp_{o}(T+S)=\Exp(T)\cdot \exp_{o}(S)$.
This means that $\exp_{o}$ is a Riemannian covering map, and $\mathsf{T}^2\cdot o$ is isometric to the quotient of $\g{t}=\R^{2}$ by the lattice $\Gamma=\exp_{o}^{-1}(o)$.
Now, given $u$, $v\in \R$, we see that $\exp_{o}(u e_{1}+ve_{4})$ is equal to
\[
	\left(\operatorname{diag}\left(e^{\frac{1}{6} i \left(\sqrt{3} u+3 v\right)},e^{-\frac{1}{6} i \left(\sqrt{3} u+3 v\right)}\right),\operatorname{diag}\left(e^{-\frac{i u}{\sqrt{3}}},e^{\frac{i u}{\sqrt{3}}}\right),\operatorname{diag}\left( e^{\frac{1}{6} i \left(\sqrt{3} u-3 v\right)},e^{-\frac{1}{6} i \left(\sqrt{3} u-3 v\right)}\right)\right)\cdot o,
\]
so the corresponding lattice is $\Gamma=\vecspan_{\mathbb{Z}} \left\{ (2\pi/\sqrt{3},2\pi),(4\pi/\sqrt{3},0) \right\}$.
Thus, $\mathsf{T}^2\cdot o$ is isometric to the flat torus $\mathsf{T}^2_{\Gamma}=\R ^{2}/\Gamma$.
It turns out that the closest points in $\Gamma \setminus \{(0,0)\}$ to the origin are those in the set $\left\{
	\pm \left(2 \pi/\sqrt{3} ,2 \pi \right), \pm\left(4\pi/\sqrt{3},0\right), \pm\left(2\pi/\sqrt{3} ,-2 \pi \right)
\right\}$, so $\mathsf{T}^2_{\Gamma}$ admits three closed geodesics of minimum length $\frac{4 \pi}{\sqrt{3}}$ passing through $o$, unlike any Riemannian product of the form $\mathsf{S}^{1}(a)\times\mathsf{S}^{1}(b)$.
By construction, we have that $\g{p}_{\mathsf{T}^2_{\Gamma}}=\vecspan \{e_{1},e_{4}\}$
is a $J$-invariant subspace, so $\mathsf{T}^2_{\Gamma}$ is an almost complex surface inside $\mathsf{S}^{3}\times\mathsf{S}^{3}$.
By Corollary~\ref{cor:CanonicallyEmbeddedWellPositioned}, $\s{T}^2_{\Gamma}$ is a well-positioned totally geodesic submanifold of $\s{S}^{3}\times\s{S}^{3}$.
	
\subsubsection{Not well-positioned totally geodesic spheres \cite[Example~2]{BoltonDillenDioosVrancken}}
	
Let $g=\left(e^{\frac{i\pi}{3}F},I,e^{-\frac{i\pi}{3}F}\right)$ and take the subgroup $g\mathsf{K}g^{-1}\subseteq \mathsf{G}$.
Its Lie algebra $\Ad(g)\g{k}$ satisfies 
\begin{equation*}
	\Ad(g)\g{k}=\R (F,F,F)\oplus \vecspan \{ e_{1}+e_{5},e_{2}-e_{4} \}=\left(\Ad(g)\g{k} \cap \g{k}\right)\oplus \left( \Ad(g)\g{k}\cap \g{p} \right),
\end{equation*}
so $\Ad(g)\g{k}$ is a canonically embedded subalgebra of $\g{g}$ and the orbit $(g\mathsf{K}g^{-1})\cdot o$ is a totally geodesic surface.
The isotropy subgroup $(g\mathsf{K}g^{-1})\cdot o$ is isomorphic to $\mathsf{U}(1)$, and thus $(g\mathsf{K}g^{-1})\cdot o$ is isometric to a round sphere.
A simple computation yields that its sectional curvature is $2/3$, so $(g\mathsf{K}g^{-1})\cdot o$ is a round sphere with radius $\sqrt{3/2}$.
Its tangent space at $o$ is given by $\g{p}_{\mathsf{S}^{2}( \sqrt{3/2} )}=\Ad(g)\g{k}\cap \g{p}=\vecspan\{e_{1}+e_{5},e_{2}-e_{4}\}$.
One sees that $J$ preserves this subspace, and by homogeneity we deduce that $\mathsf{S}^{2}( \sqrt{3/2})$ is an almost complex surface in $\mathsf{S}^{3}\times\mathsf{S}^{3}$.
Clearly, $\mathsf{S}^{2}( \sqrt{3/2} )$ is not well-positioned at $o$.
	
\subsubsection{Great spheres inside $\mathsf{S}^{3}\left(2/\sqrt{3}\right)$}
\label{sec:S3S3RoundS2}
	
Let $\Sigma$ be a totally geodesic surface inside the Lagrangian round $\mathsf{S}^{3}$, so it is merely a great sphere inside $\mathsf{S}^{3}$.
Then $\Sigma\subseteq \mathsf{S}^{3}\times\mathsf{S}^{3}$ is automatically a totally real totally geodesic submanifold.
Plus, $\Sigma$ is also homogeneous.
Indeed, as $\mathsf{S}^{3}=\mathsf{SO}(4)/\mathsf{SO}(3)$ is a symmetric space, its totally geodesic submanifolds are homogeneous, so $\Sigma$ is an orbit of a subgroup $\mathsf{H}\subseteq \mathsf{SO}(4)$.
Let $\phi \colon \mathsf{N}_{\mathsf{G}}(\mathsf{S}^{3})\to \mathsf{SO}(4)$ be the double cover defined as the composition of the isometric action defined in \eqref{eq:DoubleCover} with the projection of $\mathsf{S}^{3}\times\mathsf{S}^{3}$ onto its second factor.
The preimage $\mathsf{L}=\phi^{-1}(\mathsf{H})$ is a subgroup of $\mathsf{G}$ whose orbit at any $p\in \Sigma$ coincides with $\Sigma$, so $\Sigma$ is extrinsically homogeneous.
For instance, one can take the diagonal subgroup $\s{K}\subseteq \s{N}_{\s{G}}(\s{S}^3)$, and the orbit $\s{K}\cdot (I,H)$ is an example of these spheres.
Similarly, if $\Sigma'$ is another totally geodesic surface inside the round $\mathsf{S}^{3}$, there exists an element $g\in\mathsf{SO}(4)$ such that $g\cdot \Sigma = \Sigma '$, which implies that any element $h\in \phi^{-1}(g)$ also satisfies $h\cdot \Sigma=\Sigma'$.
As $\Sigma$ is contained in the fiber of $\s{S}^{3}\times\s{S}^{3}\to\s{S}^{3}$, its tangent space at every point is contained in the vertical subspace, so $\Sigma$ is well-positioned.

It is worth noting that even though theses spheres are extrinsically homogeneous, they are not $D$-invariant.
Indeed, observe that the difference tensor restricted to $\g{p}_{1}$ is given by
\[
	\frac{1}{2}[(X,-2X,X),(Y,-2Y,Y)]_{\g{p}}=\frac{1}{2}(-[X,Y],2[X,Y],-[X,Y]), \quad X,Y\in\g{su}(2),
\]
which means that the $D$-invariant subspaces of $\g{p}_{1}$ are in a one-to-one correspondence with the Lie subalgebras of $\g{su}(2)$.
As $\g{su}(2)$ admits no codimension one subalgebras, it follows that no two-dimensional subspace of $\g{p}_{1}$ (and thus no totally geodesic sphere inside the fiber $\s{S}^3$) is $D$-invariant.
	
\begin{remark}
	\label{rem:alekseevski-nikonorov}
	The round $\s{S}^{2}$ described in Section~\ref{sec:S3S3RoundS2} serves as a counterexample to~\cite[Proposition~2]{AlekseevskyNikonorov}.
	In this result, the authors claim that for a compact geodesic orbit space $M=\s{G}/\s{K}$ with reductive decomposition $\g{g}=\g{k}\oplus\g{p}$, a subspace $\g{v}\subseteq\g{p}$ is tangent to an extrinsically homogeneous totally geodesic submanifold if and only if it generates a canonically embedded subalgebra $\g{s}\subseteq \g{g}$ satisfying $\g{s}_{\g{p}}=\g{v}$ and $U(\g{v},\g{v})\subseteq \g{v}$ (recall that $U$ is defined by~\eqref{eq:UTensor}).
	In the naturally reductive setting, because $U=0$, the proposition would imply that every extrinsically homogeneous totally geodesic submanifold of $M$ is automatically $D$-invariant, which we know from the aforementioned example that is not true in general.
\end{remark}

\section{Riemannian cones and totally geodesic submanifolds}\label{sec:tgcones}

In this section we start by recalling the definition and the basic properties of Riemannian cones.
After that we prove a structure result for totally geodesic submanifolds of Riemannian cones.
We refer the reader to \cite{Leistner} for a detailed account on semi-Riemannian cones.

Let $M$ be a Riemannian manifold, which for our purposes is assumed to be real analytic and complete.
We define its \textit{Riemannian cone} as the warped product $\hat{M}=\R^{+}\times_{f} M$, where $f\colon \R^{+}\to\R^{+}$ is the identity map.
More explicitly, if $g=\langle\cdot,\cdot\rangle$ denotes the metric on $M$ and $r\colon \hat{M}\to\R^{+}$ is the projection on the first factor, the inner product on $\hat{M}$ is given by $\hat{g}=\mathrm{d}r^{2}+r^{2}g$.
	
Every vector field $X\in \g{X}(\R^{+})$ (respectively, $X\in \g{X}(M)$) admits a natural extension to $\hat{M}$, which we will also denote by $X$.
In particular, if $\partial_{r}$ is the unit radial vector field on $\R^{+}$, then its natural extension to $\hat{M}$ is called the \textit{cone vector field} or the \textit{radial vector field}.
Note that at every point $(\tau,p)$ of $\hat{M}$ we have the orthogonal decomposition $T_{(\tau,p)}\hat{M}=\R(\partial_{r})_{(\tau,p)}\oplus T_{p}M$.
The Levi-Civita connection $\hat{\nabla}$ of $\hat{M}$ is characterized by the following equations for $X$, $Y\in \g{X}(M)$:
\begin{equation}\label{eq:ConeConnection}
	\begin{aligned}
		\hat{\nabla}_{\partial_{r}} \partial_{r}={}& 0, & \hat{\nabla}_{X}\partial_{r}={}&\hat{\nabla}_{\partial_{r}}X=\frac{1}{r}X, & \hat{\nabla}_{X}Y{}={}&\nabla_{X}Y-r\langle X,Y\rangle \partial_{r}.
	\end{aligned}
\end{equation}
As a consequence, the curvature tensor $\hat{R}$ is determined by the conditions
\begin{equation}
	\label{eq:curvaturecone}
	\begin{aligned}
		\hat{R}(\partial_{r},\cdot)\cdot={}&\hat{R}(\cdot,\partial_{r})\cdot=\hat{R}(\cdot,\cdot)\partial_{r}=0, \\
		\hat{R}(u,v)w={}&R(u,v)w-\langle v,w\rangle u+\langle u,w\rangle v, & u,v,w\in T_{p}M.
	\end{aligned}	
\end{equation}

Let $X=a\partial_{r}+v\in T_{(\tau,p)}\hat{M}$ be arbitrary.
From the equations above we see that the Jacobi operator associated with $X$ satisfies
\[
	\begin{aligned}
		\hat{R}_{X}(b \partial_{r}+w)={}&\hat{R}(b \partial_{r}+w,a\partial_{r}+v)(a\partial_{r}+v)=\hat{R}(w,v)v={}R_{v}w-\lvert v \rvert^{2}w+\langle v,w \rangle v.
	\end{aligned}
\]
Although the process is more tedious, it is possible to compute the covariant derivatives of the curvature tensor from~\eqref{eq:ConeConnection} and~\eqref{eq:curvaturecone}.
For instance, one can show that
\begin{equation}\label{eq:coneNablaR}
	(\hat{\nabla}_{x}\hat{R})(u,v,w)=(\nabla_{x}R)(u,v,w)-\langle x,\hat{R}(u,v)w \rangle \tau\partial_r
\end{equation}
for all $(\tau,p)\in \hat{M}$ and $x$, $u$, $v$, $w\in T_{p}M$.
We will make use of this formula later on.

\begin{remark}
	Let $M=\s{S}^{n}(1)$ be the standard round sphere of radius one. Then, its cone is the punctured Euclidean space $\R^{n+1}\setminus \{0\}$.
	However, if $r\neq 1$, the cone of $\s{S}^{n}(r)$ is not flat due to~\eqref{eq:curvaturecone}.
	This illustrates that two homothetic manifolds may not have homothetic cones.
\end{remark}
	
We can also describe the geodesics of $\hat{M}$ in terms of those of $M$.
Let $(\tau,p)\in\hat{M}$ be any point and consider the tangent vector $w=a\partial_{r}+v$, where $a\in\R$ and $v\in T_{p} M$ are arbitrary.
From~\cite[Equation~2.7]{Leistner}, we see that the geodesic $\hat{\gamma}(t)=\hat{\exp}_{(\tau,p)}(tw)=(\rho(t),\alpha(t))$ is given in a neighborhood of $t=0$ by
\begin{equation}\label{eq:ConeGeodesics}
	\rho(t)={} \sqrt{(at+\tau)^{2}+\lvert v \rvert^{2}\tau^{2}t^{2}},\quad
	\alpha(t)={}\exp_{p}(f(t)v),
\end{equation}
where
\[
	f(t)=\begin{cases}
		\frac{1}{\lvert v \rvert}\arctan\left(\frac{\lvert v \rvert \tau t}{at+\tau}\right), & v\neq0, \\
		0, & v=0.
		\end{cases}
\]
As a consequence, the maximal interval of definition of $\hat{\gamma}(t)$ contains the interval
\begin{equation}\label{eq:ConeGeodesicMaximalInterval}
	I_{a}=\begin{cases}
		\R, & a=0, \\
		\left(-\frac{\tau}{a},\infty \right), & a>0, \\
		\left(-\infty,-\frac{\tau}{a} \right) & a<0.
	\end{cases}
\end{equation}
Note that $I_{a}$ only depends on $a$.
A consequence of~\eqref{eq:ConeGeodesics} is that if $\gamma$ is a geodesic of $\hat{M}$, its projection to $M$ is a pregeodesic of $M$.
Observe that $\hat{M}$ is never complete.
However, it is clear that it is an analytic Riemannian manifold. The following lemma shows that actually the only incomplete geodesics are those tangent to the cone vector field.
\begin{lemma}\label{lemma:ConeGeodesicsIntervalOfDefinition}
	Let $M$ be a complete Riemannian manifold, $(\tau,p)\in\hat{M}$ a point in its Riemannian cone and $w=a\partial_{r}+v\in T_{(\tau,p)}M$ a unit vector, where $a\in \R$ and $v\in T_{p}M$.
	Consider the maximal geodesic $\hat{\gamma}(t)$ such that $\hat{\gamma}(0)=(\tau,p)$ and $\hat{\gamma}'(0)=w$.
	The following statements hold:
	\begin{enumerate}[\rm (i)]
		\item\label{lemmaItem:VerticalGeodesic} If $v=0$, then the maximal interval of definition of $\hat{\gamma}(t)$ is precisely $I_{a}$.
		\item\label{lemmaItem:BaseGeodesic} If $v\neq 0$, then $\hat{\gamma}(t)$ is defined on all $\R$.
	\end{enumerate}
\end{lemma}
\begin{proof}
	Firstly, assume that $v=0$.
	Without loss of generality, we may also suppose that $a$ is positive, so $a=1$.
	From~\eqref{eq:ConeGeodesics}, we see that the curve $\beta\colon (-\tau,\infty)\to \hat{M}$ defined by $\beta(t)=(t + \tau,p)$ is a geodesic of $\hat{M}$ with initial conditions $\beta(0)=(\tau,p)$, $\beta'(0)= \partial_{r}$.
	Since $t+\tau$ converges to zero as $t$ converges to $-\tau$, the curve $\beta$ is a maximal geodesic, thus proving~(\ref{lemmaItem:VerticalGeodesic}).
		
	Now, assume that $v\neq 0$, so that the curve $\beta\colon I_{a}\to \hat{M}$, $\beta(t)=(\rho(t),\alpha(t))$ defined by~\eqref{eq:ConeGeodesics} is a geodesic of $\hat{M}$ with $\beta(0)=p$, $\beta'(0)=w$.
	Note that the derivative
	\[
		\rho'(t)=\frac{2a(at+\tau)+2\lvert v\rvert^{2}\tau^{2}t}{2\rho(t)}=\frac{a\tau+(a^{2}+\lvert v\rvert^{2}\tau^{2})t}{\sqrt{(at+\tau)^{2}+\lvert v \rvert^{2}\tau^{2}t^{2}}}=\frac{a\tau+t}{\sqrt{(at+\tau)^{2}+\lvert v \rvert^{2}\tau^{2}t^{2}}}
	\]
	vanishes at $t_{0}=-a\tau\in I_{a}$ (this last inclusion holds because $\lvert a \rvert \in (0,1)$ and $\tau>0$).
	As a consequence, $\beta'(t_{0})=\alpha'(t_{0})\in T_{\alpha(t_{0})}M$.
	Looking at~\eqref{eq:ConeGeodesicMaximalInterval}, we obtain that the geodesic $\hat{\exp}_{\beta(t_{0})}(\beta'(t_{0}))$ is defined on all of $\R$, and by uniqueness it must coincide with the curve $\beta(t+t_{0})$.
	Thus, $\beta$ can be extended to all $\R$, so~(\ref{lemmaItem:BaseGeodesic}) holds.\qedhere
\end{proof}
Suppose that $f\colon M\to N$ is a smooth map.
We define its associated \textit{cone map} as the map $\hat{f}\colon \hat{M}\to\hat{N}$ given by $\hat{f}(\tau,p)=(\tau,f(p))$.
It can be easily checked that $\hat{f}$ is an isometric immersion (respectively, an isometry) if and only if $f$ is an isometric immersion (respectively, an isometry). In the following, we provide some information about the isometry group of Riemannian cones. 

\begin{proposition}\label{prop:IsometricCones}
	Let $M$ and $N$ be two complete Riemannian manifolds. Every isometry $f\colon \hat{M}\to \hat{N}$ is the cone map of an isometry $g\colon M\to N$. In particular, 	if the cones $\hat{M}$ and $\hat{N}$ are isometric, then $M$ and $N$ are also isometric.
\end{proposition}
\begin{proof}
	We take the subset $\mathcal{C}_{(\tau,p)}=\{ X\in T_{(\tau,p)}\hat{M}\colon \hat{\exp}_{(\tau,p)}(tX)\text{ is defined on }\R \}$ of $T_{(\tau,p)}\hat{M}$ for each $(\tau,p)\in \hat{M}$.
	It is clear from Lemma~\ref{lemma:ConeGeodesicsIntervalOfDefinition} that $\mathcal{C}_{(\tau,p)}=T_{(\tau,p)}\hat{M}\setminus \R (\partial_{r})_{(\tau,p)}$.
	
	Now, let $f\colon \hat{M}\to\hat{N}$ be an isometry and fix $(\tau,p)\in \hat{M}$ with image $(s,q)=f(\tau,p)$.
	Since $f$ is an isometry, it sends $\mathcal{C}_{(\tau,p)}$ to $\mathcal{C}_{(s,q)}$, and thus we have $f_{*(\tau,p)}((\partial_{r})_{(\tau,p)})=\pm (\partial_{r})_{(s,q)}$.
	The equality $f_{*(\tau,p)}((\partial_{r})_{(\tau,p)})=~- (\partial_{r})_{(s,q)}$ is not possible, because in that case the maximal geodesic $\gamma(t)=\hat{\operatorname{exp}}_{(\tau,p)}(t\partial_r)$ would be mapped to the maximal geodesic $\beta(t)=\hat{\operatorname{exp}}_{(s,q)}(-t\partial_r)$, which is not possible because the first geodesic is defined for all $t\geq 0$, while the second one is not.
	We deduce that $f_{*(\tau,p)}((\partial_{r})_{(\tau,p)})=~(\partial_{r})_{f(\tau,p)}$.
	In addition, the maximal geodesic $\gamma(t)=\hat{\operatorname{exp}}_{(\tau,p)}(t\partial_r)$ is mapped to the maximal geodesic $\beta(t)=~\hat{\operatorname{exp}}_{(s,q)}(t\partial_r)$, and their corresponding intervals of definition are $(-\tau,0)$ and $(-s,0)$.
	Hence, $\tau=s$.
	All in all, we have seen that $f$ sends the link $\{\tau\}\times M$ to $\{\tau\}\times N$ for each $\tau\in \R^+$, so $f$ takes the form $f(\tau,p)=(\tau,h(\tau,p))$ for a map $h\colon \hat{M}\to N$.
	Furthermore, at each $(\tau,p)\in\hat{M}$ we have
	\[
		\begin{aligned}
			(\partial_{r})_{f(\tau,p)}={}&f_{*(\tau,p)}((\partial_{r})_{\tau,p})=\frac{d}{dt}\bigg\vert_{t=0}(\tau+t,h(\tau+t,p))=(\partial_{r})_{f(\tau,p)}+h_{*(\tau,p)}((\partial_{r})_{(\tau,p)}),
		\end{aligned}
	\]
	which means that $h_{*(\tau,p)}((\partial_{r})_{(\tau,p)})=0$, so $h$ does not depend on $\tau$.
	In other words, there exists a map $g\colon M\to N$ such that $h(\tau,p)=g(p)$ for all $(\tau,p)\in\hat{M}$.
	The fact that $f$ is an isometry readily implies that $g$ is also an isometry. \qedhere
\end{proof}

\begin{corollary}\label{cor:ConeIsometryGroup}
	For a complete Riemannian manifold $M$, the map $f\in I(M)\mapsto \hat{f}\in I(\hat{M})$ is a Lie group isomorphism.
\end{corollary}

Many geometric properties of Riemannian manifolds can be translated into geometric properties of their cones.
For instance, as a consequence of~\eqref{eq:curvaturecone}, $M$ is an Einstein manifold with $\operatorname{Ric}=(n-1)\langle\cdot,\cdot\rangle$ if and only if $\hat{M}$ is Ricci-flat.
It turns out that nearly Kähler structures on six-dimensional manifolds are related to $\s{G}_{2}$-structures on their cones.
We briefly describe this relationship, see \cite{bar} for details.
For the general theory of $\s{G}_{2}$-manifolds, we refer the reader to \cite[Chapter~10]{joyce}.

Let $M$ be a six-dimensional strictly nearly Kähler manifold.
Then $M$ is Einstein with positive Ricci curvature~\cite[Theorem~5.2]{GrayStructure} and after rescaling the metric we may assume that the Einstein constant of $M$ is $\lambda=5$.
In that case, one defines a three-form $\phi\in\Omega^{3}(\hat{M})$ via the equations (for $X,Y,Z\in \g{X}(M)$)
\[
	\begin{aligned}
		\phi(X,Y,Z)={}&r^{3}\langle Y,(\nabla_{X}J)Z \rangle, &
		\phi(\partial_{r},X,Y)=&-\phi(X,\partial_{r},Y)=\phi(X,Y,\partial_{r})=r^2 \langle X,JY\rangle,
	\end{aligned}
\] 
and checks that $\phi$ is a parallel three-form inducing a $\s{G}_{2}$ structure on $\hat{M}$.
In addition, the restricted holonomy group of $\hat{M}$ is precisely $\s{G}_{2}$ whenever $M$ is not locally isometric to $\s{S}^6$.
Conversely, suppose that $\hat{M}$ is a $\s{G}_{2}$-manifold whose structure is given by the parallel three-form $\phi$.
Then the almost complex structure $J$ defined on $M$ by
$
	\langle X,JY \rangle=\phi(\partial_{r},X,Y)
$
is strictly nearly Kähler and $M$ is an Einstein manifold with $\operatorname{Ric}=5\langle\cdot,\cdot\rangle$.
\begin{remark}
Notice that the metrics of the nearly Kähler manifolds $\C\s{P}^3, \s{F}(\C^3), \s{S}^3\times\s{S}^3$ that we are considering have Einstein constants $5/2$, $5/2$ and $5/3$, respectively. Therefore, one would have to rescale these metrics by $1/2$, $1/2$, and $1/3$ to obtain the $\s{G}_2$-cones over them. However, for our purposes this is not a problem, since the totally geodesic property is preserved under rescalings of the ambient manifold and, as we will see, the maximal totally geodesic submanifolds of these $\s{G}_2$-cones are cones over the totally geodesic submanifold of a homogeneous nearly Kähler $6$-manifold.
\end{remark}

There is also a relationship between submanifolds of $M$ that have a nice interaction with $J$ and calibrated cones inside 
$\hat{M}$.
The notions of calibrated geometry were introduced in the seminal paper \cite{HarveyLawson} by Harvey and Lawson.
We remind that a \textit{calibration} on a Riemannian manifold $N$ is a closed differential form $\omega\in\Omega^{k}(M)$ satisfying
$
	\omega(v_{1},\dots,v_{k})\leq 1
$
whenever $v_{1},\dots,v_{k}$ are unit vectors in $TN$.
A $k$-dimensional oriented submanifold $S$ of $N$ is \textit{calibrated} if the restriction of $\omega$ to $S$ is equal to the Riemannian volume form of $S$, and it follows that $S$ is a minimal submanifold.
It can be shown that for the case of a $\s{G}_{2}$-manifold $(N,\phi)$, both $\phi$ and its Hodge dual $\star\phi$ are calibrations (\cite[Theorem~1.4 and Theorem~1.16]{HarveyLawson}).
We say in this case that $S$ is \textit{associative} (respectively, \textit{coassociative}) if it is calibrated with respect to $\phi$ (respectively, $\star\phi$).
Coming back to the case that $N=\hat{M}$ is the cone of a six-dimensional strict nearly K\"{a}hler manifold, it is known that the cone of a $J$-holomorphic curve is an associative submanifold, whereas the cone of a Lagrangian submanifold is a coassociative submanifold.

\subsection{Totally geodesic submanifolds of Riemannian cones}\label{subsec:tgcones}

We now discuss the relationship between the totally geodesic submanifolds of a Riemannian cone (over a complete real analytic manifold) and those of its base.
We are interested in determining the maximal totally geodesic submanifolds of the cone $\hat M$ over $M$.

This first result shows that every totally geodesic submanifold of the base induces a totally geodesic submanifold of the cone by means of the cone map.
\begin{lemma}
\label{lemma:tgcone}
Let $M$ be a Riemannian manifold and $\phi\colon S\to M$ an isometric immersion of a $k$-dimensional submanifold $S$. Then the following statements hold:
\begin{enumerate}[\rm (i)]
	\item\label{lemmaItem:tgcone} 
The immersion $\phi$ is totally geodesic  if and only if the cone map $\hat{\phi}\colon \hat{S}\to \hat{M}$ is totally geodesic.
\item\label{lemmaItem:tgconecompatible}
The totally geodesic immersion $\phi$ is compatible  if and only if the cone map $\hat{\phi}\colon \hat{S}\to \hat{M}$ is compatible. 
\end{enumerate}
\end{lemma}

\begin{proof}
	First of all, as being totally geodesic is a local property, we may suppose that $S\subseteq M$ is embedded and $\phi$ is the inclusion map.
	As a consequence, $\hat{S}=\R^{+}\times S$ as a subset of $\hat{M}$.
	
	Firstly, assume that $S$ is totally geodesic in $M$.
	Given $(\tau,p)\in\hat{S}$ and $w=a \partial_{r}+v \in T_{(\tau,p)}\hat{S}$ (so $v\in T_{p}S$), we know by~\eqref{eq:ConeGeodesics} that the geodesic $\hat{\gamma}(t)=\hat{\exp}_{(\tau,p)}(tw)$ is of the form $\hat{\gamma}(t)=(\rho(t),\beta(t))$, where $\beta(t)$ is a pregeodesic of $M$ such that $\beta'(0)=v$.
	Since $S$ is totally geodesic, there exists an $\varepsilon>0$ such that $\beta(t)\in S$ for all $t\in(-\varepsilon,\varepsilon)$, so $\hat{\gamma}(t)\in \hat{S}$ for all $t\in(-\varepsilon,\varepsilon)$.
	Therefore, $\hat{S}$ is totally geodesic in $\hat{M}$.
	
	Conversely, suppose that $\hat{S}$ is totally geodesic, and let $p\in S$, $v\in T_{p}S$.
	The geodesic $\hat{\gamma}(t)=\hat{\exp}_{(1,p)}(tw)$ is locally of the form $(\rho(t),\beta(t))$, where $\beta(t)=\exp_{p}(f(t)v)$ for a diffeomorphism $f(t)$ such that $f(0)=0$ and $f'(0)=1$.
	Thus, as $\hat{S}$ is totally geodesic, there exists an $\varepsilon>0$ such that $\hat{\gamma}(t)\in\hat{S}$ for $\lvert t \rvert < \varepsilon$,  which means that $\exp_{p}(f(t)v)\in S$ for the same values of $t$.
	As $f^{-1}$ is continuous at $t=0$, it follows that there exists $\delta > 0$ such that $\exp_{p}(sv)\in S$ for $\lvert s \rvert <\delta$.
	We conclude that $S$ is totally geodesic in $M$, proving  \eqref{lemmaItem:tgcone}.
	
	Finally, observe that since $\hat{\phi}_{*(\tau,p)}(T_{(\tau,p)}\hat{S})=\R (\partial_{r})_{\phi(\tau,p)}\oplus \phi_{*p}(T_{p}S)$ for all $(\tau,p)\in\hat{S}$, it follows that the induced map of $\hat{\phi}$ is an injection of $\hat{S}$ to $\s{G}_{k+1}(T\hat{M})$ if and only if the induced map of $\phi$ is an injection of $S$ to $\s{G}_k(TM)$. This yields \eqref{lemmaItem:tgconecompatible}.
\end{proof}

It was shown in~\cite{JentschMoroianuSemmelmann} that certain types of Riemannian manifolds with special holonomy do not admit totally geodesic hypersurfaces.
We deduce that the same result holds for Sasakian--Einstein, $6$-dimensional nearly Kähler and nearly parallel $\s{G}_{2}$-manifolds.

\begin{theorem}
\label{th:hyp}
Let $M$ be a complete Riemannian manifold with non-constant sectional curvature. Assume that $M$ satisfies one of the following conditions:
\begin{enumerate}[\rm (i)]
	\item $M^{2n+1}$ is Sasakian--Einstein,
	\item $M^6$ is a $6$-dimensional strictly nearly Kähler manifold, or
	\item $M^7$ is a nearly parallel $\s{G}_2$-manifold.
\end{enumerate}
Then, $M$ does not admit a totally geodesic hypersurface.
\end{theorem}
\begin{proof}
Observe that in all three cases $M$ is an Einstein manifold with positive Ricci curvature.
As the (non)existence of totally geodesic hypersurfaces is a purely local question, we may suppose that $M$ is simply connected.
Furthermore, as their existence is also independent of rescalings of the metric, we may also suppose that the Einstein constant of $M$ is equal to $\dim M - 1$.

Let $\Sigma$ be a totally geodesic hypersurface of $M$. By Lemma~\ref{lemma:tgcone}, $\hat{\Sigma}$ is a totally geodesic hypersurface of $\hat{M}$. By Gallot's theorem~(see \cite{gallot}), $\hat{M}$ is locally irreducible since $M$ has non-constant sectional curvature. Moreover, by \cite{bar}, we know that:
\begin{enumerate}[\rm (i)]
	\item If $M^{2n+1}$ is a Sasakian--Einstein manifold, then the restricted holonomy of $\hat{M}$ is contained in $\s{SU}(n+1)$.
	\item If $M^6$ is a strictly nearly Kähler manifold, then the restricted holonomy of $\hat{M}$ is contained in $\s{G}_{2}$.
	\item If $M^7$ is a nearly parallel $\s{G}_2$-manifold, then the restricted holonomy of $\hat{M}$ is contained  in $\s{Spin}(7)$.
\end{enumerate}
Now, since $\hat{M}$ is Einstein (indeed Ricci-flat), by  \cite[Theorem 4.3]{JentschMoroianuSemmelmann} the restricted holonomy of the cone $\hat{M}$ is $\s{SO}(T_p \hat{M})$.  This contradicts the fact that the holonomy of $\hat{M}$ is contained in one of three aforementioned groups, yielding the result.
\end{proof}

The following result is concerned with the extendability of cones over totally geodesic submanifolds of the base.

\begin{lemma}\label{lemma:ConesOverTGSubmanifolds}
	Let $M$ be a complete real analytic Riemannian manifold and $\phi\colon S\to M$ a compatible totally geodesic immersion of a $k$-dimensional complete submanifold $S$.
	Then $\hat{\phi}\colon\hat{S}\to\hat{M}$ is an inextendable compatible totally geodesic immersion.
\end{lemma}

\begin{proof}
	We have already seen in Lemma~\ref{lemma:tgcone} that $\hat{\phi}\colon \hat{S}\to\hat{M}$ is compatible, so we only need to show inextendability.
	Let $(\tau,p)\in\hat{S}$ be arbitrary and $w= a (\partial_{r})_{(\tau,p)}+v$ be a nonzero tangent vector, where $v\in T_{p}M$.
	We consider the $\hat{S}$-geodesic  $\gamma(t)=\exp_{(\tau,p)}(tw)$.
	If $v\neq 0$, then $\gamma$ is defined on all $\R$ due to Lemma~\ref{lemma:ConeGeodesicsIntervalOfDefinition}, so $\hat{\phi}\circ \gamma$ is also globally defined.
	Otherwise, we have $v=a(\partial_{r})_{(\tau,p)}$, and Lemma~\ref{lemma:ConeGeodesicsIntervalOfDefinition} implies that $\gamma$ is defined precisely on $I_{a}$.
	Because $\hat{\phi}_{*(\tau,p)}((\partial_{r})_{(\tau,p)})=(\partial_{r})_{(\tau,\phi(p))}$, the maximal $\hat{M}$-geodesic $\exp_{(\tau,\phi(p))}(t\hat{\phi}_{*(\tau,p)}(w))$ is also defined exactly on $I_{a}$, so it coincides with $\hat{\phi}\circ \gamma$.
	We conclude that $\hat{\phi}$ sends maximal geodesics of $\hat{S}$ to maximal geodesics of $\hat{M}$, so it is inextendable by Proposition~\ref{prop:TGInextendableCharacterization}. \qedhere
\end{proof}

We can now prove the following characterization of totally geodesic submanifolds in cones:

\begin{theorem}\label{th:TGSubmanifoldsOfCones}
	Let $M$ be a connected $n$-dimensional complete real analytic Riemannian manifold and consider its Riemannian cone $\hat{M}$.
	Suppose $\Sigma$ is a $k$-dimensional manifold with $1\leq k\leq n$ and $f\colon \Sigma\to \hat{M}$ is an inextendable compatible totally geodesic immersion, and let $x\in\Sigma$, $(\tau,p)=f(x)$ and $V=\tilde{f}(x)=f_{*x}(T_{x}\Sigma)$.
	Then exactly one of the following two situations occur:
	\begin{enumerate}[\rm (i)]
		\item\label{thItem:partialRIsTangent}  The vector $(\partial_{r})_{(\tau,p)}$ is in $V$.
		In this case, $\Sigma$ is incomplete, the vector field $\partial_{r}$ is everywhere tangent to the immersion $f$ and there exists a complete compatible totally geodesic immersion $g\colon S\to M$ such that $f$ and $\hat{g}$ are equivalent.
		\item\label{thItem:partialRIsNotTangent} The vector $(\partial_{r})_{(\tau,p)}$ is not in $V$.
		In this case, $\Sigma$ is complete, the vector field $\partial_{r}$ is nowhere tangent to the immersion and there exists:
		\begin{itemize}
			\item a complete compatible totally geodesic immersion $g\colon S\to M$,
			\item a complete compatible totally geodesic immersion $h\colon E\to \hat{S}$, where $E$ is a hypersurface in $\hat{S}$,
			\item and a surjective local isometry $\rho\colon E\to \Sigma$,
		\end{itemize}
		such that the following diagram commutes:
		\[
		\begin{tikzcd}
			E \arrow[d, "\rho"'] \arrow[r, "h"] & \hat{S} \arrow[d, "\hat{g}"] \\
			\Sigma \arrow[r, "f"]               & \hat{M}                     
		\end{tikzcd}
		\]
	\end{enumerate}
\end{theorem}
\begin{proof}
	We work with the natural projection $\pi\colon \hat{M}\to M$.
	Let $W=\pi_{*(\tau,p)}(V)\subseteq T_{p}M$, which is precisely the orthogonal projection of $V$ onto $T_{p}M$.
	We have that $V\subseteq \R (\partial_{r})_{(\tau,p)}\oplus W$ and the dimension of $W$ is either $k-1$ or $k$, depending on whether $(\partial_{r})_{(\tau,p)}$ is in $V$ or not.
	
	Firstly, suppose that $(\partial_{r})_{(\tau,p)}\in V$, so $V=\R(\partial_{r})_{(\tau,p)}\oplus W$.
	Then $\Sigma$ is not complete because the geodesic $\exp_{(\tau,p)}(t(\partial_{r})_{(\tau,p)})$ is not defined on all $\R$, and for every $y\in \Sigma$ the tangent space $\tilde{f}(y)$ contains $(\partial_{r})_{f(y)}$ (otherwise, $\Sigma$ would be complete by Corollary~\ref{cor:CompleteTGSubmanifolds}).
	Therefore, the radial vector field $\partial_{r}$ is everywhere tangent to the immersion $f\colon \Sigma\to\hat{M}$.
	Now, consider an $\varepsilon>0$ such that $\hat{\exp}_{(\tau,p)}$ is a diffeomorphism of the ball $B_{T_{(\tau,p)}\hat{M}}(0,\varepsilon)$ onto its image and the set $F=\hat{\exp}_{(\tau,p)}(V\cap B_{T_{(\tau,p)}\hat{M}}(0,\varepsilon))$ is an embedded totally geodesic submanifold of $\hat{M}$.
	Then $\partial_{r}$ is everywhere tangent to $F$ and the restriction of $\pi$ to $F$ has constant rank equal to $k-1$, so the constant rank theorem implies that (perhaps after shrinking $\varepsilon$) the image $\pi(F)$ is a $(k-1)$-dimensional embedded submanifold of $M$ and $\pi\colon F\to\pi(F)$ is a surjective submersion.
	Let $(s,q)\in F$ be any point and consider a nonzero $w\in T_{q}\pi(F)=\pi_{*(s,q)}(T_{(s,q)}F)$.
	Then, since $(\partial_{r})_{(s,q)}$ is tangent to $F$, the vector $w\in T_{(s,q)}\hat{M}$ is also tangent to $F$.
	As $F$ is totally geodesic, we may choose  $\delta>0$ such that the $\hat{M}$-geodesic $\hat{\gamma}(t)=\hat{\exp}_{(s,q)}(tw)$ is in $F$ for all $t\in(-\delta,\delta)$, and as a consequence the curve $\pi(\hat{\gamma}(t))$ is also contained in $\pi(F)$ for $t\in(-\delta,\delta)$.
	Recall from~\eqref{eq:ConeGeodesics} that $\pi(\hat{\gamma}(t))=\exp_{q}(f(t)w)$, where $f\colon \R \to \R$ is a homeomorphism satisfying $f(0)=0$.
	Because of this, the geodesic $\exp_{q}(tw)$ is contained in $\pi(F)$ for small values of $t$.
	This proves that $\pi(F)$ is a totally geodesic submanifold of $M$, and in particular the subspace $W$ is totally geodesic in $T_{p}M$.
	Consider the complete compatible totally geodesic immersion $g\colon S\to M$ associated with $W$, and let $y\in S$ be the unique point with $g(y)=p$ and $\tilde{g}(y)=W$.
	The cone map $\hat{g}\colon \hat{S}\to\hat{M}$ is an inextendable compatible totally geodesic immersion by Lemma~\ref{lemma:ConesOverTGSubmanifolds} and it satisfies $\hat{g}_{*(\tau,y)}(T_{(\tau,y)}\hat{S})=\R (\partial_{r})_{(\tau,p)}\oplus W = V$, so $f$ and $\hat{g}$ are equivalent by uniqueness.
	
    Secondly, assume that $(\partial_{r})_{(\tau,p)}\notin V$, so that $V$ is a hyperplane in $\R (\partial_{r})_{(\tau,p)}\oplus W$.
	In this setting, $\Sigma$ is complete by Corollary~\ref{cor:CompleteTGSubmanifolds} and for all $y\in\Sigma$ the tangent space $\tilde{f}(y)$ does not contain the vector $(\partial_{r})_{f(y)}$ (otherwise, $\Sigma$ would admit a noncomplete geodesic).
	Thus, the radial vector field is nowhere tangent to $f\colon \Sigma \to M$.
	We now argue in a similar way as in the previous paragraph.
	Let $\varepsilon>0$ be such that $\hat{\exp}_{(\tau,p)}$ is a diffeomorphism of $B_{T_{(\tau,p)}\hat{M}}(0,\varepsilon)$ to its image and $F=\hat{\exp}_{(\tau,p)}(V\cap B_{T_{(\tau,p)}\hat{M}}(0,\varepsilon))$ is an embedded totally geodesic submanifold of $\hat{M}$.
	Then, as $\partial_{r}$ is nowhere tangent to $F$, the restriction of $\pi$ to $M$ has constant rank equal to $k$, so we may shrink $\varepsilon$ so as to have that $\pi(F)$ is a $k$-dimensional embedded submanifold of $M$ and $\pi\colon F\to \pi(F)$ is a diffeomorphism.
	The same argument as above shows that $\pi(F)$ is a totally geodesic submanifold, so in particular $W$ is a totally geodesic subspace of $T_{p}M$.
	Let $g\colon S\to M$ be the associated complete compatible totally geodesic extension and $y\in S$ the unique point with $g(y)=(\tau,p)$ and $\tilde{g}(y)=W$.
	Then the cone map $\hat{g}\colon\hat{S}\to\hat{M}$ is the inextendable compatible totally geodesic immersion associated with $\R(\partial_{r})_{(\tau,p)}\oplus W$.
	Because $V\subseteq \R(\partial_{r})_{(\tau,p)}\oplus W$, we may use Proposition~\ref{prop:TGInclusions} to conclude.
\end{proof}

\begin{corollary}\label{cor:MaximalTGSubmanifoldsOfCones} Let $M$ be an analytic Riemannian manifold and $\widehat{M}$ its Riemannian cone.
	If $\Sigma$ is a maximal totally geodesic submanifold of $\hat{M}$, then $\Sigma$ is either a hypersurface of $\hat{M}$ or the cone over a maximal totally geodesic submanifold $S$ of $M$.
\end{corollary}

Theorem~\ref{th:TGSubmanifoldsOfCones} reduces the classification of (maximal) totally geodesic submanifolds of cones to that of totally geodesic submanifolds of the base manifold, and separately to that of totally geodesic hypersurfaces.
We note that these hypersurfaces may not arise as cones over totally geodesic hypersurfaces in the base space, as we will see in Example~\ref{ex:unitspherecone} and Example~\ref{ex:nontrivalexconetg}.

\begin{remark}\label{remark:SkewCones}
	Let $M$ be a Riemannian manifold and suppose that $\Sigma$ is a totally geodesic hypersurface of $M$ that is not tangent to the cone vector field $\partial_{r}$.
	We may assume without loss of generality that $\Sigma\subseteq \hat{M}$ is embedded, and we choose a point $(\tau,p)\in \Sigma$.
	The tangent space $V=T_{(\tau,p)}\Sigma\subseteq T_{(\tau,p)}\hat{M}$ is a totally geodesic hyperplane satisfying $\partial_{r}\notin V$.
	This means that there exists a unique (possibly zero) vector $\eta\in T_{p}M$ such that $T_{(\tau,p)}\hat{M}\ominus V=\R (\partial_{r}+\eta)$.
	Let $X=a\partial_{r}+v\in V$ be arbitrary, where $a\in \R$ and $v\in V$.
	Then, since $V$ is $\hat{R}_{X}$-invariant and $\hat{R}_{X}$ is symmetric, it follows that $\partial_{r}+\eta$ is an eigenvector of $\hat{R}_{X}$.
	However, by~\eqref{eq:curvaturecone} the image of $\hat{R}_{X}$ is contained in $T_{p}M$, and because $\partial_{r}+\eta$ is not tangent to the link we must have $\hat{R}_{X}(\partial_{r}+\eta)=0$, so $\hat{R}(\eta,v)v=\hat{R}(\partial_{r}+\eta,X)X=0$	for all $X\in V$.
	As the orthogonal projection of $V$ onto $T_{p}M$ is a linear isomorphism, we deduce that if $V=\R(\partial_{r}+\eta)^{\perp}$ is a totally geodesic hyperplane, then $\hat{R}(\eta,v)v=0$ for all $v\in T_{p}M$.
\end{remark}

\begin{proposition}\label{prop:ConstantCurvatureCone}
	Let $M$ be a space of constant sectional curvature $\kappa\in\R$. Then, every totally geodesic submanifold of dimension $d\ge2$ of the Riemannian cone $\hat{M}$ is a cone over a totally geodesic submanifold $S$ of $M$ if and only if $\kappa\neq1$.	
\end{proposition}	
\begin{proof}
	Due to Theorem~\ref{th:TGSubmanifoldsOfCones}, we may focus only on totally geodesic hypersurfaces.
	Let $M$ be a connected complete Riemannian manifold of constant sectional curvature $\kappa\in \R$ and dimension $n\geq 2$, and suppose that $\hat M$ admits a totally geodesic hypersurface $\Sigma$ that is not tangent to the cone vector field.
	By shrinking $\Sigma$ if necessary, we can assume that $\Sigma\subseteq \hat M$ is embedded and every $(\tau,p)\in\Sigma$ is such that $V=T_{(\tau,p)}\Sigma$ does not contain $\partial_{r}$, so its orthogonal complement must be generated by a vector of the form $\partial_{r}+\eta$ for a certain $\eta\in T_{p}M$.
	Now, from Remark~\ref{remark:SkewCones} and~\eqref{eq:curvaturecone} we deduce that
	$
		0=\hat{R}(\eta,v)v=(\kappa-1)(\lvert v\rvert^{2}\eta-\langle \eta,v \rangle v)
	$
	for all $v\in T_{p}M$, which means that either $\kappa=1$ or $\eta=0$.
	If $\kappa\neq 1$, we obtain that $V=T_{(\tau,p)}M$ for all $(\tau,p)\in \Sigma$, so $\Sigma$ is an integral manifold of the distribution $\mathcal{D}=\partial_{r}^{\perp}$ on $\hat{M}$.
	The maximal integral manifolds of $\mathcal{D}$ are precisely the links $\{\tau\}\times M$ for each $\tau>0$, so $\Sigma$ is actually an open subset of a leaf $\{\tau_{0}\}\times M$ for a certain $\tau_{0}>0$.
	However, the last equation in~\eqref{eq:ConeConnection} shows that the leaves are never totally geodesic, so we arrive at a contradiction.
	We conclude that if $\kappa\neq 1$, the inextendable totally geodesic hypersurfaces of $\hat{M}$ are precisely the cones over the complete totally geodesic hypersurfaces of $M$.
	
	If $M$ has constant sectional curvature equal to $1$, then $\hat{M}$ is flat by~\eqref{eq:curvaturecone}, so for every point $(\tau,p)\in\hat{M}$ and every hyperplane $V\subseteq T_{(\tau,p)}\hat{M}$ there exists an inextendable compatible totally geodesic hypersurface $\Sigma\subseteq \hat{M}$ such that $(\tau,p)\in \Sigma$ and $T_{(\tau,p)}\Sigma=V$.
	In particular, $\Sigma$ is not (contained in) a cone over a totally geodesic hypersurface of $M$ if and only if $(\partial_{r})_{(\tau,p)}\notin V$.
\end{proof}

\begin{example}
\label{ex:unitspherecone}
Let us assume that $M=\s{S}^n(1)$.
	Then $\hat{M}$ is isometric to $\R^{n+1}\setminus \{0\}$ in such a way that the cones over the totally geodesic submanifolds of $\hat{M}$ are of the form $V\setminus \{0\}$, where $V$ is an arbitrary vector subspace of $\R^{n+1}$.
	In particular, any affine hyperplane $\Sigma\subseteq \R^{n+1}$ not containing the origin is a totally geodesic hypersurface that does not appear as a cone over a totally geodesic hypersurface of $\s{S}^n$.
\end{example}
\begin{proposition}
\label{prop:bergercone}
Let $M$ be either $\s{S}^{3}_{\C,\tau}(r)$ or $\R\s{P}^{3}_{\C,\tau}(r)$  and let $\hat M$ denote the Riemannian cone over $M$. Then, $\hat{M}$ admits a totally geodesic hypersurface if and only if $\tau=r=1$.
\end{proposition}
\begin{proof}
	Let $M=\s{S}^{3}_{\C,\tau}(r)$ be a three-dimensional Berger sphere of radius $r$ and deformation parameter $\tau$.
	We show that $\hat{M}$ does not admit totally geodesic hypersurfaces unless $r=\tau=1$ (that is, $M$ is the unit round sphere).
	
	We first establish some notation.
	Recall that $\s{S}^{3}_{\C,\tau}(r)=\s{U}(2)/\s{U}(1)$ as a homogeneous space, and we have a reductive decomposition $\g{u}(2)=\g{u}(1)\oplus\g{p}$, where $\g{u}(1)=\R K$ and $\g{p}=\spann\{E,X,Y\}$ for the matrices
	\[
		K=iE_{11}, \quad E=\frac{i}{r\sqrt{\tau}}E_{22}, \quad X=\frac{1}{r}(E_{21}-E_{12}), \quad Y=\frac{i}{r}(E_{21}+E_{12}).
	\]
	Furthermore, if $\langle\cdot,\cdot \rangle$ denotes the inner product on $\g{p}$ induced by the Berger metric on $M$, then $E$, $X$ and $Y$ are orthonormal vectors with respect to this metric.
	In addition, the vertical and horizontal subspaces at $o=e\s{U}(1)$ with respect to the Hopf fibration are $\mathcal{V}_{o}=\R E$ and $\mathcal{H}_{o}=\spann\{X,Y\}$.
	
	Let us suppose that $M\neq \s{S}^{3}(1)$ and $\Sigma$ is a totally geodesic hypersurface of $\hat{M}$.
	We may assume that $\Sigma$ is embedded in $M$.
	Because $M$ is homogeneous, Corollary~\ref{cor:ConeIsometryGroup} allows us to suppose that $\Sigma$ passes through a point of the form $(t,o)$ with tangent space $V\subseteq T_{(t,o)}\hat{M}\equiv \R \partial_{r}\oplus \g{p}$.
	As $M$ does not admit totally geodesic hypersurfaces~\cite[Theorem~A]{RodriguezVazquezOlmos}, we have that $\partial_{r}\notin V$, so $V^{\perp}$ must be spanned by a vector of the form $\partial_{r}+\eta$, where $\eta\in \g{p}$.
	We may write $\eta=a_{1}E+a_{2}X+a_{3}Y$ for some constants $a_{1}$, $a_{2}$, $a_{3}\in\R$.
	From Remark~\ref{remark:SkewCones} we also know that $\hat{R}(\eta,Z)Z=0$ for all $Z\in \g{p}$.
	A polarization argument shows that the previous condition is equivalent to $\hat{R}(\eta,Z)W+\hat{R}(\eta,W)Z=0$ for all $Z$, $W\in\g{p}$.
	
	Firstly, suppose that $\tau\neq r^{2}$.
	Then the equations
	\[
		\begin{aligned}
			0=&{}\hat{R}(\eta,E)Y+\hat{R}(\eta,Y)E=\left(1-\frac{\tau}{r^{2}}\right)(a_{3}E+a_{1}Y), \\
			0=&{}\hat{R}(\eta,E)X+\hat{R}(\eta,X)E=\left(1-\frac{\tau}{r^{2}}\right)(a_{2}E+a_{1}X),
		\end{aligned}
	\]
	imply that $a_{1}=a_{2}=a_{3}=0$, so $\eta=0$ and $V=\g{p}$.
	However, from~\eqref{eq:coneNablaR} and the fact that $\hat{R}$ is not identically zero on $\g{p}$ we deduce that $\g{p}$ is not a totally geodesic subspace, giving us a contradiction. 
	
	Secondly, suppose that $\tau=r^{2}$ and $r\neq 1$.
	Since $0=\hat{R}(\eta,X)Y+\hat{R}(\eta,Y)X=\frac{4(r^{2}-1)}{r^{2}}(a_{3}X+a_{2}Y)$, we obtain $a_{2}=a_{3}=0$.
	As a consequence, $\eta\in \mathcal{V}_{o}$, and $V$ contains the horizontal subspace $\mathcal{H}_{o}$.
	Using~\eqref{eq:coneNablaR}, we obtain that
	$
		\left( 4-\frac{4}{r^{2}} \right)t\partial_{r}=(\hat{\nabla}_{X}\hat{R})(X,Y,Y)\in V
	$,
	so $\partial_{r}\in V$, which again yields a contradiction.
	
	All in all, we have shown that $\hat{M}$ does not admit totally geodesic hypersurfaces except in the case $M=\s{S}^{3}(1)$.
	Since the natural projection $\s{S}^{3}_{\C,\tau}(r)\to\R\s{P}^{3}_{\C,\tau}(r)$ is a Riemannian covering map, the same result holds for the three-dimensional Berger projective space.
\end{proof}

In view of the latter results, one could think that totally geodesic hypersurfaces nowhere tangent to the radial vector can only appear if the base is a sphere of radius one.
However, the following construction shows that there are many nontrivial local examples of totally geodesic hypersurfaces satisfying this condition.
\begin{example}
	\label{ex:nontrivalexconetg}
	Consider two smooth positive functions $p,q\colon \R^2\to \R$. 
	We consider the Riemannian manifold $M=(0,\pi/2)\times \R^2$ (endowed with Cartesian coordinates $x$, $y$, $z$) with the metric
	\[
	g=dx^2+(\sin x)^2 p(y,z)^2 dy^2+(\sin x)^2 q(y,z)^2dz^2.
	\]
	We see that the metric on the cone $\hat{M}=(0,\infty)\times (0,\pi/2)\times\R^2$ is now given by
	\[
	\hat{g}=dr^2+r^2dx^2+r^2(\sin x)^2 p(y,z)^2 dy^2+r^2(\sin x)^2 q(y,z)^2dz^2
	\]
	with respect to the product coordinates $(r,x,y,z)$.
	A calculation shows that the hypersurface
	\[
	\Sigma=\left\{ \left(\frac{1}{\cos x},x,y,z\right)\colon (x,y,z)\in M \right\}\subseteq \hat{M}
	\]
	is totally geodesic in $\hat{M}$.
	Note that $\Sigma$ does not arise as a cone over a totally geodesic hypersurface of $M$ because the cone vector field is nowhere tangent to it.
\end{example}

The abundance of examples highlights the appropriateness of further exploring this class of totally geodesic hypersurfaces in cones.

\section{Proofs of the main theorems}\label{sec:mainths}

In this section we provide the proofs of the main theorems of this article.
We go through each one of the homogeneous nearly Kähler $6$-manifolds with non-constant sectional curvature, and classify their totally geodesic submanifolds.

\subsection{The complex projective space}\label{subsec:cp3proof}
	
\begin{lemma}\label{lemma:CP3WellPositioned}
	If $\g{v}\subseteq \g{p}$ is a totally geodesic subspace and $\g{v}$ contains a vertical or a horizontal vector, then $\g{v}$ is well-positioned.
\end{lemma}
\begin{proof}
	If $X\in\g{v}$ is a unit vertical vector, we may assume by means of the isotropy representation that $X=e_{1}$.
	Since the spectrum of the Jacobi operator $R_{e_{1}}\colon \g{p}\ominus \R e_{1}\to \g{p}\ominus \R e_{1}$ consists of the eigenvalues $2$, with eigenspace $\R e_{2}$, and $1/8$, with eigenspace $\g{p}_{2}$, the claim follows from Remark~\ref{remark:JacobiInvariance}.
	Similarly, if $X$ is horizontal we may suppose that $X=e_{3}$, and in this case the eigenvalues of $R_{e_{3}}\colon \g{p}\ominus \R e_{3}\to \g{p}\ominus \R e_{3}$ are $1/8$, with eigenspace $\g{p}_{1}$, $1$, with eigenspace $\R e_{4}$, and $5/8$, with eigenspace $\vecspan\{e_{5},e_{6}\}$, so the result also holds in this case.
	\qedhere 
\end{proof}
	
\begin{proposition}
	\label{prop:4tgcp3}
	There are no dimension four totally geodesic submanifolds in $\C\mathsf{P}^{3}$.
\end{proposition}
\begin{proof}
	Assume, on the contrary, that there exists a totally geodesic submanifold $\Sigma$ of $M$ of dimension four passing through $o$, and let $\g{v}\subseteq \g{p}$ be its corresponding totally geodesic subspace.
	From Lemma~\ref{lemma:CP3WellPositioned} and by dimension reasons, we know that $\g{v}$ is well-positioned.
	We will distinguish three possibilities according to the dimension of $\g{v}\cap \g{p}_{1}$.
		
	If $\g{v}\cap \g{p}_{1}=\g{p}_{1}$, then $\g{v}\cap\g{p}_{2}$ is two-dimensional.
	By using the isotropy representation if necessary, we may suppose that $e_{3}\in \g{v}$.
	We can therefore consider a basis of $\g{v}$ of the form $\{e_{1},e_{2},e_{3},a_{4}e_{4}+a_{5}e_{5}+a_{6}e_{6}\}$.
	In particular, $a_{6}e_{5}-a_{5}e_{6}$ is orthogonal to $\g{v}$, and the equality
	\[
		0=\langle R(e_{1},e_{2})(a_{4}e_{4}+a_{5}e_{5}+a_{6}e_{6}),a_{6}e_{5}-a_{5}e_{6} \rangle=\frac{3(a_{5}^{2}+a_{6}^{2})}{4}
	\]
	yields $a_{5}=a_{6}=0$, so actually $\g{v}=\vecspan\{e_{1},e_{2},e_{3},e_{4}\}$.
	This is a contradiction due to the fact that $(\nabla_{e_{1}}R)(e_{1},e_{2},e_{3})=-\frac{3}{4\sqrt{2}}e_{6}\notin \g{v}$,
	so this case is not possible.
		
	If $\g{v}\cap \g{p}_{1}$ is one-dimensional (which forces $\dim \g{v}\cap\g{p}_{2}=3$), we may use the isotropy representation to assume that $\g{v}$ contains $e_{1}$ and $e_{3}$.
	In particular, $e_{5}=4\sqrt{2}(\nabla_{e_{3}}R)(e_{3},e_{1},e_{3})$ also belongs to $\g{v}$.
	As a consequence, $\g{v}$ admits a basis of the form $\{e_{1},e_{3},e_{5},a_{4}e_{4}+a_{6}e_{6}\}$, which means that $a_{6}e_{4}-a_{4}e_{6}\in \g{p}\ominus \g{v}$, and
	\[
		0=\langle R(e_{3},e_{5})(a_{4}e_{4}+a_{6}e_{6}), a_{6}e_{4}-a_{4}e_{6} \rangle=-\frac{1}{8}(a_{4}^{2}+a_{6}^{2}),
	\]
	so $a_{4}=a_{6}=0$, another contradiction.
		
	If $\g{v}\cap \g{p}_{1}=0$, then $\g{v}=\g{p}_{2}$, which is also not possible, since $(\nabla_{e_{3}}R)(e_{3},e_{4},e_{6})=-\frac{1}{4\sqrt{2}}e_{1}$ is not in $\g{v}$. In conclusion, no such $\g{v}$ can exist, and the claim follows.\qedhere
\end{proof}
	
\begin{proposition}
	\label{prop:3tgcp3}
	Let $\Sigma$ be a complete totally geodesic submanifold of $\C\s{P}^{3}$ with $\dim \Sigma = 3$.
	Then $\Sigma$ is congruent to the standard $\R\mathsf{P}^{3}_{\C,1/2}(2)$.
\end{proposition}
\begin{proof}
	Let $\Sigma$ be such a submanifold, and assume without loss of generality that $\Sigma$ passes through $o$ with tangent space $\g{v}$.
	Once again, by dimension reasons we see that $\g{v}\cap \g{p}_{2}\neq 0$, and Lemma~\ref{lemma:CP3WellPositioned} implies that $\g{v}$ is well-positioned.
	We consider three cases according to the dimension of $\g{v}\cap\g{p}_{1}$.
		
	If $\g{v}\cap \g{p}_{1}=0$, then $\g{v}\subseteq \g{p}_{2}$ is a hyperplane, and since the isotropy representation is transitive on the unit sphere of $\g{p}_{2}$, we may assume that $\g{v}=\vecspan\{e_{4},e_{5},e_{6}\}$.
	However, since $R(e_{4},e_{5})e_{6}=\frac{1}{8}e_{3}$, we obtain a contradiction.
		
	If $\g{v}\cap \g{p}_{1}$ is one-dimensional, then by using the isotropy representation we may suppose that $\g{v}$ contains $e_{1}$ and $e_{3}$.
	Note that $4\sqrt{2}(\nabla_{e_{3}}R)(e_{3},e_{1},e_{3})=e_{5}$ also belongs to $\g{v}$, which gives $\g{v}=\vecspan\{e_{1},e_{3},e_{5}\}=\g{p}_{\R\mathsf{P}^{3}_{\C,1/2}(\sqrt{2})}$.
	Therefore, in this case we obtain $\Sigma=\R\mathsf{P}^{3}_{\C,1/2}(\sqrt{2})$.
	
	Finally, if $\g{v}\cap \g{p}_{1}=\g{p}_{1}$, then by using the isotropy representation we can assume that $\g{v}=\vecspan\{e_{1},e_{2},e_{3}\}$.
	This is not possible, since $R(e_{1},e_{2})e_{3}=\frac{3}{4}e_{4}$ is not in $\g{v}$.
	This finishes the proof.
\end{proof}
\begin{proposition}
	\label{prop:2tgcp3}
	Let $\Sigma$ be a complete totally geodesic surface inside $\C\mathsf{P}^{3}$.
	Then $\Sigma$ is congruent to one of the spheres described in Table~\ref{table:CP3}.
\end{proposition}
\begin{proof}
	Suppose that $\g{v}\subseteq \g{p}$ is a totally geodesic plane, and consider the corresponding complete totally geodesic submanifold $\Sigma$ of $M$.
	Notice that $\Sigma$ must be intrinsically homogeneous, and thus a space of constant curvature since it is of dimension two.
	Also, note that either $\g{v}$ is completely contained in one of the irreducible $\mathsf{K}$-submodules of $\g{p}$ or it contains a vector that projects nontrivially on $\g{p}_{1}$ and $\g{p}_{2}$ at the same time.
	Since the case $\g{v}=\g{p}_{1}$ already corresponds to $\Sigma$ being the fiber of the twistor fibration, we may skip this case.
		
	Assume that $\g{v}\subseteq \g{p}_{2}$.
	Using the isotropy representation if necessary, we can suppose that $e_{3}\in \g{v}$.	
	One sees that the kernel of the Cartan operator $C_{X}$ is spanned by $e_{3}$ and $e_{4}$, so we must have $\g{v}=\vecspan\{e_{3},e_{4}\}=\g{p}_{\mathsf{SU}(2)\cdot o}$, since $\Sigma$ has constant curvature, which means that $\Sigma=\mathsf{SU}(2)\cdot o$.
	
	Finally, suppose that there exists a vector $X\in \g{v}$ such that $X_{\g{p}_{1}}$ and $X_{\g{p}_{2}}$ are nonzero.
	By using the isotropy representation and rescaling, we can assume that $X=e_{1}+ \lambda e_{3}$ for a certain $\lambda>0$.
	In this case, $\ker C_{X}$ is spanned by $X$ and $Y=3 \lambda e_{2}+(6-\lambda^{2})e_{4}$, so necessarily $\g{v}=\vecspan\{X,Y\}$, since $\Sigma$ has constant curvature.
	In particular, we have $0=4 \sqrt{2}\langle (\nabla_{X}R)(Y,X,Y),e_{5} \rangle = -\lambda  \left(2 \lambda ^4+3 \lambda^2-9\right)$,
	and this is only possible if $\lambda=\sqrt{3/2}$.
	Therefore, $\g{v}=\vecspan\{\sqrt{2}e_{1}+\sqrt{3}e_{3},\sqrt{2}e_{2}+\sqrt{3}e_{4}\}=\g{p}_{\mathsf{SU}(2)_{\Lambda_{3}}\cdot o}$.
	As a consequence, we see that in this case $\Sigma = \mathsf{SU}(2)_{\Lambda_{3}}\cdot o$.
	This finishes the proof. \qedhere
\end{proof}
	
\begin{proof}[Proof of Theorem~\ref{th:tg-cp3-classification}]
	The theorem follows from combining Theorem~\ref{th:hyp}, Proposition~\ref{prop:4tgcp3}, Proposition~\ref{prop:3tgcp3}, and Proposition~\ref{prop:2tgcp3}.
\end{proof}
	 
\subsection{The flag manifold}\label{subsec:flagproof}
\begin{proposition}
	\label{prop:4tgflag}
	The flag manifold $\mathsf{F}(\C^{3})$ does not admit any codimension two totally geodesic submanifolds.
\end{proposition}
\begin{proof}
	Suppose that $\mathsf{F}(\C^{3})$ admits a totally geodesic submanifold of dimension four.
	This means that there exists a totally geodesic subspace $\g{v}\subseteq\g{p}$ with $\dim \g{v}=4$.
	By a dimension argument, one sees that the intersection $\g{v}\cap(\g{p}_{1}\oplus\g{p}_{2})$ is nontrivial, and using both the isotropy representation of $\mathsf{T}^2$ and conjugating by a permutation matrix if necessary, we may suppose that $\g{v}$ admits a nonzero vector of the form $X=e_{1}+\lambda e_{3}$, where $\lambda\in \R$ is a nonnegative number.
	The Cartan operator $C_{X}$ is diagonalizable with eigenvalues $0$, $\frac{3\lambda\sqrt{1+\lambda^{2}}}{2\sqrt{2}}$ and $-\frac{3\lambda\sqrt{1+\lambda^{2}}}{2\sqrt{2}}$,
	and corresponding eigenspaces
	\[
		\vecspan\{e_{1},e_{3},e_{5},\lambda e_{2}+e_{4}\},\quad \R\left(e_{2}-\lambda e_{4}+\sqrt{1+\lambda^{2}}e_{6}\right), \quad \R\left(-e_{2}+\lambda e_{4}+\sqrt{1+\lambda^{2}}e_{6}\right).
	\]
		
	First, assume that $\lambda > 0$, so the three eigenvalues given above are pairwise distinct.
	We prove that $\g{v}$ coincides with the kernel of the Cartan operator $C_{X}$.
		
	If $C_{X}\vert_{\g{v}}$ is not identically zero, then $\g{v}$ contains a vector of the form $Y=\varepsilon e_{2}-\lambda \varepsilon e_{4}+\sqrt{1+\lambda^{2}}e_{6}$, where $\varepsilon \in \{\pm 1\}$.
		
	If $\lambda \neq 1$, then we can construct a basis of $\g{v}$ with the vectors
	\begin{align*}
		X={}&e_{1}+\lambda e_{3}, \\
		Y={}&\varepsilon e_{2}-\lambda \varepsilon e_{4}+\sqrt{1+\lambda^{2}}e_{6}, \\
		U={}&8R(X,Y)X=-2\varepsilon \left(5\lambda^{2}+8\right)e_{2}+2\varepsilon\lambda\left(8\lambda^{2}+5\right)e_{4}-\left(1+\lambda^{2}\right)^{3/2}e_{6}, \\
		V={}&8\sqrt{2}(\nabla_{X}R)(X,Y,Y)=-3 \lambda  \sqrt{\lambda ^2+1} \left(3 \lambda ^2+5\right) \varepsilon e_{1}-3 \sqrt{\lambda ^2+1} \left(5 \lambda ^2+3\right) \varepsilon e_{3}\\
		&-6 \lambda  \left(\lambda ^2-1\right)e_{5}.
	\end{align*}
	Therefore, the vector $T=-\lambda  \sqrt{\lambda ^2+1} \left(5 \lambda ^2+3\right) e_{2}-\sqrt{\lambda ^2+1} \left(3 \lambda ^2+5\right) e_{4}+2 \lambda  \left(\lambda ^2-1\right) \varepsilon e_{6}$ is orthogonal to $\g{v}$. 
	Now, we see that $0=\langle R(X,U)X,T \rangle=36 \varepsilon\lambda ^3 \left(\lambda ^2-1\right) \sqrt{\lambda ^2+1}$, which is a contradiction.
	
	If $\lambda = 1$, the equation $R(X,Y)X+\frac{13}{4}Y=3\sqrt{2}e_{6}$ implies that the vectors $X=e_{1}+e_{3}$, $Z=e_{2}-e_{4}$, and $T=e_{6}$ are in $\g{v}$.
	We can therefore complete $X$, $Z$ and $T$ to a basis of $\g{v}$ by adding a vector of the form $U=c_{1}e_{1}+c_{2}e_{2}-c_{1}e_{3}+c_{2}e_{4}+c_{3}e_{5}$,
	where $c_i\in\R$ for each $i\in\{1,2,3\}$.
	In particular, we have that the vector $-c_{2}e_{1}+c_{1}e_{2}+c_{2}e_{3}+c_{1} e_{4}$ is orthogonal to $\g{v}$, so $0=2\langle R(X,Z)U,-c_{2}e_{1}+c_{1}e_{2}+c_{2}e_{3}+c_{1} e_{4}  \rangle=-9(c_{1}^{2}+c_{2}^{2})$, which
	forces $c_{1}=c_{2}=0$.
	Therefore $\g{v}=\spann\{X,Z,e_{5},e_{6}\}$.
	However, we have $8 R(X,e_{5})e_{6}=3(e_{2}+e_{4})\notin\g{v}$,	which yields a contradiction.
		
	From all of the above, we see that $\g{v}$ must coincide with $\ker C_{X}=\vecspan\{e_{1},e_{3},e_{5},\lambda e_{2}+e_{4}\}$.
	However, this is also not possible, since $\frac{8}{3}R(e_{1},e_{3})(\lambda e_{2}+e_{4})=-e_{2}+\lambda e_{4}\notin \g{v}$.
	Thus, the case $\lambda > 0$ is not possible.
		
	Now, suppose that $\lambda = 0$, which implies that $e_{1}\in \g{v}$.
	Since the intersection of $\g{v}$ with $\g{p}_{2}\oplus \g{p}_{3}$ is at least two-dimensional, we can use the isotropy representation to assume that there is a tangent vector of the form $Y=e_{3}+\mu e_{5}$, where $\mu \in \R$.
	However, $\mu=0$, for if $\mu \neq 0$, by means of the full isotropy representation we can conjugate $\g{v}$ to a new totally geodesic subspace containing a tangent vector of the form $e_{1}+\mu e_{3}$, and we may use the previous case to derive a contradiction.
	Now this yields that $e_{1}+e_{3}$ is also in $\g{v}$, which is yet another contradiction.
	We conclude that this case is also not possible, and therefore there are no codimension two totally geodesic subspaces in $\g{p}$. \qedhere
\end{proof}
	
\begin{proposition}
	\label{prop:3tgflag}
    Let $\Sigma$ be a complete three-dimensional totally geodesic submanifold of the flag manifold $\s{F}(\C^3)$.
	Then $\Sigma$ is congruent to either $\mathsf{F}(\R^{3})$ or to $\mathsf{S}^{3}_{\C,1/4}(\sqrt{2})$.
\end{proposition}
	
\begin{proof}
	We need to classify three-dimensional totally geodesic subspaces of $\g{p}$.
	Let $\g{v}\subseteq \g{p}$ be such a subspace, and $\Sigma$ the corresponding complete totally geodesic submanifold.
	Then we know that there exists a vector $X\in\g{v}\cap(\g{p}_{1}\oplus\g{p}_{2})$, and using the full isotropy representation, we may assume that it is of the form $X=e_{1}+\lambda e_{3}$ for a certain $\lambda \geq 0$.
	The Cartan operator $C_{X}$ is diagonalizable with eigenvalues $0$, $\frac{3\lambda\sqrt{1+\lambda^{2}}}{2\sqrt{2}}$, and $-\frac{3\lambda\sqrt{1+\lambda^{2}}}{2\sqrt{2}}$, and corresponding eigenspaces
	\[
		\vecspan\{e_{1},e_{3},e_{5},\lambda e_{2}+e_{4}\},\quad \R\left(e_{2}-\lambda e_{4}+\sqrt{1+\lambda^{2}}e_{6}\right), \quad \R\left(-e_{2}+\lambda e_{4}+\sqrt{1+\lambda^{2}}e_{6}\right).
	\]
		
	First, assume that $\lambda > 0$, so the three eigenvalues given above are pairwise distinct.
	We prove in this case that either $\Sigma=\mathsf{SU}(2)_{(1,0,1)}\cdot o$ or $\g{v}\subseteq\ker C_{X}$.
	Indeed, if there is a vector of the form $Y=\varepsilon e_{2}-\lambda \varepsilon e_{4}+\sqrt{1+\lambda^{2}}e_{6}$ in $\g{v}$, where $\varepsilon\in\{\pm 1\}$, then we may construct a basis of $\g{v}$ by adding the vector
	\begin{equation*}
		Z=8 R(X,Y)X=-2 \left(5 \lambda ^2+8\right) \varepsilon e_{2}+2 \lambda 
		\left(8 \lambda ^2+5\right) \varepsilon e_{4}-\left(\lambda
		^2+1\right)^{3/2}e_{6}\in \g{v}.
	\end{equation*}
	In particular, $-\lambda e_{1}+e_{3}$ is orthogonal to $\g{v}$, and we must have $0=4\langle R(X,Y,Y),-\lambda e_{1}+e_{3} \rangle=3\lambda(\lambda^{2}-1)$.
	This forces $\lambda = 1$.
	Note that $4R(X,Y)X=-13\varepsilon e_{2}+13 \varepsilon e_{4}-\sqrt{2}e_{6}\in \g{v}$.
	This means that $\g{v}$ is spanned by $e_{1}+e_{3}$, $e_{2}-e_{4}$ and $e_{6}$, so $\Sigma$ coincides with $\mathsf{SU}(2)_{(1,0,1)}\cdot o$.
	Now, suppose that $\g{v}\subseteq \ker C_{X}$.
	Then either $\g{v}=\vecspan\{e_{1},e_{3},e_{5}\}$ (which yields $\Sigma=\mathsf{F}(\R^{3})$) or using the isotropy representation we can find a basis of $\g{v}$ given by the vectors of the form
	\[
		X={}e_{1}+\lambda e_{3}, \quad
		Y={}a_{1}e_{1}+a_{3}e_{3}+a_{5}e_{5}+(\lambda e_{2}+e_{4}), \quad
		Z={}c_{1}e_{1}+c_{3}e_{3}+c_{5}e_{5},
	\]
	for some constants $a_{i}$, $c_{j}\in\R$.
	In particular, the vectors $e_{6}$ and $e_{2}-\lambda e_{4}$ are in $\g{p}\ominus\g{v}$.
	Now, we also see that $0=\langle R(X,Y)X,e_{2}-\lambda e_{4}\rangle=\frac{3\lambda(\lambda^{2}-1)}{4}$, which means that $\lambda = 1$.
	On the other hand, we have
	\begin{align*}
		\langle (\nabla_{X}R)(X,Y,Y),e_{6} \rangle={}&\frac{3(a_{3}-a_{1})}{4\sqrt{2}},&
		\langle (\nabla_{X}R)(X,Y,Y),e_{2}-e_{4} \rangle={}&\frac{-3 a_{5}}{2\sqrt{2}}, \\
		\langle (\nabla_{X}R)(X,Z,Y),e_{6} \rangle={}&\frac{3(c_{3}-c_{1})}{4\sqrt{2}},&
		\langle (\nabla_{X}R)(X,Z,Y),e_{2}-e_{4} \rangle={}&\frac{-3 c_{5}}{2\sqrt{2}}.
		\end{align*}
	Since all of these inner products are zero, we deduce that $a_{5}=c_{5}=0$, $a_{1}=a_{3}$ and $c_{1}=c_{3}$.
	In particular, $Z$ and $X$ are proportional, a contradiction.
		
	We now assume $\lambda = 0$, so $e_{1}\in \g{v}$.
	Since $\g{p}\cap(\g{p}_{2}\oplus\g{p}_{3})$ is nonzero, we may use the isotropy representation to suppose that a vector of the form $e_{3}+\mu e_{5}$ belongs to $\g{v}$.
	Note that if $\mu \neq 0$, then using an element of the full isotropy group permuting the factors of the isotropy representation, we can carry this setting to the one in the previous paragraph, so we may assume that $\mu = 0$, and thus $e_{3}\in \g{v}$.
	As a consequence, $e_{1}+e_{3}\in \g{v}$, and we may use the arguments in the case $\lambda > 0$ to derive the same conclusions, and the proposition is proved. \qedhere
\end{proof}
\begin{proposition}
	\label{prop:2tgflag}
	Every complete totally geodesic surface of $\mathsf{F}(\C^{3})$ is congruent to one of the following:
		\begin{enumerate}[\rm (i)]
			\item the totally geodesic $\mathsf{T}^2_{\Lambda}$,
			\item the  Berger sphere $\s{S}^3_{\C,1/4}(\sqrt{2})$,
			\item the fiber $\C\s{P}^1$,
			\item or a totally geodesic $\R\mathsf{P}^{2}\left(2\sqrt{2}\right)\subseteq \mathsf{F}(\R^{3})$.
		\end{enumerate}
\end{proposition}
\begin{proof}
	Let $\g{v}$ be a totally geodesic subspace of dimension $2$ in $\g{p}$.
	We define $r\in\{1,2,3\}$ to be the largest number such that there exists a vector $X\in \g{v}$ that has nontrivial projection onto $r$ of the irreducible submodules of $\g{p}$.
		
	\textit{The case} $r=3$.
	Let $X\in \g{v}$ be a vector that projects nontrivially onto each of the submodules $\g{p}_{i}$.
	Then, by means of the isotropy representation, we know that $X$ is (up to $\mathsf{T}^2$-conjugacy and scaling) of the form
	$X=e_{1}+a_{3}e_{3}+a_{5}e_{5}+a_{6}e_{6}$,
	where $a_{3}$, $a_{5}^{2}+a_{6}^{2}\neq 0$.
		
	First, suppose that $a_{6}\neq 0$.
	Then, one sees that the kernel of the Cartan operator $C_{X}$ is spanned by $X$ and
	$Y=a_{3}a_{5}e_{1}+a_{3}a_{6}e_{2}+a_{5}e_{3}+a_{6}e_{4}+a_{3}e_{5}$.
	As a consequence, if $X$ is tangent to a totally geodesic surface $\Sigma$, then its tangent space $\g{v}=T_{o}\Sigma$ is precisely $\g{v}=\vecspan\{X,Y\}$.
	In particular, since $\Sigma$ is intrinsically homogeneous and therefore of constant sectional curvature, we have the equations
	\begin{equation*}
		\begin{aligned}
			0=&{}\langle (\nabla_{X}R)(X,Y,Y),e_{4} \rangle=\frac{3 a_{3}^{2} a_{6}(1-a_{5}^{2}-a_{6}^{2})}{8\sqrt{2}},\\
			0=&{}\langle (\nabla_{X}R)(X,Y,Y),e_{5} \rangle=\frac{3 a_{3} a_{6}^{2}(a_{3}^{2}-1)}{8\sqrt{2}},
		\end{aligned}
	\end{equation*}
	which imply $a_{3}\in \{\pm 1\}$ and $a_{5}^{2}+a_{6}^{2}=1$.
	Therefore, we may rewrite $\g{v}$ as the span of
	\begin{equation*}
		X=e_{1}+\varepsilon e_{3}+\cos \phi e_{5}+ \sin \phi e_{6}, \quad
		Y=\varepsilon \cos \phi e_{1}+\varepsilon \sin \phi e_{2}+\cos \phi e_{3}+\sin \phi e_{4}+\varepsilon e_{5},
	\end{equation*}
	where $\varepsilon\in\{\pm 1\}$ and $\sin \phi \neq 0$.
	It turns out that $\g{v}=\g{z}_{\g{p}}(X)$ is the centralizer of $X$ in $\g{p}$, and in particular it is (maximal) abelian.
	If $\varepsilon=1$, then the element $k=\operatorname{diag}(e^{i\phi/3},1,e^{-i\phi/3})\in \s{T}^2$
	carries the subspace $\g{p}_{\mathsf{T}^2_{\Lambda}}=\vecspan\{e_{1}+e_{3}+e_{5},e_{2}+e_{4}-e_{6}\}$ to $\g{z}_{\g{p}}(X)=\g{v}$, which means that $\Sigma$ is congruent to $\mathsf{T}^2_{\Lambda}$.
	Similarly, if $\varepsilon=-1$, the element $k=\operatorname{diag}(e^{i\phi/3},e^{i\pi/3},e^{-i(\phi+\pi)/3})\in\s{T}^{2}$
	carries $\g{p}_{\mathsf{T}^2_{\Lambda}}$ to $\g{z}_{\g{p}}(X)=\g{v}$, and thus $\Sigma$ is congruent to $\mathsf{T}^2_{\Lambda}$.
	
	Now, assume that $a_{6}=0$.
	In this case, the kernel of the Cartan operator $C_X$ is equal to $\g{so}(3)=\g{p}_{\s{F}(\C^3)}$.
	Thus, the only totally geodesic surfaces containing $X$ are projective planes contained in $\mathsf{F}(\R^{3})$.
	
	\textit{The case} $r=2$.
	We can find a vector $X\in \g{v}$ that projects nontrivially onto two of the three irreducible submodules.
	By using the full isotropy representation and rescaling, we can assume that $X=e_1+\lambda e_3$ for a certain $\lambda > 0$.
	Take any vector $Y\in \g{v}\ominus \R X$.
		
	On  the one hand if, $D_X Y=0$, then $Y\in \ker D_{X}\ominus \R X=\R(e_{2}+\lambda e_{4})$, so we may directly assume that $Y=e_{2}+\lambda e_{4}$, and we have $\g{v}=\vecspan\{X,Y\}$.
	In particular, observe that $-\lambda e_{2}+e_{4}$ is orthogonal to $\g{v}$, and the condition $0=\langle R(X,Y)X,-\lambda e_{2}+e_{4} \rangle=3\lambda(\lambda^{2}-1)$	forces $\lambda = 1$.
	Thus, $\g{v}=\vecspan\{e_{1}+e_{3},e_{2}+e_{4}\}$, so $\Sigma = \mathsf{SO}(3)^{\sigma}\cdot o$.
		
	On the other hand, if $D_X Y\neq 0$, by Proposition~\ref{prop:TojoSurfaces}, the vectors $D_X^k Y$ (for $k\geq 0$) must lie in a common eigenspace of $R_{X}$ (of dimension greater than one because $Y$ and $D_XY$ are orthogonal) and in the kernel of $C_{X}$.
		
	Moreover, the spectrum of the Jacobi operator $R_{X}$ consists of the (pairwise distinct) eigenvalues
	\begin{equation*}
		0, \quad \frac{\lambda^{2}+1}{8}, \quad \frac{17+17 \lambda ^2+3
			\sqrt{25 \lambda ^4-14 \lambda^2+25}}{16},\quad\frac{17+17 \lambda ^2-3
			\sqrt{25 \lambda ^4-14 \lambda
			^2+25}}{16},		
	\end{equation*}
	and the only eigenspace of dimension greater than one is that of $\frac{\lambda^{2}+1}{8}$, which actually is the direct sum $\R(-\lambda e_{1}+e_{3})\oplus \g{p}_{3}$.
	On the other hand, the kernel of the Cartan operator $C_{X}$ is $\vecspan\{e_{1},e_{3},e_{5},\lambda e_{2}+e_{4}\}$.
	Thus, $Y\in \vecspan\{-\lambda e_{1}+e_{3},e_{5}\}$, and in particular $\g{v}$ is contained in $\g{so}(3)$, so $\Sigma$ is contained in $\mathsf{F}(\R^{3})$.
		
	\textit{The case} $r=1$.
	Here, we simply have $\g{p}=\g{p}_{k}$ for $k\in \{1,2,3\}$, so actually $\Sigma$ is congruent to the fiber $\C\mathsf{P}^{1}$ of the submersion $\mathsf{F}(\C^{3})\to \C \mathsf{P}^{2}$. \qedhere
\end{proof}

\begin{proof}[Proof of Theorem~\ref{th:tg-fc3-classification}]
	The result now follows by combining Theorem~\ref{th:hyp}, Proposition~\ref{prop:4tgflag}, Proposition~\ref{prop:3tgflag}, and Proposition~\ref{prop:2tgflag}.
\end{proof}
	
\subsection{The almost product $\mathsf{S}^{3}\times\mathsf{S}^{3}$}\label{subsec:s3s3proof}
	
\begin{lemma}\label{lemma:S3S3WellPositioned}
	Let $\g{v}\subseteq \g{p}$ be a totally geodesic subspace.
	If $\g{v}$ contains a nonzero vertical or horizontal vector, then $\g{v}$ is well-positioned.
\end{lemma}
\begin{proof}
	Suppose that $\g{v}$ contains a nonzero vector $X\in \g{p}_{1}$ (respectively, $X\in \g{p}_{2}$).
	Since the isotropy representation of $\Delta\mathsf{SU}(2)$ is transitive on the unit sphere of $\g{p}_{1}$ (respectively, $\g{p}_{2}$), we may suppose that $X=e_{1}$ (respectively, $X=e_{4}$).
	Note that the matrices of $R_{e_{1}}$ and $R_{e_{4}}$ are given by $R_{e_{1}}=\operatorname{diag}(0,3/4,3/4,0,1/12,1/12)$ and $R_{e_{4}}=\operatorname{diag}(0,1/12,1/12,0,3/4,3/4)$, which means that $\g{v}=(\g{v}\cap\R e_{1})\oplus (\g{v}\cap \vecspan\{e_{2},e_{3}\})\oplus (\g{v}\cap \R e_{4})\oplus (\g{v}\cap \vecspan\{e_{5},e_{6}\})$ in both cases.
	This last equation implies directly that $\g{v}$ is well-positioned. \qedhere
\end{proof}
	
\begin{proposition}
	\label{prop:4tgs3s3}
	The space $M=\mathsf{S}^{3}\times\mathsf{S}^{3}$ does not admit any codimension two totally geodesic submanifolds.
\end{proposition}
\begin{proof}
	Suppose on the contrary that there exists a four-dimensional totally geodesic submanifold $\Sigma$ of $M$, and without loss of generality assume that $\Sigma$ passes through $o$ with tangent space $\g{v}$.
	A dimension argument yields that $\g{v}\cap \g{p}_{1}$ is nonzero.
	Since the isotropy representation is transitive on the unit sphere of $\g{p}_{1}$, we can further assume that $e_{1}\in \g{v}$.
	Because $\g{v}$ is $R_{e_{1}}$-invariant, using the eigenspace decomposition of $R_{e_{1}}$ (obtained in the proof of the previous lemma) we get $\g{v}=(\g{v}\cap\spann\{e_1,e_4\})\oplus(\g{v}\cap\spann\{e_2,e_3\})\oplus(\g{v}\cap\spann\{e_5,e_6\})$.
	In particular it follows that either $e_{4}\in \g{v}$ or $e_{4}\in\g{p}\ominus \g{v}$. 
		
	Firstly, suppose that $e_{4}\in \g{v}$.
	If $e_{2}$ and $e_{3}$ are also tangent to $\g{v}$, then  $\g{v}=\vecspan\{e_{1},e_{2},e_{3},e_{4}\}$, which is a contradiction because $12 R(e_{1},e_{2})e_{4}=-5e_{5}\notin \g{v}$.
	Similarly, if $e_{5}$ and $e_{6}$ are tangent to $\Sigma$ we deduce that $\g{v}=\vecspan\{e_{1},e_{4},e_{5},e_{6}\}$, which is also not possible, as in that case we would have $6R(e_{1},e_{4})e_{5}=e_{2}\notin \g{v}$.
	Thus, we see that $\dim \g{v}\cap \vecspan\{e_{2},e_{3}\}=\dim \g{v}\cap \vecspan\{e_{5},e_{6}\}=1$.
	Conjugating by an adequate element in $\mathsf{K}$, we can assume that $\g{v}\cap \vecspan\{e_{2},e_{3}\}=\R e_{2}$.
	As a consequence, $e_{5}=-6R(e_{1},e_{4})e_{2}\in \g{v}$, and we obtain $\g{v}=\vecspan\{e_{1},e_{2},e_{4},e_{5}\}$.
	However, the equality $6\sqrt{3}(\nabla_{e_{1}}R)(e_{1},e_{4},e_{2})=e_{6}$ implies that $\g{v}$ is not $\nabla R$-invariant, which yields a contradiction once again.
	Therefore, $e_{4}$ is not tangent to $\Sigma$, so it must be normal to $\g{v}$.
		
	We suppose that $e_{1}\in \g{v}$ and $e_{4}\in \g{p}\ominus \g{v}$.
	By dimensional reasons, and using the isotropy representation if necessary, there are $\lambda,c_{5},c_{6}\in \R$ such that $e_2 +\lambda e_3,c_{5} e_5 + c_{6} e_6\in \g{v}$. Moreover,
	\begin{align*}
		0&=\langle R(e_1,e_2 +\lambda e_3)(c_{5}e_5 + c_{6} e_6),e_4\rangle=\frac{5}{12}\left(c_{5}+\lambda c_{6}\right),\\
		0&=\langle (\nabla_{e_1} R)(e_1, e_2 +\lambda e_3, c_{5} e_5 + c_{6} e_6),e_4\rangle=\frac{5}{12\sqrt{3}}\left(\lambda c_{5}-c_{6} \right),
		\end{align*}
	which implies $c_{5}=c_{6}=0$, a contradiction.
We conclude that the existence of $\Sigma$ is not possible. \qedhere
\end{proof}
\begin{proposition}\label{prop:3tgs3s3}
	Let $\Sigma\to\mathsf{S}^{3}\times\mathsf{S}^{3}$ be a three-dimensional complete totally geodesic submanifold. 
	Then $\Sigma$ is congruent to either the round sphere $\mathsf{S}^{3}$ (viewed as the first factor) or the Berger sphere $\mathsf{S}^{3}_{\C,1/3}(2)$.
\end{proposition}
\begin{proof}
	Suppose that $\Sigma$ is such a submanifold, and assume without loss of generality that $\Sigma$ passes through $o$ with tangent space $\g{v}$.

	We start by assuming that $\g{v}$ is well-positioned, so $\g{v}=(\g{v}\cap\g{p}_{1})\oplus(\g{v}\cap\g{p}_{2})$.
	We consider several cases according to the dimension of $\g{v}\cap \g{p}_{1}$.
		
	If $\g{v}\cap\g{p}_{1}=0$, then $\g{v}=\g{p}_{2}$.
	However, the equation $3\sqrt{3}(\nabla_{e_{4}}R)(e_{4},e_{5},e_{4})=e_{3}$ implies that $\g{v}$ is not $\nabla R$-invariant, a contradiction.
	
	If $\g{v}\cap \g{p}_{1}$ is a one-dimensional subspace, we may use the isotropy representation to assume that this intersection is spanned by $e_{1}$.
	Now, since $\g{v}\cap\g{p}_{2}$ is two-dimensional, it must intersect $\vecspan\{e_{5},e_{6}\}$, and we may use an element of $\mathsf{K}$ that fixes $e_{1}$ if necessary to assume that $e_{5}\in \g{v}$.
	As a consequence, $e_{6}=3\sqrt{3}(\nabla_{e_{5}}R)(e_{5},e_{1},e_{5})$ is also in $\g{v}$, so $\g{v}=\vecspan\{e_{1},e_{5},e_{6}\}$, which implies that $\Sigma=\mathsf{S}^{3}_{\C,1/3}(2)$ is the Berger sphere.
	
	If $\g{v}\cap \g{p}_{1}$ is two-dimensional, then $\g{v}\cap \g{p}_{2}$ is one-dimensional, and we can assume that it is spanned by~$e_{4}$.
	Arguing in an identical fashion as above, we can also assume that $e_{2}\in \g{v}$, and therefore we have $e_{6}=3\sqrt{3}(\nabla_{e_{4}} R)(e_{2},e_{4},e_{4})\in \g{v}$, but this contradicts our hypothesis that $\g{v}\cap\g{p}_{2}$ is one-dimensional.
		
	Lastly, if $\g{v}\cap \g{p}_{1}=\g{p}_{1}$, then we actually have $\g{v}=\g{p}_{1}$, so $\Sigma$ is simply the fiber of the fibration $\mathsf{S}^{3}\times\mathsf{S}^{3}\to\mathsf{S}^{3}$.
		
	Now, suppose that $\g{v}$ is not well-positioned, so that $\g{v}\cap \g{p}_{1}=\g{v}\cap \g{p}_{2}=0$ by Lemma~\ref{lemma:S3S3WellPositioned}.
	We start by proving (up to the isotropy representation) that $\g{v}$ admits a vector of the form $X=e_{1}+\lambda e_{4}$ for a certain $\lambda \in \R\setminus\{0\}$.
		
	Firstly, note that since $\dim \g{v}=3$, the intersection $\g{v}\cap \vecspan\{e_{1},e_{4},e_{5},e_{6}\}$ is nontrivial.
	Conjugating by an adequate element of $\mathsf{K}$ and rescaling, we can assume that a vector of the form $X=e_{1}+\rho \cos \theta e_{4}+\rho \sin \theta e_{5}$ is in $\g{v}$, where $\rho\geq 0$ and $\theta \in [0,2\pi)$.
	The condition $\g{v}\cap \g{p}_{1}=\g{v}\cap \g{p}_{2}=0$ forces $\rho > 0$, and the orthogonal projection maps $\g{v}\to \g{p}_{i}$ are vector space isomorphisms.
	Note that if $\theta \in \{0,\pi\}$, then $X=e_{1}\pm \rho e_{4}$, and our assertion is proved.
	On the other hand, if $\theta\notin\{0,\pi\}$, we define polynomials
	\[
		\begin{aligned}
			f(x)={}&\frac{1}{432} \left(-\rho ^2+12
			x-1\right) \left(-4 \rho ^2 \cos (2
			\theta )+\rho ^2 (4-27 x)+9 x (4
			x-3)\right), \\
			g(x)={}&\frac{1}{144} \left(-32 \rho^2 \cos (2\theta )+9 \rho ^4+50 \rho ^2+144x^2-120 \left(\rho ^2+1\right)x+9\right).
		\end{aligned}
	\]
	One sees that the product $fg$ is precisely the characteristic polynomial of $R_{X}\in\operatorname{End}(\g{p}\ominus \R X)$, so in particular we obtain that $f(R_{X})g(R_{X})=0$.
	Furthermore, $f$ and $g$ are relatively prime.
	As a consequence, because $\g{v}$ is an $R_{X}$-invariant subspace, it may be decomposed as $\g{v}=\R X \oplus (\g{v}\cap \ker f(R_{X}))\oplus (\g{v}\cap \ker g(R_{X}))$.
	It turns out that
	\[
		\ker f(R_{X})=\spann \{e_{2},\rho \sin \theta e_{1}-e_{5},\rho \cos \theta e_{1}-e_{4}\}, \quad
		\ker g(R_{X})=\spann \{e_{3},e_{6}\},
	\]
	and as the projections of $\g{v}$ onto $\R e_{3}$ and $\R e_{6}$ are both nontrivial, we obtain that there is a vector in $\g{v}$ of the form $e_{3}+\lambda e_{6}$ (where $\lambda \neq 0$), and by conjugating via the isotropy representation we can change $\g{v}$ so that $e_{1}+\lambda e_{4}\in \g{v}$.
		
	We now let $Y=e_{1}+\lambda e_{4}\in \g{v}$, where $\lambda \in \R\setminus \{0\}$.
	The Jacobi operator $R_{Y}\in\operatorname{End}(\g{p}\ominus \R Y)$ has three different eigenvalues $0$, $\frac{1+\lambda^{2}}{12}$ and $\frac{3(1+\lambda^{2})}{4}$,	with corresponding eigenspaces $ \R(\lambda e_{4}-e_{1})$, $\vecspan\{\lambda e_{2}-e_{5},\lambda e_{3}-e_{6}\}$ and $\vecspan\{e_{2}+\lambda e_{5},e_{3}+\lambda e_{6}\}$.
	Note that the isotropy subgroup $\mathsf{K}_{e_{1}}$ fixes $Y$ and acts transitively on the set of lines in both $\vecspan\{\lambda e_{2}-e_{5},\lambda e_{3}-e_{6}\}$ and $\vecspan\{e_{2}+\lambda e_{5},e_{3}+\lambda e_{6}\}$.
	Now, the fact that $\dim \g{v}=3$ implies that $\g{v}$ must contain a nonzero eigenvector from one of the last two eigenspaces.
	Our next step is to prove that $\lambda^{2}=3$ or $\lambda^{2}=1/3$.
		
	If $\g{v}\cap\vecspan\{\lambda e_{2}-e_{5},\lambda e_{3}-e_{6}\} \neq 0$, we can conjugate by an element of $\mathsf{K}_{e_{1}}$ to assume that $Z=\lambda e_{2}-e_{5}\in \g{v}$ is tangent to $\Sigma$.
	If we suppose that $\lambda^{2}\notin \left\{3,1/3\right\}$, then we can give a basis of $\g{v}$ with the vectors
	\begin{equation*}
		Y=e_{1}+\lambda e_{4}, \quad
		Z=\lambda e_{2}-e_{5}, \quad
		T=3\sqrt{3} (\nabla_{Z}R)(Z,Y,Z)=\lambda(1-3\lambda^{2})e_{3}+(3\lambda^{2}-1)e_{6}.
	\end{equation*}
	As a consequence, $e_{3}+\lambda e_{6}$ is orthogonal to $\g{v}$, and we deduce that
	\[
		0={3 \sqrt{3}}\langle (\nabla_{Y}R)(Y,Z,Y) ,e_{3}+\lambda e_{6} \rangle=\lambda  \left(3-\lambda^{2}\right) \left(\lambda^2+1\right),
	\]
	a contradiction, so we obtain that $\lambda^{2}\in \left\{3,1/3\right\}$.
		
	Similarly, in the case that $\g{v}\cap \vecspan\{e_{2}+\lambda e_{5},e_{3}+\lambda e_{6}\}\neq 0$, we can use an element of $\mathsf{K}_{e_{1}}$ to assume that $Z=e_{2}+\lambda e_{5}$ also belongs to $\g{v}$.
	Here, if $\lambda^{2}\notin \{3,1/3\}$, then we can construct a basis of $\g{v}$ with the vectors
	\begin{align*}
		Y=e_{1}+\lambda e_{4}, \quad
		Z=e_{2}+\lambda e_{5}, \quad
		T=\frac{3\sqrt{3}}{\lambda}(\nabla_{Y}R)(Y,Z,Y)=\lambda(\lambda^{2}-3)e_{3}+(3-\lambda^{2})e_{6}.
	\end{align*}
	In particular, $e_{3}+\lambda e_{6}\in \g{p}\ominus \g{v}$, and we have
	\[
		0=\langle (\nabla_{T}R)(Y,Z,T),e_{3}+\lambda e_{6} \rangle=\frac{\left(\lambda ^2-3\right)^2 \left(\lambda ^2+1\right) \left(3 \lambda^2-1\right)}{6 \sqrt{3}},
	\]
	contradicting our assumption.
		
	All in all, we have proved that the totally geodesic subspace $\g{v}$ contains $Y=e_{1}+\lambda e_{4}$, and $\lambda^{2}=3$ or $\lambda^{2}=1/3$.
	Now, consider the isometry $s_o\colon \s{S}^{3}\times\s{S}^{3}\to\s{S}^{3}\times\s{S}^{3}$ defined in Subsection~\ref{subsection:S3S3Description}.
	One sees that the differential $(s_o)_{*o}$ satisfies the identities
	\begin{align*}
		(s_{o})_{*o} (e_{1})={}&-\frac{1}{2}(e_{1}+\sqrt{3}e_{4}) & (s_{o})_{*o}^{-1} (e_{1})={}&-\frac{1}{2}(e_{1}-\sqrt{3}e_{4}), \\
		(s_{o})_{*o} (e_{4})={}&\frac{1}{2}(\sqrt{3}e_{1}-e_{4}), & (s_{o})_{*o}^{-1}(e_{4})={}&-\frac{1}{2}(\sqrt{3}e_{1}+e_{4}).
	\end{align*}
	Using these equations, we deduce that either $(s_{o})_{*o} (\g{v})$ or $(s_{o})_{*o}^{2}(\g{v})$ is a totally geodesic subspace that contains either $e_{1}$ or $e_{4}$, so by Lemma~\ref{lemma:S3S3WellPositioned} it is well-positioned, and thus $\Sigma$ is congruent to $\mathsf{S}^{2}(2/\sqrt{3})$ or $\mathsf{S}^{3}_{\C,1/3}(2)$. \qedhere
\end{proof}
	
\begin{proposition}
	\label{prop:2tgs3s3}
	Let $\Sigma\to \mathsf{S}^{3}\times \mathsf{S}^{3}$ be a complete totally geodesic surface.
	Then, either $\Sigma$ is congruent to either $\s{T}^2_{\Gamma}$, the not well-positioned $\s{S}^{3}(\sqrt{3/2})$, or a great sphere inside the round $\s{S}^{3}\left(2/\sqrt{3}\right)$.
\end{proposition}
\begin{proof}
	Let $\Sigma$ be a complete totally geodesic surface, and assume without loss of generality that it passes through $o$ with tangent space $\g{v}$.
	We distinguish two situations for $\g{v}$ according to whether it is well-positioned or not.
		
	First, suppose that $\g{v}$ is well-positioned.
	If $\g{v}\subseteq \g{p}_{1}$, then $\Sigma$ is merely a round sphere inside the round $\mathsf{S}^{3}$.
	If $\g{v}\subseteq \g{p}_{2}$, we may suppose that $\g{v}=\vecspan\{e_{4},e_{5}\}$, but the equation $3\sqrt{3}(\nabla_{e_{4}}R)(e_{4},e_{5},e_{4})=e_{3}$ implies that $\g{v}$ is not $\nabla R$-invariant, a contradiction.
	Now, suppose that $\g{v}\cap \g{p}_{1}$ and $\g{v}\cap \g{p}_{2}$ are both one-dimensional.
	By using the isotropy representation, we may suppose that $\g{v}\cap\g{p}_{1}=\R e_{1}$.
	The Jacobi operator $R_{e_{1}}$ preserves $\g{p}_{2}$ and its restriction to $\g{p}_{2}$ has eigenvalues $0$ (with eigenspace $\R e_{4}$) and $\frac{1}{12}$ (with eigenspace $\vecspan\{e_{5},e_{6}\}$).
	Thus, we can further assume (up to the action of $\mathsf{K}$) that either $\g{v}=\vecspan\{e_{1},e_{4}\}$ or $\g{v}=\vecspan\{e_{1},e_{5}\}$.
	The first case simply yields $\Sigma=\mathsf{T}^2_{\Gamma}$, while the second case gives a contradiction because $3\sqrt{3}(\nabla_{e_{5}}R)(e_{5},e_{1},e_{5})=e_{6}$.
	This completes the case that $\g{v}$ is well-positioned.
		
	Suppose $\g{v}$ is not well-positioned.
	In particular $\g{v}_{\g{p}_{1}}\neq 0$, and this combined with the isotropy representation lets us assume that $\g{v}$ contains a vector of the form $X=e_{1}+\rho \cos \theta e_{4} + \rho \sin \theta e_{5}$, where $\rho > 0$ and $\theta\in [0,2\pi)$.
		
	Suppose $\g{v}$ is invariant under $D$, which yields $\g{v}\subseteq \ker D_{X}$ as it is two-dimensional.
	If $\theta \in \{0,\pi\}$, then $\ker D_{X}=\vecspan \{e_{1},e_{4}\}$, which forces $\g{v}=\vecspan\{e_{1},e_{4}\}$, contradicting our hypothesis that $\g{v}$ is not well-positioned.
	If $\theta\notin\{0,\pi\}$, the kernel of $D_{X}$ is spanned by $X$ and $Y=\rho \cos \theta e_{1}+\rho \sin \theta e_{2}-e_{4}$, so $\g{v}=\vecspan\{X,Y\}$.
	As a consequence, we see that the vector $Z=~-\rho \sin \theta e_{1}+\rho \cos \theta e_{2}+e_{5}$ is orthogonal to $\g{v}$, and thus
	$0=\langle R(X,Y)X,Z \rangle=\frac{4}{3}\rho^{2}\sin 2\theta$,
	which forces $\theta = \pi/2$ or $\theta=3\pi/2$. 
	We group these cases together by writing $X=e_1+t e_5$ and $Y=te_2-e_4$	for a nonzero $t\in \R$.
	As a consequence, $e_{2}+t e_{4}$ is orthogonal to $\g{v}$, and	$0=\langle R(X,Y)X,e_{2}+t e_{4} \rangle=\frac{4}{3}t(t^{2}-1)$,
	so $t=\pm 1$.
	The cases $t=1$ and $t=-1$ give rise to congruent submanifolds.
	Indeed, the element
		$k=(\operatorname{diag}(i,-i),\operatorname{diag}(i,-i),\operatorname{diag}(i,-i))\in\mathsf{K}$
	satisfies $\Ad(k)(e_{1}+e_{5})=e_{1}-e_{5}$ and $\Ad(k)(e_{2}-e_{4})=-e_{2}-e_{4}$.
	As a consequence, we see that $\g{v}$ is conjugate to $\operatorname{span}\{e_{1}+e_{5},e_{2}-e_{4}\}$, so $\Sigma$ is congruent to the not well-positioned $\s{S}^{2}(\sqrt{3/2})$.
		
	Now, assume that $\g{v}$ is not $D$-invariant.
	Arguing as before, we may suppose that $\g{v}$ contains a vector of the form $X=e_{1}+\rho \cos \theta e_{4} + \rho \sin \theta e_{5}$ for $\rho > 0$ and $0 \leq \theta < 2 \pi$.
	On the one hand, one sees that $C_{X}=0$ if $\rho=\sqrt{3}$ and $\theta\in\{0,\pi\}$, but in this case $\g{v}$ is congruent to a well-positioned example.
	Indeed, we argue as in the end of Proposition~\ref{prop:3tgs3s3}.
	Consider the isometry $s_o\colon \mathsf{S}^{3}\times\mathsf{S}^{3}\to\mathsf{S}^{3}\times\mathsf{S}^{3}$ defined as in Subsection~\ref{subsection:S3S3Description}.
	Then, if $\theta=0$ we have $X=-2 (s_{o})_{*o}(e_{1})$, so $\g{v}$ is congruent to $(s_o)_{*o}^{-1}(\g{v})$, and this subspace is well-positioned by Lemma~\ref{lemma:S3S3WellPositioned}, so we may apply the conclusions from the previous case.
	Similarly, if $\theta=\pi$, then $X=-2(s_{o})^{-1}_{*o}(e_{1})$, and $(s_o)_{*o}(\g{v})$ is well-positioned.
	On the other hand, if $\rho \neq \sqrt{3}$ or $\theta\neq \pi$, we have
	\begin{align*}
		C_{X}e_{1}={}&\frac{\rho^{2}\sin 2\theta}{6\sqrt{3}}e_{3}+\frac{\rho^{3}\sin\theta}{3\sqrt{3}}e_{6}, & 
		C_{X}e_{2}={}&\frac{\rho ^2 \sin ^2\theta }{3 \sqrt{3}} e_{3}-\frac{\rho  \left(\rho ^2-3\right) \cos \theta }{3 \sqrt{3}}e_{6}, \\
		C_{X}e_{4}={}&-\frac{\rho  \left(\rho ^2-2\right) \sin\theta }{3 \sqrt{3}}e_{3}-\frac{\rho ^2 \sin 2 \theta }{6 \sqrt{3}}e_{6}, &
		C_{X}e_{5}={}&\frac{\rho  \left(\rho ^2-3\right) \cos\theta}{3 \sqrt{3}}e_{3}-\frac{\rho ^2 \sin ^2 \theta }{3 \sqrt{3}}e_{6}.
	\end{align*}
	A straightforward computation gives that these four vectors span the subspace $\vecspan\{e_{3},e_{6}\}$, so we have $\vecspan\{e_{3},e_{6}\}\subseteq \operatorname{im}C_{X}$.
	In particular, since $C_{X}$ is a symmetric endomorphism, we have $\g{p}=\ker C_{X}\oplus \operatorname{im}C_{X}$ orthogonally, so $\ker C_{X}$ is orthogonal to $e_{3}$ and $e_{6}$.
	As a consequence, $\g{v}$ is spanned by $X$ and a vector of the form $Y=a_{1}e_{1}+a_{2}e_{2}+a_{4}e_{4}+a_{5}e_{5}$,
	where $a_{1}$, $a_{2}$, $a_{4}$, $a_{5}\in\R$.
	However,
	\[
		D_X Y=\frac{a_4 \rho  \sin \theta -a_5 \rho  \cos
			\theta +a_2}{2 \sqrt{3}} e_{3}+\frac{a_1 \rho  \sin \theta -a_2 \rho  \cos
			\theta -a_5}{2 \sqrt{3}}e_{6}
	\]
	is in $\ker C_{X}$ by Proposition~\ref{prop:TojoSurfaces} and in $\vecspan\{e_{3},e_{6}\}\subseteq \g{p}\ominus\ker C_{X}$, which forces $D_X Y=0$, and thus $\g{v}$ is $D$-invariant, contradicting our assumption.
	This finishes the proof. \qedhere
\end{proof}

\begin{proof}[Proof of Theorem~\ref{th:tg-s3s3-classification}]
	The result follows from combining Theorem~\ref{th:hyp}, Proposition~\ref{prop:4tgs3s3}, Proposition~\ref{prop:3tgs3s3}, and Proposition~\ref{prop:2tgs3s3}. 
\end{proof}

Finally, we can also prove Theorem~\ref{th:tg-g2-cones}.
\begin{proof}[Proof of Theorem~\ref{th:tg-g2-cones}]
	Let $M$ be a homogeneous nearly Kähler $6$-manifold with non-constant curvature.
	Since totally geodesic submanifolds are preserved under Riemannian coverings and the universal cover of $\hat{M}$ is the cone of the universal cover of $M$, we may assume that $M$ is simply connected, so $M$ is either $\C\s{P}^3$, $\s{F}(\C^3)$ or $\s{S}^3\times\s{S}^3$.
	
	Let $\Sigma$ be a complete totally geodesic submanifold of the $\s{G}_2$-cone $\hat{M}$.
	By Corollary~\ref{cor:MaximalTGSubmanifoldsOfCones}, $\Sigma$ is either a hypersurface of $\hat{M}$ or the cone of a maximal totally geodesic submanifold of $M$.
	The first case is not possible due to~\cite[Theorem~1.2]{JentschMoroianuSemmelmann}, so we conclude that $\Sigma=\hat{S}$ for a maximal totally geodesic submanifold $S\subseteq M$, which yields the desired result.
\end{proof}
\begin{remark}
	Let $M\in\{\C\s{P}^3,\s{F}(\C^3),\s{S}^3\times\s{S}^3 \}$.
	A look at Tables~\ref{table:CP3},~\ref{table:FC3} and~\ref{table:S3S3} shows that two (complete) totally geodesic submanifolds of $M$ are congruent if and only if they are isometric.
	Combining this with Proposition~\ref{prop:IsometricCones}, we deduce that the congruence classes of maximal totally geodesic submanifolds of $\hat{M}$ are in a bijective correspondence with the congruence classes of maximal totally geodesic submanifolds of $M$.

	We also note that in order to obtain the classification of all totally geodesic submanifolds of $\hat{M}$, it suffices to iterate Corollary~\ref{cor:MaximalTGSubmanifoldsOfCones} and take into account Propositions~\ref{prop:ConstantCurvatureCone} and~\ref{prop:bergercone}.
\end{remark}

\end{document}